\documentclass[12pt,reqno]{amsart}
\usepackage[margin=1in]{geometry}
\makeatletter
\def\section{\@startsection{section}{1}%
	\z@{.7\linespacing\@plus\linespacing}{.5\linespacing}%
	{\bfseries\normalfont\scshape
		\centering
}}
\def\@secnumfont{\bfseries}
\makeatother
\usepackage{amsmath,amsfonts,amsthm,txfonts}
\usepackage{dsfont}

\usepackage{accents}
\usepackage{trimclip}

\newcommand{\newhat}{\scalebox{1.5}[.75]{\trimbox{0pt 1.1ex}{\textasciicircum}}}

\newcommand{\stretchedhat}[1]{\accentset{\newhat}{#1}}
\def\shat{\stretchedhat}

\usepackage{graphicx,amssymb}
\usepackage{amsmath,amssymb,amsfonts}
\usepackage{hyperref}
\usepackage{amscd}
\usepackage{amsthm}
\usepackage{xcolor}
\numberwithin{equation}{section}
\newtheorem{Thm}{Theorem}[section]
\newtheorem{Lem}{Lemma}[section]
\newtheorem{Pro}{Proposition}[section]
\newtheorem{Rem}{Remark}[section]

\newtheorem{Cor}{Corollary}[section]
\newtheorem{Def}{Definition}[section]

\usepackage[shortlabels]{enumitem}

\usepackage{cancel}

\usepackage{cases} 

\allowdisplaybreaks

\begin{document}
\title[Optimal control of the 2D and 3D Landau-Lifshitz-Bloch Equation]
{Magnetization control problem for the 2D and 3D evolutionary \medskip\\ Landau-Lifshitz-Bloch equation}
\author[Sidhartha Patnaik \ and \ Kumarasamy Sakthivel ]{Sidhartha Patnaik \ and \ Kumarasamy Sakthivel}
\address{Department of Mathematics \\
	Indian Institute of Space Science and Technology (IIST) \\
	Trivandrum- 695 547, INDIA}
\email{sidharthpatnaik96@gmail.com, sakthivel@iist.ac.in}
\curraddr{}


\begin{abstract}
	In this study, we investigate the optimal control of the Landau-Lifshitz-Bloch  equation within confined domains in $\mathbb R^n$ for $n= 2, 3.$ We establish the existence of strong solutions for dimensions $n=1, 2, 3$ under suitable growth conditions on the control, and analyze the existence and uniqueness of regular solutions.  We formulate the control problem in which only  a fixed set of finite magnetic field coils can constitute the external magnetic field (control). We define a cost functional by aiming at minimizing the energy discrepancy between the evolving magnetic moment and the desired state. We demonstrate the existence of an optimal solution pair and employ the classical adjoint problem approach to derive a first-order necessary optimality condition. Given the non-convex nature of the optimal control problem, we derive a second-order sufficient optimality condition using a cone of critical directions. Finally, we prove two crucial results, namely, a global optimality condition and uniqueness of an optimal control.
\end{abstract}
\maketitle

\section{Introduction}
The Landau–Lifshitz–Bloch (LLB) equation stands as a pivotal model in dynamic micromagnetics, providing a comprehensive framework for understanding magnetic behaviors in various materials, especially at elevated temperatures. Unlike its predecessor, the Landau–Lifshitz–Gilbert (LLG) equation, LLB incorporates both transverse and longitudinal relaxation components, making it particularly relevant for high-temperature applications. The LLB equation introduces a dynamic nature to the magnetization magnitude, no longer preserving its constancy as observed in the LLG equation, making it more
crucial for scenarios near or above the Curie temperature(T$_C$). In these contexts, where traditional models fall short, the LLB equation offers a more accurate depiction of magnetization dynamics. Above the Curie temperature, a ferromagnetic material transitions to the paramagnetic state, where there is no spontaneous magnetization. In this paramagnetic state, the LLB equation still applies, but the terms related to ferromagnetic interactions, such as exchange and anisotropy, become negligible. In the context of the LLB equation, the damping parameters may still be present to model the relaxation processes, but their specific values and interpretation could be different in the paramagnetic state compared to the ferromagnetic state.

Consider the magnetization $\bar m$ with $\bar m_s^0$ as its saturation magnetization value at $t=0$.  Then for $t>0$ and $\Omega\subset \mathbb{R}^n$, $n=1,2,3$,  the average spin polarization $m=\bar m/m_s^0$ satisfies the following LLB equation (\cite{DAG,DAG2})
\begin{equation}\label{MGE}
	m_t= \gamma m\times H_{eff} + D_L\  \frac{1}{|m|^2} (m\cdot H_{eff})\ m -D_T\ \frac{1}{|m|^2}m \times (m \times H_{eff}),
\end{equation}
where $|\cdot |$ stands for the Euclidean norm in $\mathbb{R}^3$, $\gamma>0$ is the gyromagnetic ratio, $D_L$ and $D_T$ are the longitudinal and transverse damping parameters, respectively.  The effective field ($H_{eff}$) represents the total magnetic field acting on a ferromagnetic material. It encompasses contributions from external fields, magnetic anisotropy, and thermal fluctuations.


Above the Curie temperature (T$_C$), ferromagnetic materials lose their ordered magnetic structure and transition into a paramagnetic state. When the temperature is high, changes in temperature throw off the ordered magnetic structure. This makes the longitudinal ($D_L$) and transverse ($D_T$) damping parameters equal. This balanced damping, observed in the LLB equation, signifies a significant shift in the material's magnetic behavior, impacting its response to external magnetic fields.  If we consider the longitudinal susceptibility as $\chi_\|$ and the external magnetic field as $u$, then the effective field $H_{eff}$ is given by (\cite{QLBG})
$$H_{eff}= \Delta m - \frac{1}{\chi_\|}\left(1+\frac{3}{5}\frac{T}{T-T_C}|m|^2\right)m+u.$$
Substituting the value of $H_{eff}$ in \eqref{MGE},  using the equality $D_L=D_T=\kappa$,   and employing the vector-product property $m \times (m \times H_{eff})=(m \cdot H_{eff}) m - |m|^2H_{eff},$  one can derive
$$m_t =\nu \ m \times \Delta m + \nu\ m \times u + \kappa \left(\Delta m- \frac{1}{\chi_\|}\left(1+\frac{3}{5}\frac{T}{T-T_C}|m|^2\right)m+u\right).$$
For simplicity, we take the value of the constants $\nu=\kappa=\chi_\|=\frac{3}{5} \frac{T}{T-T_C}=1$.   Therefore, for any smooth bounded domain $\Omega\subset \mathbb{R}^n$, the  controlled LLB equation with initial data $m_0$, control $u$  and homogeneous Neumann boundary condition is given by 
\begin{equation}\label{NLP}
	\begin{cases}
		m_t- \Delta m= m \times \Delta m  + m \times u- \big(1+ |m|^2\big)m + u, \ \ \ (x,t)\in \Omega_T,\\
		
		\frac{\partial m}{\partial \eta}=0, \ \ \ \ \ \ \ (x,t) \in  \partial\Omega_T,\\
		
		m(\cdot,0)=m_0 \ \ \text{in} \ \Omega,
	\end{cases}	
\end{equation}
where $\eta$ is the outward unit  normal vector to the boundary $\partial \Omega$. Moreover, we have assumed that the initial data $m_0:\Omega\to\mathbb R^3$ satisfies the following conditions 
\begin{equation}\label{IC}
		m_0 \in H^2(\Omega),\ \ \ \frac{\partial m_0}{\partial \eta}=0 \ \ \text{on} \ \partial \Omega. 
\end{equation}

In this paper, we investigate the optimal control problem with an external magnetic field generated by a finite number of fixed magnetic field coils as motivated in \cite{PK}.  Let there be $`N$' such coils, and denote the magnetic field of the $k$-th coil as follows:
$$u_k:\Omega_T\to \mathbb{R}^3, \ \ \ \ \ (x,t)\mapsto u_k(x,t):= U_k(t)B_k(x),$$
where $U_k:[0,T]\mapsto \mathbb{R}$ represents the intensity of the magnetic field, proportional to the electric current flowing through the $k$-th coil. The function $B_k:\Omega \mapsto \mathbb{R}^3$ characterizes the geometry of the $k$-th coil. Finally, considering linear superposition, the total external magnetic field is expressed as given below:
\begin{equation}\label{gotc}
	u(x,t)=\sum_{k=1}^N U_k(t)B_k(x) \ \ \ \ \ \forall \ (x,t)\in \Omega_T.
\end{equation}

Here, we assume that the geometry of each field coil $B_k\in H^1(\Omega,\mathbb{R}^3)$ for $k=1,2,...,N$. Defining $B_k$ within this space offers the advantage that, for any control $U=(U_1,U_2,...,U_N)\in L^2(0,T;\mathbb{R}^N)$, the space-time dependent function $u(x,t)$ belongs to $L^2(0,T;H^1(\Omega))$. Indeed, let $\zetaup$ be an operator from $L^2(0,T;\mathbb{R}^N)$ to $L^2(0,T;H^1(\Omega))$ defined by  $\zetaup(U)=\sum_{k=1}^N U_k(t)B_k(x)$.  Then by using simple norm estimates, we can find a bound for $\zetaup$ as follows
\begin{equation}\label{PEST} 	
	\|\zetaup(U)\|_{L^2(0,T;H^1(\Omega))} \leq \sum_{k=1}^N \|B_k\|_{H^1(\Omega)}\ \|U_k\|_{L^2(0,T)}\leq \max_{1\leq k\leq N}\|B_k\|_{H^1(\Omega)} \ \|U\|_{L^2(0,T;\mathbb{R}^N)}.	
\end{equation} 	
This control setup holds significant physical relevance and offers practical numerical implementation opportunities.  Further,  the analysis of optimality conditions for the control problem demands more regularity on a space-time dependent control for the regularity of solutions, which restricts us from deriving important second-order conditions (Theorem \ref{T-SOLO}) in a critical cone, whereas the control beloging to $L^2(0,T;\mathbb{R}^N)$ relaxes this requirement and allows us to work with admissible box-type constraints in $L^2(0,T;\mathbb{R}^N)$. 

Let $m_d: \Omega \times [0,T] \to \mathbb{R}^3$ be the desired evolutionary magnetic moment and $m_\Omega : \Omega \to \mathbb{R}^3$ be the final time target moment. We consider  the optimal control problem of minimizing the objective/cost functional $\mathcal{J}:\mathcal{M} \times \mathbb{U}_{ad} \to \mathbb{R}^+$ defined  by
\begin{equation}\label{CF-2}
	\mathcal J(m,U):= \frac{1}{2} \int_0^T \int_{\Omega} |  m- m_d|^2\ dx\  dt +  \frac{1}{2} \int_{\Omega} |m(x,T) -m_\Omega(x)|^2\ dx +\frac{1}{2} \sum_{i=1}^N\|U_i\|^2_{L^2(0,T)},					
\end{equation}
where $\mathcal{M}$ is the set of regular solutions and $\mathbb{U}_{ad}$ is  the admissible set of fixed magnetic field coils with box-type constraints defined later.  Besides, the pair  $(m,U)$ solves the following  system controlled by a finite number of fixed magnetic field coils:
\begin{equation}\label{NLP-EV}
	\begin{cases}
		m_t- \Delta m= m \times \Delta m  + m \times \zetaup(U)- \big(1+ |m|^2\big)m + \zetaup(U), \ \ \ (x,t)\in \Omega_T,\\
		
		\frac{\partial m}{\partial \eta}=0, \ \ \ \ \ \ \ (x,t) \in  \partial\Omega_T,\\
		
		m(\cdot,0)=m_0 \ \ \text{in} \ \Omega.
	\end{cases}	
\end{equation}

This research explores the dynamics of LLB by shedding light on its behavior when temperatures surpass the Curie threshold. By delving into the intricacies of LLB-based micromagnetics, this study contributes to a deeper understanding of magnetic phenomena in high-temperature environments since it has found extensive applications in simulating magnetization dynamics near Curie temperatures (\cite{UADH}). The LLB equation is used to model how the magnetization changes with temperature in heat-assisted magnetic recording (HAMR), especially close to the Curie temperature \cite{MKEG}. This application enables the calculation of reversal times and mechanisms under varying thermal conditions, thereby offering critical insights into optimizing the thermal stability and switching behaviors of magnetic nanograins, which are essential for the development of next-generation HAMR technologies. The LLB equation is utilized to model the ultrafast magnetization switching dynamics induced by femtosecond laser pulses, accounting for the inverse Faraday effect and thermal effects through the two-temperature model \cite{KVAM}. The LLB equation is also used to model how domain walls move in response to changes in temperature, taking into account how exchange stiffness and entropy change with temperature, as shown in \cite{DHUN}.

{\bf A brief literature survey of LLG and LLB equations.} Over the last few decades, researchers have conducted a comprehensive examination of the Landau-Lifshitz-Gilbert (LLG) equation. In \cite{FAAS}, the existence and non-uniqueness of a global weak solution of the LLG equation in a 3D bounded domain have been explored. Global regular solutions for both a bounded domain in $\mathbb{R}^n,n=2,3$ and the entire domain $\mathbb{R}^3$ have been established in \cite{GCPF} and \cite{GCPF2}, respectively.  The papers \cite{TDMKAP} and \cite{SPSK,SPSK2} respectively analyzed the optimal control problem of the 1D and 2D LLG equations. These studies collectively contribute to an enhanced understanding of the dynamics of the LLG equation and its diverse applications. For more results on the solvability and controllability of the LLG equation, one may refer to \cite{SAGCSLCP, FAKB, WFB, GCSL2, AP}.
However, there are only a limited number of articles available concerning the LLB equation without the \emph{nonlinear control} in the effective field. In \cite{KNL}, the focus was on investigating the existence of weak solutions in a bounded domain. The work presented in \cite{KHDH} delves into the entirety of the temperature spectrum, demonstrating the global existence of strong solutions, time-periodic solutions, and steady-state solutions for the LLB equation. The 1D compressible LLB equation in $\mathbb R$ was the subject of study in \cite{BGFL}, in which the existence and uniqueness of a global smooth solution were established. The authors in \cite{QLBG} showed that the LLB equation has a smooth solution in the domain $\mathbb{R}^2$. They also got the same result in the domain $\mathbb{R}^3$ with small initial data.

{\bf Main contributions of the paper.} To the best of our knowledge, this is the first paper to study the optimal control of the LLB equation in a bounded domain in $\mathbb R^n,n=1,2,3.$ The  control input is introduced through an external magnetic field in the effective field. We demonstrate the existence of strong solutions for a bounded domain when n = 1, 2, and 3. In contrast to the 2D LLG equation (\cite{GCPF, SPSK}),  where the regular solution exists when the initial field $m_0$ is sufficiently small, the 2D LLB equation admits a regular solution for any initial data and control. But in the case of the 3D LLB equation, due to the severe nonlinearity in the magnetic field $m$ as well as the control $\zetaup(U),$ we only prove that the regular solution exists for a small enough $m_0$ and $U$ (see Theorem  \ref{T-RS}).
We examine an optimal control problem in which a finite number of fixed magnetic field coils generate the control. This construction is crucial from an application point of view, as it provides a more practical and adequate method for implementing the control. We verify the existence of an optimal control and present a first-order optimality condition in the form of a variational inequality. The second pivotal result is the derivation of a second-order sufficient optimality condition on a cone of critical directions, achieved through a rigorous technical argument of \cite{ECFT}. To achieve this, we demonstrated the Frechet differentiability of both the control-to-state operator and the control-to-costate operator, as well as their Lipschitz continuity. The last and most important contribution is a sufficient condition on a final time $T$ that indicates the uniqueness of an optimal control, together with a global optimality result that is necessary for numerical analysis. 

The manuscript is organized as follows. The section \ref{MRFSI} is devoted to state all the main results of this paper. We demonstrate the existence of strong and regular solutions in section \ref{S-WSRS}. The existence of optimal control, analysis of control-to-state operator and  first-order optimality conditions are given in section \ref{S-FOOC}. The section \ref{S-SOOC} concentrates on the crucial result of second-order local optimality conditions, while a global optimality is analyzed in section \ref{S-GOC}. Finally, in section \ref{S-UOO}, we prove the uniqueness of optimal solution.

\section{Main Results, Function Spaces and Inequalities}\label{MRFSI}

\subsection{Main Results}\label{MR}
In this subsection, we list the function spaces used consistently in our paper and provide a brief overview of the main findings and key results of our study.
To address the optimal control problem associated with the higher-dimensional LLB equation, the upcoming theory necessitates more than just strong solutions. Consequently, we define the following set of admissible solution spaces:
	\begin{flalign*}
		\mathcal{M}_S&:= W^{1,2}(0,T;H^2(\Omega),L^2(\Omega))\ =\{m \in L^2(0,T;H^2(\Omega)) \ |\ m_t \in L^{\frac{4}{3}}(0,T;L^2(\Omega))\},\\
		\text{and}\ \ \mathcal{M}&:=W^{1,2}(0,T;H^3(\Omega),H^1(\Omega))=\{m \in L^2(0,T;H^3(\Omega)) \ |\ m_t \in L^2(0,T;H^1(\Omega))\}.
	\end{flalign*}
	A function $m_u \in \mathcal{M}_S$ is said to be a ``strong solution" of system \eqref{NLP} if it satisfies the system almost everywhere, with the control function $u\in L^2(0,T;L^s(\Omega))$. On the other hand, a  function $m_u \in \mathcal{M}$ is referred to as a ``regular solution" of system \eqref{NLP} if it satisfies the system almost everywhere while being associated with the control function $u\in L^2(0,T;H^1(\Omega))$.
	
	Define the solution space for the weak adjoint problem as $$\mathcal{Z} := W^{1,2}(0,T;H^1(\Omega),H^1(\Omega)^*):=\left\{z \in L^2(0,T;H^1(\Omega)) \mid z_t \in L^2(0,T;H^1(\Omega)^*)\right\}.$$

	\begin{Thm}(Existence of Strong Solution for n=1,2 and 3)\label{T-SS}	
		For any $\delta >0$, let us fix the value 
		$$s= \begin{cases}
			2  \ \  & \text{for}\ \ n=1, \\
			2+\delta  \ \  & \text{for}\ \ n=2,\\
			3 \ \  & \text{for}\  \ n=3.
		\end{cases}
		$$
		Suppose the control $u\in L^2(0,T;L^s(\Omega))$ and the initial data $m_0\in H^1(\Omega)$. 
		Then there exists a strong solution $m\in \mathcal{M}_S$ of system \eqref{NLP}. Moreover, there exists a constant $C(\Omega,T)>0$ such that the following estimate holds:
		\begin{align}\label{SSEE}
			&\|m\|^2_{L^\infty(0,T;H^1(\Omega))} + \|m\|^2_{L^2(0,T;H^2(\Omega))}  + \|m_t\|^2_{L^{\frac{4}{3}}(0,T;L^2(\Omega))}    \nonumber\\
			&\  \ \ \     \leq C \ \left(1+\|m_0\|^2_{H^1(\Omega)} + \|u\|^2_{L^2(0,T;L^2(\Omega))} \right)^3 \  \exp\left\{C\ \|u\|^2_{L^2(0,T;L^s(\Omega))}\right\}.
		\end{align}	
		Furthermore, for $n=2$ the strong solution will be unique.
	\end{Thm}
	
	\noindent The proof of this theorem is presented in subsection \ref{SS-EOSS}. Next we will show the existence of a regular solution for $n=2$ and $3$.

\begin{Rem} Hereafter,  we mainly concentrate only on dimensions $n=2$ and $3,$ though all the theories presented here also applies to the case $n=1.$
Further,  in the  1D case more precise results, such as first-order and second-order optimality conditions, can be obtained using strong solutions instead of regular solutions. This approach simplifies most of the results obtained for $n=2$ and $3$.
\end{Rem}	
	
	\begin{Thm}(Existence of Regular Solution for n=2,3)\label{T-RS} 
		Suppose the control $u\in L^2(0,T;H^1(\Omega))$ and the initial data $m_0$ satisfies the condition \eqref{IC}. Then for $n=2$, system \eqref{NLP} admits a unique global regular solution $m\in \mathcal{M}$. Moreover, there exist constants $M(m_0,u,\Omega,T,n)$ and $C(\Omega,T)>0$ such that the following estimate holds:
		\begin{equation}\label{SSEE2}
				\|m\|^2_{L^\infty(0,T;H^2(\Omega))} + \|m\|^2_{L^2(0,T;H^3(\Omega))} \leq 	M(m_0,u,\Omega,T),
		\end{equation}
		\begin{flalign*}
			\text{where}\ \ M(m_0,u,\Omega,T) &:= \left(\|m_0\|^2_{L^2(\Omega)}+\|\Delta m_0\|^2_{L^2(\Omega)}+\|u\|^2_{L^2(0,T;H^1(\Omega))}\right)&\nonumber\\
			&  \times  \exp\left\{C \left[1+\|m_0\|^2_{H^1(\Omega)} + \|u\|^2_{L^2(0,T;L^2(\Omega))} \right]^2 \ \exp\left\{C\left(\|u\|^2_{L^2(0,T;H^1(\Omega))}\right)\right\}\right\}.
		\end{flalign*}
		For $n=3$, system \eqref{NLP} admits a unique ``local in time" regular solution for every control in $L^2(0,T;H^1(\Omega))$. However, there exists a constant $\widetilde{C}(\Omega)>0$ such that under the following smallness assumption:\vspace{-0.2cm}
		\begin{equation}\label{SAOIC}
			 M(m_0,u,\Omega,T) < \frac{1}{\widetilde{C}^{\frac{1}{2}}},
		\end{equation}
		the solution exists globally and satisfies the energy estimate \eqref{SSEE2}.
	\end{Thm}
	\noindent 	The proof of Theorem \ref{T-RS} is given in subsection \ref{SS-EORS}.



	Furthermore, we will define our admissible control set over which we will analyze the optimal control problem. For any $a_i,b_i \in L^2(0,T)$, $i=1,2,...,N$, consider the set
$$\mathbb{U}_{a,b}:=\Big\{U=(U_1,U_2,...,U_N)\in L^2(0,T;\mathbb{R}^N)\ \big|\ a_i(t)\leq U_i(t)\leq b_i(t)\ \ \ \ \text{for } i=1,2,...,N\Big\}.$$
	Since, the existence of regular solution (cf. Theorem \ref{T-RS}) for $n=3$ is only known for small initial data and control, 
	let us define the control set 
	$$\mathbb{U}_{sa}:= \Big\{ U=(U_1,U_2,...,U_N)\in L^2(0,T;\mathbb{R}^N) \ \big| \ M(m_0,\zetaup(U),\Omega,T)\leq  \frac{1}{2\widetilde{C}^{\frac{1}{2}}} \ \text{with sufficiently small}\ m_0\Big\}.$$
	Thus, we work with the following admissible class of controls
	$
	\mathbb{U}_{ad}=\begin{cases}
		\mathbb{U}_{a,b} \ \ \ &\text{for} \ n=2,\\
		\mathbb{U}_{a,b}\cap\mathbb{U}_{sa} \ &\text{for} \ n=3.
	\end{cases} 	
	$
	
	We will demonstrate the Fr\'echet differentiability of the control-to-state and the control-to-costate operators over an open set $\mathbb{U}_R$ given by 
	$$\mathbb{U}_R:=\left\{U=(U_1,U_2,...,U_N)\in L^2(0,T;\mathbb{R}^N)\ \big| \  \|U\|_{L^2(0,T;\mathbb{R}^N)}=\sum_{i=1}^N \left(\int_0^T |U_i(t)|^2\ dt\right)^{\frac{1}{2}}  <R \right\}.$$
	For $n=2$, the constant $R$ can be taken as $R=\|a\|_{L^2(0,T;\mathbb{R}^N)}+\|b\|_{L^2(0,T;\mathbb{R}^N)}+1$. However, the global solvability of regular solutions of \eqref{NLP} for $n=3$ is only guaranteed under the smallness assumption. Therefore, for $n=3$, with appropriate small initial data $m_0$, the value of the constant $R$ can be choosen such that  $M(m_0,\zetaup(U),\Omega,T)<\frac{1}{\widetilde{C}^{\frac{1}{2}}}$.    
		
	
	
	The solvability of regular solutions for system \eqref{NLP} prompts the definition of an operator that associates each control in the set $\mathbb{U}_R$ with the admissible solution space $\mathcal{M}$, namely, the control-to-state operator, denoted as $G:\mathbb{U}_R\to \mathcal{M}$.

	An admissible pair $(m,U)$ is defined as a solution to system \eqref{NLP-EV} where $m$ belongs to $\mathcal{M}$ and $U$ belongs to $\mathbb{U}_{ad}$. The set of all admissible pairs is denoted as $\mathcal{A} \subseteq  \mathcal{M} \times \mathbb{U}_{ad}$.  In order to do  the analysis for $n=3,$   we  assume that there exists at least one control $U\in \mathbb{U}_{a,b}\cap \mathbb{U}_{sa}$ such that there exists a regular solution corresponding to $U,$  that is,  $\mathcal A\neq \emptyset.$
	The optimal control problem is stated as follows:
	\begin{equation*}\label{OCP}
		\text{(OCP)}\begin{cases}
			\text{minimize}\ \mathcal J(m,U),\\
			(m,U) \in \mathcal{A}.
		\end{cases}	
	\end{equation*}

	\begin{Thm}\label{T-EOOC}
		Suppose the initial data $m_0$ satisfies \eqref{IC}, and in the case of $n=3$, \eqref{SAOIC} holds true. Then there exists a solution $(\widetilde{m}, \widetilde{U})\in \mathcal{M} \times \mathbb{U}_{ad}$ for the optimal control problem (OCP).
		
	\end{Thm}
	For the proof of Theorem \ref{T-EOOC}, consult subsection \ref{S-FOOC}. Next, we proceed to establish a first-order necessary optimality condition for the problem OCP in the form of a variational inequality by employing the classical adjoint problem approach. Utilizing the formal Lagrange method \cite{FT}, we derive the adjoint system:
	\begin{equation}\label{AS}
		\begin{cases}
			\phi_t + \Delta \phi  + \Delta (\phi \times \widetilde{m})+(\Delta  \widetilde{m}\times \phi)-\big(\phi \times \zetaup(\widetilde{U})\big) -\left(1+|\widetilde{m}|^2\right)\phi -2\ \big(\widetilde{m}\cdot \phi\big) \ \widetilde{m} =-\big(\widetilde{m}-m_d\big) \ \ \ \text{in}\ \Omega_T,\vspace{0.2cm}\\
			\frac{\partial \phi}{\partial \eta}=0 \ \ \ \text{in}\ \partial \Omega_T,\vspace{0.1cm}\\
			\phi(T)=\widetilde{m}(x,T)-m_{\Omega}(x)\ \ \ \text{in} \ \Omega.
		\end{cases}
	\end{equation}
	
	
	\noindent	Consequently, the subsequent step involves defining the weak formulation of the adjoint problem.
	
	
	\begin{Def}[Weak formulation]\label{AWSD}
		A function $\phi \in \mathcal{Z}$ is said to be a weak solution of system \eqref{AS} if for every $\vartheta \in L^2(0,T;H^1(\Omega))$, the following holds:
		\begin{enumerate}[(i)]
			\item $\begin{aligned}[t]
				&\int_0^T \langle \phi'(t),\vartheta (t)\rangle_{H^1(\Omega)^*\times H^1(\Omega)}dt - \int_{\Omega_T} \nabla \phi \cdot \nabla  \vartheta\ dx \ dt- \int_{\Omega_T}\nabla (\phi \times \widetilde{m})\cdot \nabla  \vartheta\ dx \ dt\nonumber\\
				&\hspace{1cm}+ \int_{\Omega_T}(\Delta \widetilde{m}\times \phi)\cdot \vartheta\ dx\ dt-\int_{\Omega_T} \big(\phi \times \zetaup(\widetilde{U})\big)\cdot \vartheta\ dx\ dt- \int_{\Omega_T} \left( 1 + |\widetilde{m}|^2 \right) \phi\cdot \vartheta\  dx\ dt\nonumber\\
				&\hspace{1cm}	 -2\int_{\Omega_T} \big((\widetilde{m} \cdot \phi)\ \widetilde{m}\big)\cdot \vartheta \ dx\ dt =-\int_{\Omega_T} \big(\widetilde{m}-m_d\big)\cdot \vartheta\ dx\ dt,  
			\end{aligned}$\vspace{0.1cm}
			\item $\phi(T)=\widetilde{m}(T)-m_{\Omega}.$
		\end{enumerate}
	\end{Def}

	\begin{Thm}\label{T-AWS}
		Let $(\widetilde{m},\widetilde{U})\in \mathcal{A}$ be an admissible pair for OCP. Then there exists a unique weak solution $\phi \in \mathcal{Z}$ for the adjoint system \eqref{AS} in the sense of Definition \ref{AWSD}\ , which satisfies the following estimate:
		\begin{flalign}\label{AEEE}
			\|\phi\|^2_{L^{\infty}(0,T;L^2(\Omega))}&+ \|\phi\|^2_{L^2(0,T;H^1(\Omega))} + \|\phi_t\|^2_{L^2(0,T;H^1(\Omega)^*)}\leq \left(\|\widetilde{m}(T)-m_\Omega\|^2_{L^2(\Omega)} +  \| \widetilde{m}- m_d\|^2_{L^2(0,T;L^2(\Omega))}\right)\nonumber\\
			& \times \ \exp\bigg\{C  \left(1+ \|\widetilde{m}\|^2_{L^2(0,T;H^3(\Omega))} + \|\widetilde{m}\|^4_{L^\infty(0,T;H^2(\Omega))}+\|\zetaup(\widetilde{U})\|^2_{L^2(0,T;H^1(\Omega))} \right) \bigg\}.	
		\end{flalign}
	\end{Thm}
	\noindent The proof of this theorem constitutes a special case of Lemma \ref{AL-SLS} by substituting $g=-\big(\widetilde{m}-m_d\big)$.
	
	Now we are ready to state a first-order optimality condition satisfied by the optimal control $\widetilde{U}\in \mathbb{U}_{ad}$.
	
	\begin{Thm}\label{FOOCT}
		Let $\widetilde{U} \in \mathbb{U}_{ad}$ be  the optimal control for the OCP with the associated state $\widetilde{m}$. If $\phi \in \mathcal{Z}$ is the weak solution of the adjoint system \eqref{AS} corresponding to the admissible pair $(\widetilde{m}, \widetilde{U})$, then the triplet $(\phi, \widetilde{m}, \widetilde{U})$ satisfies the following variational inequality: 
		\begin{equation}\label{FOOC}
			\sum_{i=1}^N\int_0^T  \widetilde{U_i}\  \big(U_i-\widetilde{U_i} \big)\  dt + \sum_{i=1}^N \int_0^T \big(U_i-\widetilde{U}_i\big)  \int_{\Omega} \big( \phi \times \widetilde{m}+  \phi\big) \cdot B_i \ dx\ dt \geq 0, \ \ \ \ \  \forall \ (U_1,...,U_N) \in \mathbb{U}_{ad}.	
		\end{equation}
	\end{Thm}		
	\noindent The proof of this theorem is provided in subsection \ref{SS-FOOC}. The following characterization of the optimal control is a direct consequence of \eqref{FOOC} from Theorem \ref{FOOCT}.
	
		\begin{Cor}
		Let $\widetilde{U}\in \mathbb{U}_{ad}$ be an optimal control of the minimization problem OCP and $\widetilde{m}\in \mathcal{M}$ be the corresponding state variable. Then the optimal control can be characterized by the projection formula 
		\begin{equation}\label{PF}
			\widetilde{U}_i(t)=\mathcal{P}_{[a_i(t),b_i(t)]}\left(-\int_{\Omega} \big(\phi \times \widetilde{m}+\phi\big) \cdot B_i\ dx\right)\ \ \ \text{for a.e.}\ t\in[0,T]\ \ \text{and}\ i=1,2,...,N,
		\end{equation}
		where the projection operator is defined as
		$$\mathcal{P}_{[\alpha,\beta]}(x)=min\ \big\{\ \beta,max\ \big\{\alpha,x\ \big\}\big\}\ \ \ \ \ \ \text{for}\ \  \alpha,\beta,x\in \mathbb{R}\ \ \  \text{with}\ \ \alpha\leq \beta. $$
	\end{Cor}

	Next, we will explore a local second-order optimality condition that the optimal control adheres to. While the first-order necessary optimality condition, encapsulated by the variational inequality \eqref{FOOC}, offers valuable insights into the optimality of a control, it is the second-order condition that provides a more clear understanding of the local behavior around an optimal solution. The condition $\frac{\partial^2 \mathcal{J}}{\partial U^2}(\widetilde{U})[h,h]>0$ for all directions $h\in L^2(0,T;\mathbb{R}^N) \ \backslash \ \{0\}$ ensures that the Hessian of the objective functional $\mathcal{J}$ is positive definite, signifying a strict local minimum. However, the need for a cone of critical directions (see, \cite{ECFT, FT}) arises to relax the stringent requirement that $h$ spans all controls in $L^2(0,T;\mathbb{R}^N)$ except ${0}$, making the second-order condition more adaptable and applicable in practical scenarios. 

	\begin{Def}{(Critical Cone)}\label{D-CC}
	For any control $\widetilde{U}\in \mathbb{U}_{ad}$, let $\Lambda(\widetilde{U})$ denotes the set of all $h\in L^2(0,T;\mathbb{R}^N)$ such that for $i\in \{1,2,...,N\}$ and almost all $t \in [0,T]$, 
	$$h_i(t)
	\begin{cases}
		\geq 0 \ \ \ \text{if}\ \ \widetilde{U}_i(t)=	a_i(t),\\
		\leq 0 \ \ \ \text{if}\ \ \widetilde{U}_i(t)= b_i(t),\\
		=0     \ \ \ \text{if}\ \ \left(\varUpsilon_{\widetilde{U}}\right)_i \neq 0,
	\end{cases}
	$$
	where $\left(\varUpsilon_{\widetilde{U}}\right)_i:=\widetilde{U}_i+\int _{\Omega}\big(\phi \times \widetilde{m}+\phi \big)\cdot B_i\ dx$, where $B=(B_1,B_2,...,B_N)$ is  the geometry of the coil given in \eqref{gotc}. 
	The set of controls $\Lambda(\widetilde{U})$ is called the cone of critical directions.
\end{Def}

	To establish a second-order sufficient optimality condition, it is imperative to demonstrate the Fréchet differentiability of both the control-to-state and control-to-costate operators, along with verifying their Lipschitz continuity. The culmination of these results enable us to articulate second-order conditions for local optimality.
	
	\begin{Thm}\label{T-SOLO}
		Let $\widetilde{U}\in \mathbb{U}_{ad}$ be any control satisfying the variational inequality \eqref{FOOC}. Assume that $\frac{\partial^2 \mathcal J}{\partial U^2}(\widetilde{U})[h,h]> 0$ for all $h\in \Lambda(\widetilde{U}) \backslash \{0\}$, that is 		
		\begin{align*}
			&\int_0^T  h^2\  dt + \int_{\Omega_T} \big(\phi'_{\widetilde{U}}[h]\times \widetilde{m}\big) \cdot \zetaup(h)\ dx\ dt + \int_{\Omega_T} \big(\phi_{\widetilde{U}} \times m'_{\widetilde{U}}[h]\big)\cdot \zetaup(h)\ dx\ dt& \nonumber\\
			&\hspace{3cm} + \int_{\Omega_T} \phi'_{\widetilde{U}}[h]\cdot \zetaup(h) \ dx\ dt >0 \hspace{1cm}  \forall \ h\in \Lambda(\widetilde{U}) \backslash \{0\}.
		\end{align*}
		Then there exist positive constants $\epsilon$ and $\sigma$ such that for any control $U\in \mathbb{U}_{ad}$ satisfying  $\|U-\widetilde{U}\|_{L^2(0,T;\mathbb{R}^N)} \leq \delta $, the following inequality holds:
		\begin{equation}\label{QGC}
			\mathcal{J}(U) \geq \mathcal{J}(\widetilde{U}) + \frac{\epsilon}{2} \ \|U-\widetilde{U}\|^2_{L^2(0,T;\mathbb{R}^N)}.
		\end{equation}  	
		In particular, the inequality \eqref{QGC} implies that $\widetilde{U}$ stands as a local minimum for the functional $\mathcal{J}$ within the set of admissible controls $\mathbb{U}_{ad}$.
	\end{Thm}
	\noindent The proof of Theorem \ref{T-SOLO} is presented in subsection \ref{SS-SOOC}. By virtue of the control-to-state operator we can define a \textit{reduced cost functional} $\mathcal{I}  : \mathbb{U}_{ad} \to \mathbb{R}$ by $\mathcal{I}(U)=\mathcal J(G(U),U)$. Therefore, the optimal control problem (OCP) can be redefined as follows:
	\begin{equation*}\text{(MOCP)}
		\left\{	\begin{array}{lclclc}
			\text{minimize} \ \ \mathcal{I}(U)	\\
			U\in \mathbb{U}_{ad}.
		\end{array}
		\right.
	\end{equation*}
	
	Though the second-order sufficient condition offers valuable insights into optimality, its inherently local nature necessitates the exploration of global optimality conditions.
	To derive global optimality conditions, we have adopted the methodologies outlined in \cite{ADH, MEPK, SKG}. By incorporating these techniques, we aim to extend our analysis beyond local optimality considerations, ensuring a more thorough exploration of optimal control dynamics for the given system through a condition on the solution of the direct problem \eqref{NLP-EV} and the adjoint problem \eqref{AS}.

	\begin{Thm}\label{T-GO}
		Assume any control $\widetilde{U} \in \mathbb{U}_{ad}$ with the associated solution $\widetilde{m}\in \mathcal{M}$ and adjoint state $\phi \in \mathcal{Z}$ satisfies the variational inequality \eqref{FOOC}. Moreover, if the triplet $(\widetilde{U},\widetilde{m},\phi)$ satisfies the following condition:
		\begin{equation}\label{GO-C}
			C(\Omega,T,a,b,m_0)\left\{ 1+ \|\widetilde{m}\|_{L^2(0,T;H^1(\Omega))}\right\}  \|\phi\|_{L^2(0,T;L^2(\Omega))} \leq \frac{1}{2},
		\end{equation}
		then $\widetilde{U}$ is a global optimal control of MOCP.
		Furthermore, in the event of strict inequality in \eqref{GO-C}, the global optimum is unique. 
	\end{Thm}
	
	 

	Since the OCP is non-convex, the uniqueness of the optimal control cannot be expected. However, it is essential for implementing accurate numerical methods to approximate the optimal control and for ensuring the stability and efficiency of these computational approaches. Therefore, finally, we also establish a condition on the time `T' that ensures the uniqueness of the optimal control.
	\begin{Thm}\label{UOO}
		Let $\widetilde{U}\in \mathbb{U}_{ad}$ be an optimal control of the minimization problem OCP. Then $\widetilde{U}$ is the unique optimal control under the following condition:
		\begin{equation}\label{ULOC}
		 2C^4_{4,n} \|B\|^2_{H^1(\Omega)} \bigg(4C_2 \|\phi_{\widetilde{U}}\|^2_{L^\infty(0,T;L^2(\Omega))} + \left( C_2 +2 C_3\right)\left( 4C_2 C_{a,b} + \|\widetilde{m}\|_{L^\infty(0,T;H^1(\Omega))}+|\Omega|\right) \bigg)<\frac{1}{T},
		\end{equation}
		where $C_{a,b}=\|a\|^2_{L^2(0,T;\mathbb{R}^N)}+\|b\|^2_{L^2(0,T;\mathbb{R}^N)}$, $C_{4,n}(\Omega)$ is the embedding constant for  $H^1(\Omega) \hookrightarrow L^4(\Omega))$, and $C_2$ and $C_3$ are the Lipschitz constants obtained in estimates \eqref{LCCTSO} and \eqref{LCI}, respectively.
	\end{Thm}
\noindent The proofs of Theorem \ref{T-GO} and Theorem \ref{UOO} are given in section \ref{S-GOC} and section \ref{S-UOO}, respectively.

	\subsection{Inequalities}
	
	\noindent In this subsection, we present specific equivalent norms and inequalities that play a crucial role in our paper. We begin by discussing the basic properties of the cross product and equivalent norm estimates as presented in the following lemmas.
	\begin{Lem}\label{CPP}
		Let $a,b$ and $c$ be three vectors of $\mathbb{R}^3$, then the following vector identities hold: $a\cdot(b \times c)=-(b \times a)\cdot c$,  $a \cdot (a \times b)=0$,  $a \times (b \times c)=(a \cdot c) b - (a \cdot b) c$. Moreover, assume that $1 \leq r,s \leq \infty, \ (1/r)+(1/s)=1$ and $p\geq 1$, then if $f\in L^{pr}(\Omega)$ and $g\in L^{ps}(\Omega),$ we have
		\begin{equation*}\label{ES0}
			\|f \times g\|_{L^p(\Omega)} \leq \|f\|_{L^{pr}(\Omega)} \|g\|_{L^{ps}(\Omega)}.	
		\end{equation*}	
	\end{Lem}
	\noindent The proof of Lemma \ref{CPP} follows from a direct application of H\"older's inequality.
	\begin{Lem}[see, \cite{KW}]\label{EN}
		Let $\Omega$ be a bounded smooth domain in $\mathbb{R}^n$ and $k \in \mathbb{N}$. There exists a constant $C_{k,n}>0$ such that for all $m \in H^{k+2}(\Omega)$ and $\frac{\partial m}{\partial \eta}\big|_{\partial \Omega}=0,$ it holds that
		\begin{equation*}\label{ES1}
			\|m\|_{H^{k+2}(\Omega)} \leq C_{k,n} \left(\|m\|_{L^2(\Omega)}+ \|\Delta m\|_{H^k(\Omega)}\right).	
		\end{equation*}
	\end{Lem}
	\noindent Using Lemma \ref{EN}, we can define an equivalent norm on $H^{k+2}(\Omega)$ as follows 
	$$\|m\|_{H^{k+2}(\Omega)}:=\|m\|_{L^2(\Omega)}+\|\Delta m\|_{H^k(\Omega)}.$$
	
	The following inequalities will be consistently utilized in this paper. It is crucial to emphasize that the constant $C>0$ may assume different values for each inequality, ranging from the estimate \eqref{ES2} to \eqref{AEE-10}.
	\begin{Pro}\label{PROP1}
		Let $\Omega$ be a regular bounded subset of  $\mathbb{R}^2$ or $\mathbb{R}^3$. There exists a constant $C>0$ depending on $\Omega$ such that for all $m \in H^2(\Omega)$ with $\frac{\partial m}{\partial \eta}\big|_{\partial\Omega}=0,$ we have
		\begin{eqnarray} 
			\|m\|_{L^\infty(\Omega)} &\leq& C\ \left(\|m\|^2_{L^2(\Omega)}+ \|\Delta m\|^2_{L^2(\Omega)}\right)^{\frac{1}{2}},\label{ES2}\\
			\|\nabla m\|_{L^s(\Omega)} &\leq& C\  \|\Delta m\|_{L^2(\Omega)}, \ \ \forall \ s \in [1,6],\label{ES3}\\
			\|D^2m\|_{L^2(\Omega)} &\leq& C\ \|\Delta m\|_{L^2(\Omega)}.\label{ES4}	
		\end{eqnarray}
		Moreover, for every $m \in H^3(\Omega)$ with $\frac{\partial m}{\partial \eta}\big|_{\partial\Omega}=0,$ we have
		\begin{eqnarray}
			\|\Delta m\|_{L^2(\Omega)} &\leq& C\ \|\nabla \Delta m\|_{L^2(\Omega)},\label{ES7}\\ 	
			\|D^3m\|_{L^2(\Omega)} &\leq& C\ \|\nabla \Delta m\|_{L^2(\Omega)},\label{ES10}\\
			\|D^2m\|_{L^3(\Omega)} &\leq& C\ \|\Delta m\|^\frac{1}{2}_{L^2(\Omega)} \|\nabla \Delta m\|^\frac{1}{2}_{L^2(\Omega)}.\label{ES9} 		
		\end{eqnarray}
	\end{Pro}
	The proof of this proposition can be found in Proposition 2.1, \cite{SPSK}.
	
	Some dimension specific particular cases of the Gagliardo-Nirenberg inequality(see, \cite{LN}) combined with estimates \eqref{ES3}-\eqref{ES9}  are given as follows: 
	\begin{numcases}{(n=2)}
		\|\nabla m\|_{L^4(\Omega)} \leq C\ \|\nabla m\|^{\frac{1}{2}}_{L^2(\Omega)} \|\Delta m\|^{\frac{1}{2}}_{L^2(\Omega)}, \ \ \label{ES5}	\\
		\|D^2m\|_{L^4(\Omega)}\leq C\ \|\Delta m\|^{\frac{1}{2}}_{L^2(\Omega)}\ \|\nabla \Delta m\|^{\frac{1}{2}}_{L^2(\Omega)}.\label{ES20}	
	\end{numcases}
	\begin{numcases}{(n=2, 3)}
		\|m\|_{L^\infty(\Omega)}\ \  \leq C\ \|m\|_{L^6(\Omega)}^{\frac{1}{2}} \ \|m\|_{H^2(\Omega)}^{\frac{1}{2}}. \label{ES16}
	\end{numcases}
	
	We have consistently utilized the following inequalities to evaluate the $L^2(0,T;H^1(\Omega))$ norm of various cross products and non-linear terms throughout the paper. 
	\begin{Lem}\label{PROP2}
		Let $\Omega$ be a regular bounded domain of $\mathbb{R}^n$ for $n=2,3$. Then there exists a constant $C>0$ depending on $\Omega$ and $T$ such that 
		\begin{enumerate}[(\roman*)]
			\item  for $\xi \in L^\infty(0,T;H^2(\Omega))$ and $\zeta\in L^2(0,T;H^3(\Omega))$,
			\begin{equation}\label{EE-2}
				\|\xi \times \Delta \zeta\|^2_{L^2(0,T;H^1(\Omega))}\leq  C\ \|\xi\|^2_{L^\infty(0,T;H^2(\Omega))}\  \|\zeta\|^2_{L^2(0,T;H^3(\Omega))},
			\end{equation}
			\item for $\xi\in L^\infty(0,T;H^2(\Omega))$ and $\omega \in L^2(0,T;H^1(\Omega))$,
			\begin{equation}\label{EE-3}
				\|\xi \times \omega\|^2_{L^2(0,T;H^1(\Omega))} \leq  C\ \|\xi\|^2_{L^\infty(0,T;H^2(\Omega))}\  \|\omega\|^2_{L^2(0,T;H^1(\Omega))},
			\end{equation}
			\item  for $\xi,\zeta,\omega \in L^\infty(0,T;H^2(\Omega))$,
			\begin{equation}\label{EE-6}
				\|(\xi\cdot \zeta)\ \omega\|^2_{L^2(0,T;H^1(\Omega))} \leq C\ \|\xi\|^2_{L^\infty(0,T;H^2(\Omega))}\ \|\zeta\|^2_{L^\infty(0,T;H^2(\Omega))}\ \|\omega\|^2_{L^2(0,T;H^2(\Omega))}.
			\end{equation}
		\end{enumerate}
	\end{Lem}
		The paper also requires $L^2(0,T;H^1(\Omega)^*)$ norm estimates for various terms as given below.
		\begin{Lem}\label{L-CP}	
			Let $\Omega$ be a bounded subset of $\mathbb{R}^2$ with smooth boundary. Then there exists a constant $C>0$ such that 
			\begin{enumerate}[(\roman*)]
				\item for $\xi \in L^2(0,T;H^1(\Omega))$ and $\zeta \in L^\infty(0,T;H^2(\Omega))$ with $\frac{\partial \xi}{\partial \eta}=\frac{\partial \zeta}{\partial \eta}=0$,
				\begin{equation}\label{AEE-3}
					\|\Delta (\xi\times \zeta)\|_{L^2(0,T;H^1(\Omega)^*)} \leq  C\ \|\xi\|_{L^2(0,T;H^1(\Omega))}\  \|\zeta\|_{L^\infty(0,T;H^2(\Omega))},
				\end{equation}
				\item  for $\xi \in L^\infty(0,T;H^2(\Omega))$ and $\omega\in L^2(0,T;H^1(\Omega))$,
				\begin{equation}\label{AEE-4}
					\|\Delta \xi\times \omega \|_{L^2(0,T;H^1(\Omega)^*)}\leq C\ \|\xi\|_{L^\infty(0,T;H^2(\Omega))} \ \|\omega\|_{L^2(0,T;H^1(\Omega))},
				\end{equation} 
				\item  for $\xi \in L^\infty(0,T;L^2(\Omega))$ and $\omega\in L^2(0,T;H^1(\Omega))$,
				\begin{equation}\label{AEE-8}
					\| \xi\times \omega \|_{L^2(0,T;H^1(\Omega)^*)}\leq C\ \|\xi\|_{L^\infty(0,T;H^2(\Omega))} \ \|\omega\|_{L^2(0,T;H^1(\Omega))},
				\end{equation} 
				\item  for $\xi,\zeta \in L^\infty(0,T;H^1(\Omega))$ and $\omega\in L^2(0,T;L^2(\Omega))$,
				\begin{equation}\label{AEE-9}
					\| (\xi\cdot \zeta) \omega \|_{L^2(0,T;H^1(\Omega)^*)}\leq C\ \|\xi\|_{L^\infty(0,T;H^1(\Omega))} \ \|\zeta\|_{L^\infty(0,T;H^1(\Omega))} \  \|\omega\|_{L^2(0,T;L^2(\Omega))},
				\end{equation} 
				\item  for $\xi,\zeta \in L^\infty(0,T;H^1(\Omega))$ and $\omega\in L^2(0,T;L^2(\Omega))$,
				\begin{equation}\label{AEE-10}
					\| (\xi\cdot \zeta) \ \omega \|_{L^2(0,T;H^1(\Omega)^*)}\leq C\ \|\xi\|_{L^\infty(0,T;H^1(\Omega))} \ \|\omega\|_{L^2(0,T;L^2(\Omega))} \ \|\zeta\|_{L^\infty(0,T;H^1(\Omega))}.
				\end{equation} 
			\end{enumerate}
		\end{Lem}
		The proofs of Lemmas \ref{PROP2} and \ref{L-CP} are available in \cite{SPSK2}. While some of the estimates may not be identical to those in \cite{SPSK2}, they can be derived by similar arguments.

\section{Strong Solution and Regular Solution}\label{S-WSRS}

\subsection {Existence of Strong Solution}\label{SS-EOSS}
In this subsection we will prove the strong solvability of system \eqref{NLP} by applying the method of Galerkin approximation. 
\begin{proof}[Proof of Theorem \ref{T-SS}]
	Let $\xi_j$ be the $j^{th}$ eigenfunction corresponding to the eigenvalue $\rho_j$ of the operator $-\Delta + I$ with Neumann boundary condition, that is, $(-\Delta +I) \xi_j=\rho_j \xi_j$ with $\frac{\partial \xi_j}{\partial \eta}\big|_{\partial \Omega}=0$ such that $\{\xi_j\}^\infty_{j=1}$ forms a orthonormal basis of $L^2(\Omega)$. Let $W_n:= span \{\xi_1,\xi_2,...,\xi_n\}$ be the finite dimenisonal subspace of $L^2(\Omega)$ and $\mathbb{P}_n:L^2(\Omega)\to W_n$ be the orthogonal projection. Consider the following Galerkin system
	
		\begin{equation}\label{GA}
		\begin{cases}
			(m_n)_t- \Delta m_n=\mathbb{P}_n \left[ m_n \times \Delta m_n + m_n \times u-\big(1+|m_n|^2\big)\ m_n +u \right],\\
			m_n(0)=\mathbb{P}_n(m_0)
		\end{cases}
	\end{equation}
	where $m_n(t)=\sum_{k=1}^{n}a_{kn}(t)\ \xi_k\in W_n$ and $\mathbb{P}_n(m_0)=\sum_{k=1}^n b_{kn} \ \xi_k$. Then system \eqref{GA} is equivalent to the following system of ordinary differential equations
	\begin{equation}\label{ODE}  
		\frac{d}{dt}a_{kn}(t)= F_k(t,a_{n}), \ \ a_{k}(0)=b_{k}, \ \ \ \ \ k=1,2,...,n,
	\end{equation}
	where $a_{n}=(a_{1n}(t),a_{2n}(t),...,a_{nn}(t))^T$ and 
	\begin{align*}
		F_k(t,a_{n})& = -\rho_k\ a_{kn}(t)- \sum_{r,s=1}^{n} (\rho_s-1) \big(a_{rn}(t)\times a_{sn}(t)\big)\int_\Omega \xi_r \ \xi_s \ \xi_k\ dx \\
		& +  \sum_{r=1}^{n} \int_\Omega \big(a_{rn}(t)\times u\big)\ \xi_r\  \xi_k\  dx - \sum_{r,s,v=1}^{n} \big(a_{rn}(t) \ a_{sn}(t)\ a_{vn}(t)\big)\int_\Omega  \xi_r \ \xi_s \ \xi_v\ \xi_k\ dx + \int_{\Omega}u\  \xi_k\ dx.	
	\end{align*}	
	The dense embedding of $C([0,T];L^2(\Omega))$ in $L^2(0,T;L^2(\Omega))$ helps us to assume that the control function $u \in C([0,T];L^2(\Omega))$. Now, considering the continuity of $F_k(t,a_n)$ in $[0,T]\times \mathbb{R}^n$, the existence theory of ODEs (see \cite{PH}) ensures the existence of a solution $a_n \in C^1([0,t_m);\mathbb{R}^n)$, with $t_m$ being the maximal time of existence. However, we  will provide \emph{a priori} estimates to show that $|a_n(t)|$ remains bounded within $[0, T]$, that is, the solution of the ODE \eqref{ODE} exists globally.
	
	\noindent Taking $L^2$ inner product of equation \eqref{GA} with $m_n$ and using the fact that $a\cdot (a\times b)=0$, we find
	\begin{equation}\label{SS-E1}
		 \frac{d}{dt} \|m_n(t)\|^2_{L^2(\Omega)} + \|\nabla m_n(t)\|^2_{L^2(\Omega)} +  \|m_n(t)\|^2_{L^2(\Omega)} + \|m_n(t)\|^4_{L^4(\Omega)} \leq \|u(t)\|^2_{L^2(\Omega)}.
	\end{equation}
	Now, an $L^2$ inner product of equation \eqref{GA} with $-\Delta m_n$ gives us
	\begin{equation}\label{SS-E2}
		\frac{1}{2} \frac{d}{dt} \|\nabla m_n(t)\|^2_{L^2(\Omega)} + \|\Delta m_n(t)\|^2_{L^2(\Omega)} = -\ (m_n \times u, \Delta m_n) + \big(\big(1+|m_n|^2\big)m_n, \Delta m_n\big) - (u, \Delta m_n).
	\end{equation}
Next, upon analyzing the second term on the right-hand side, it becomes evident that it is non-positive for all time $t>0$. This conclusion is drawn by performing integration by parts across the spatial domain as given below:
	\begin{flalign*}
		\big(\big(1+|m_n|^2\big)m_n, \Delta m_n\big) &= - \sum_{i=1}^{n} \int_{\Omega} \left[\big(1+|m_n|^2\big) \frac{\partial m_n}{\partial x_i} \cdot \frac{\partial m_n}{\partial x_i} + 2 \left(m_n \cdot \frac{\partial m_n}{\partial x_i}\right) m_n \cdot \frac{\partial m_n}{\partial x_i} \right]\ dx&\\
		&= -\int_{\Omega} \left(\big(1+|m_n|^2\big) \ |\nabla m_n|^2 + 2\ |m_n \cdot \nabla m_n|^2\right)\ dx\leq 0.
	\end{flalign*}
For the last term $(u,\Delta m_n)$ of \eqref{SS-E2} doing a straight forward estimation, we can find
\begin{flalign*}
	(u,\Delta m_n)\leq \|u(t)\|_{L^2(\Omega)} \ \|\Delta m_n(t)\|_{L^2(\Omega)} \leq \epsilon \ \|\Delta m_n(t)\|^2_{L^2(\Omega)}+ C(\epsilon) \ \|u(t)\|^2_{L^2(\Omega)}.
\end{flalign*} 

The estimation for $(m_n \times u, \Delta m_n)$ will be different for dimension $n=1,2$ and $3$. First, for $n=1$ choosing $s=2$ and then using the embedding $H^1(\Omega)\hookrightarrow L^\infty(\Omega)$, we find
$$(m_n \times u, \Delta m_n) \leq \|m_n(t)\|_{L^\infty(\Omega)} \ \|u(t)\|_{L^2(\Omega)} \ \|\Delta m_n(t)\|_{L^2(\Omega)}\leq \epsilon \ \|\Delta m_n(t)\|^2_{L^2(\Omega)} + C(\epsilon) \   \|u(t)\|^2_{L^2(\Omega)}\ \|m_n(t)\|^2_{H^1(\Omega)}.$$
For $n=2$, choose $s=2+\delta$ for some real number $\delta >0$. Moreover, for $n=3$, choose any $s=3$. Applying H\"older's inequality and the continuous embedding $H^1(\Omega) \hookrightarrow L^s(\Omega)$, we obtain 
\begin{flalign*}
	(m_n \times u, \Delta m_n) &\leq \|m_n(t)\|_{L^{\frac{2s}{s-2}}(\Omega)} \ \|u(t)\|_{L^s(\Omega)}\ \|\Delta m_n(t)\|_{L^2(\Omega)}&\\
	&\leq \epsilon \ \|\Delta m_n(t)\|^2_{L^2(\Omega)} + C(\epsilon)\ \|u(t)\|^2_{L^s(\Omega)} \ \|m_n(t)\|^2_{H^1(\Omega)} .
\end{flalign*}
Substituting all these estimates in equation \eqref{SS-E2} and choosing a suitable value for $\epsilon$, we get
	\begin{align}\label{SS-E3}
	\frac{d}{dt} \|\nabla m_n(t)\|^2_{L^2(\Omega)} &+ \|\Delta m_n(t)\|^2_{L^2(\Omega)} + \int_{\Omega} \left(\big(1+|m_n|^2\big) \ |\nabla m_n|^2 + 2\ |m_n \cdot \nabla m_n|^2\right)\ dx&\nonumber\\
	&\hspace{0.5cm}\leq C\ \|u(t)\|^2_{L^2(\Omega)} + C \ \|u(t)\|^2_{L^s(\Omega)} \  \|m_n(t)\|^2_{H^1(\Omega)},
	\end{align}
	where $s$ takes values as defined above depending on $n$. Now, combining estimates \eqref{SS-E1} and \eqref{SS-E3}, we left with the following
\begin{equation*}
	\frac{d}{dt} \left(\|m_n(t)\|^2_{L^2(\Omega)}+\|\nabla m_n(t)\|^2_{L^2(\Omega)} \right) + \|\Delta m_n(t)\|^2_{L^2(\Omega)} \leq C\ \|u(t)\|^2_{L^2(\Omega)} + C\ \|u(t)\|^2_{L^s(\Omega)}\  \|m_n(t)\|^2_{H^1(\Omega)}.		
\end{equation*}
By implementing Gronwall's inequality and the estimate $\|\mathbb{P}_n(m_0)\|^2_{H^1(\Omega)} \leq \|m_0\|^2_{H^1(\Omega)}$, we obtain
\begin{eqnarray}\label{SS-E5}
	\lefteqn{\|m_n(t)\|^2_{L^2(\Omega)}+\|\nabla m_n(t)\|^2_{L^2(\Omega)} + \int_0^t \|\Delta m_n(\tau)\|^2_{L^2(\Omega)} d\tau } \nonumber\\
	&&\leq C \left(\|m_0\|^2_{H^1(\Omega)} + \ \|u\|^2_{L^2(0,T;L^2(\Omega))}\right) \exp\left\{C\ \|u\|^2_{L^2(0,T;L^s(\Omega))}\right\}\ \ \ \ \ \forall \ t\in [0,T].
\end{eqnarray}
Therefore, by virtue of Lemma \ref{EN}, the sequence $\{m_n\}$ is uniformly bounded in $L^\infty(0,T;H^1(\Omega))$ and $L^2(0,T;H^2(\Omega))$.

Next, we will estimate a uniform bound for the sequence $\{(m_n)_t\}$. For the first term $m_n \times \Delta m_n$, considering the $L^{\frac{4}{3}}(0,T;L^2(\Omega))$ norm and then applying H\"older's inequality followed by Gagliardo-Nirenberg interpolation inequality \eqref{ES16}, we obtain
\begin{flalign*}
	&\int_0^T \left(\int_{\Omega} |m_n \times \Delta m_n|^2\ dx\right)^{\frac{2}{3}} dt \leq \int_0^T \|m_n(t)\|^{\frac{4}{3}}_{L^\infty(\Omega)} \ \|\Delta m_n(t)\|^{\frac{4}{3}}_{L^2(\Omega)}\ dt\\
	&\hspace{1cm}\leq C\ \int_0^T \|m_n(t)\|_{L^6(\Omega)}^{\frac{2}{3}}\ \|m_n(t)\|_{H^2(\Omega)}^{\frac{2}{3}}\ \|\Delta m_n(t)\|^{\frac{4}{3}}_{L^2(\Omega)}\ dt \leq C\ \int_0^T \|m_n(t)\|_{H^1(\Omega)}^{\frac{2}{3}}\ \|m_n(t)\|_{H^2(\Omega)}^{2}\ dt\\
	&\hspace{1cm}\leq C\ \|m_n\|_{L^\infty(0,T;H^1(\Omega))}^{\frac{2}{3}} \ \|m_n\|_{L^2(0,T;H^2(\Omega))}^{2}.
\end{flalign*}
The above estimate is proved for $n=3$, which holds true  for $n=1$ and $2$ as well. Similarly, estimating the $L^{\frac{4}{3}}(0,T;L^2(\Omega))$ norm of $m_n\times u$ and $\left(1+|m_n|^2\right)m_n$, we derive
\begin{flalign*}
	&\int_0^T \left(\int_{\Omega} |m_n \times u|^2\ dx\right)^{\frac{2}{3}} dt \leq \int_0^T \|m_n(t)\|_{L^{\frac{2s}{s-2}}(\Omega)}^{\frac{4}{3}} \ \|u(t)\|^{\frac{4}{3}}_{L^s(\Omega)}\ dt\\
	&\hspace{1.5cm}\leq C \int_0^T \|m_n(t)\|_{H^1(\Omega)}^{\frac{4}{3}} \ \|u(t)\|^{\frac{4}{3}}_{L^s(\Omega)}\ dt \leq C\  \|m_n\|_{L^\infty(0,T;H^1(\Omega))}^{\frac{4}{3}} \  \|u\|^{\frac{4}{3}}_{L^2(0,T;L^s(\Omega))},
\end{flalign*} 
where we use the same range of $s$ as defined in the statement of this theorem. Moreover, we yield
\begin{flalign*}
	\int_0^T& \left(\int_{\Omega}\left|\left(1+|m_n|^2\right)m_n\right|^2 dx\right)^{\frac{2}{3}} dt \leq  \int_0^T 2^{\frac{2}{3}} \left(\|m_n(t)\|_{L^2(\Omega)}+\|m_n(t)\|^3_{L^6(\Omega)}\right)^{\frac{4}{3}} dt\\
	&\leq \int_0^T 2 \left(\|m_n(t)\|^{\frac{4}{3}}_{L^2(\Omega)}+\|m_n(t)\|^{4}_{L^6(\Omega)}\right) dt \leq C\ \left(\|m_n\|^{\frac{4}{3}}_{L^\infty(0,T;L^2(\Omega))}+ \|m_n\|^{4}_{L^\infty(0,T;H^1(\Omega))}\right).
\end{flalign*}
Therefore, considering the $L^{\frac{4}{3}}(0,T;L^2(\Omega))$ norm in equation \eqref{GA} and then substituting the above estimates, we find that
\begin{align}\label{SS-E6}
\|(m_n)_t\|^{2}_{L^{\frac{4}{3}}(0,T;L^2(\Omega))} \leq & \	C\ \left[ \|m_n\|^2_{L^2(0,T;H^2(\Omega))}  +   \|m_n\|_{L^\infty(0,T;H^1(\Omega))} \ \|m_n\|_{L^2(0,T;H^2(\Omega))}^{3} +  \|m_n\|^2_{L^\infty(0,T;L^2(\Omega))} \right.\nonumber\\
&\hspace{0.4cm}\left. +\  \|m_n\|^{6}_{L^\infty(0,T;H^1(\Omega))}
+ \|u\|^2_{L^2(0,T;L^2(\Omega))} +\|m_n\|_{L^\infty(0,T;H^1(\Omega))}^{2} \  \|u\|^{2}_{L^2(0,T;L^s(\Omega))} \right].
\end{align}
Applying the uniform bounds for $\{m_n\}$ in $L^\infty(0,T;H^1(\Omega))$ and $L^2(0,T;H^2(\Omega))$ from estimate \eqref{SS-E5} in \eqref{SS-E6}, we obtain
\begin{flalign}\label{SS-E7}
	\|(m_n)_t\|^2_{L^{\frac{4}{3}}(0,T;L^2(\Omega))} \leq C\  \left(1+\|m_0\|^2_{H^1(\Omega)}+\|u\|^2_{L^2(0,T;L^2(\Omega))}\right)^3\ \exp\left\{C\ \|u\|^2_{L^2(0,T;L^s(\Omega))} \right\}.
\end{flalign}
As a result, utilizing the uniform bound estimates \eqref{SS-E5} and \eqref{SS-E6}, along with the Alaoglu weak$^*$ compactness theorem and the reflexive weak compactness theorem (Theorem 4.18, \cite{JCR}), we obtain the following convergences:
	\begin{eqnarray} \left\{\begin{array}{cccll}
		m_n &\overset{w}{\rightharpoonup} & m \  &\mbox{weakly in}& \ L^2(0,T;H^2(\Omega)),\\
		m_n &\overset{w^\ast}{\rightharpoonup} & m \ &\mbox{weak$^*$ in}& \ L^{\infty}(0,T;H^1(\Omega)),\\
		(m_n)_t &\overset{w}{\rightharpoonup} & m_t \ &\mbox{weakly in}& \ L^{\frac{4}{3}}(0,T;L^2(\Omega)), \ \ \mbox{as} \  \ n\to \infty. \label{wc}
	\end{array}\right.	
\end{eqnarray}
Applying the Aubin-Lions-Simon lemma (refer to Corollary 4 in \cite{JS}), we can extract a subsequence from $\{m_n\}$ (which we will also denote by $\{m_n\}$ for simplicity) that converges strongly to \(m\) in both \(L^2(0,T;H^1(\Omega))\) and \(C([0,T];L^2(\Omega))\). Using this and the weak-weak$^*$ convergences from \eqref{wc}, we can find the convergence of the approximated system to the main equation \eqref{NLP}. Let us show the convergence of one of the non-linear terms:
$$ \int_0^T \int_{\Omega} \mathbb{P}_n\big(m_n \times \Delta m_n\big)\cdot v\ dx\ dt\to \int_0^T\int_{\Omega}\big(m\times \Delta m\big)\cdot v\ dx\ dt.$$
For any $v=\sum_{k=1}^n a_k(t) \xi_k$ with $a_k\in C([0,T])$, using the fact that $\mathbb{P}_n (v)=v$ and applying H\"older's inequality, estimate \eqref{ES16} along with the embedding $H^1(\Omega) \rightharpoonup L^6(\Omega)$, we derive
\begin{flalign*}
	\lefteqn{\left|\int_0^T \int_{\Omega} \big(m_n \times \Delta m_n -m \times \Delta m\big)\cdot v\ dx\ dt\right|   } \\
	&\leq  \int_0^T \int_{\Omega}|m_n-m|\ |\Delta m_n|\ |v|\ dx\ dt + \underbrace{\left| \int_0^T\int_{\Omega}(v\times m)\cdot (\Delta m_n-\Delta m)\ dx\ dt\right|}_{:=\mathcal{H}}\\
	&\leq \sum_{k=1}^n\int_0^T\|m_n(t) -m(t) \|_{L^\infty(\Omega)}\ \|\Delta m_n(t)\|_{L^2(\Omega)}\ \|\xi_k\|_{L^2(\Omega)}\ |a_k(t)|\ dt+ \mathcal{H}\\
	&\leq \sum_{k=1}^n \int_0^T|a_k(t)|\ \|m_n(t)-m(t)\|^{\frac{1}{2}}_{L^6(\Omega)}\ \|m_n(t)-m(t)\|^{\frac{1}{2}}_{H^2(\Omega)}\ \|\Delta m_n(t)\|_{L^2(\Omega)}\ dt+\mathcal{H}\\
	&\leq  \sum_{k=1}^n\|a_k\|_{C([0,T])}\ \|m_n-m\|^{\frac{1}{2}}_{L^2(0,T;H^1(\Omega))} \|m_n-m\|^{\frac{1}{2}}_{L^2(0,T;H^2(\Omega))} \|m_n\|_{L^2(0,T;H^2(\Omega))}+\mathcal{H} \ \to 0\ \ \ \text{as}\ n\to\infty,
\end{flalign*}
where we have used the uniform boundedness of $\{m_n\}$ in $L^2(0,T;H^2(\Omega))$ and the strong convergence of $\{m_n\}$ in $L^2(0,T;H^1(\Omega))$. The convergence of the second term $\mathcal{H}$ relies on the fact that $\Delta m_n \rightharpoonup \Delta m$ weakly in $L^2(0,T;L^2(\Omega))$ and $v\times m\in L^2(0,T;L^2(\Omega))$. Convergences of the rest of the terms can be shown in a similar way. Moreover, since the set of such $v=\sum_{k=1}^n a_k(t) \xi_k$ with $a_k\in C([0,T])$ are dense in $L^4(0,T;L^2(\Omega))$, the equation in \eqref{NLP} holds as an equality in $L^{\frac{4}{3}}(0,T;L^2(\Omega))$. 

The verification of initial condition and the uniqueness (for n=2) of the strong solution can be demonstrated classically. Finally, combining estimate \eqref{SS-E5} and \eqref{SS-E7}, and using the weak sequential lower semi continuity, we find the required result \eqref{SSEE}.
\end{proof}

\subsection{Existence of Regular Solution}\label{SS-EORS}
In this subsection, we will demonstrate that for dimension $n=2$ and $3$, if we assume the control function and initial data to be more regular then we can find the state solution in higher regular spaces. 

\begin{proof}[Proof of Theorem \ref{T-RS}:]
In the Galerkin approximated system \eqref{GA}, considering $L^2$ inner product with $\Delta^2 m_n$, we have
\begin{align}\label{REQ-1}
	\frac{1}{2} \frac{d}{dt}\|\Delta m_n(t)\|^2_{L^2(\Omega)} &+ \|\nabla \Delta m_n(t)\|^2_{L^2(\Omega)} = - \Big(\nabla \left(m_n\times \Delta m_n\right), \nabla \Delta m_n\Big) - \Big( \nabla \big(m_n\times u\big),\nabla  \Delta m_n\Big)\ \ \ \ \nonumber\\
	&-  \big(\nabla u,\nabla \Delta m_n\big	) +\left(\nabla \left(\left(1+|m_n|^2\right)m_n\right),\nabla \Delta m_n\right) := \sum_{i=1}^{4} E_i. \ \ \ \ \ \ \ 
\end{align}
Since the inequalities outlined in Proposition \ref{PROP1} are dimension specific, certain estimates for $n=2$ and $n=3$ will be different. Initially, we will address the global solution scenario for $n=2$. 

\noindent \textbf{Part I:\ (Global existence for n=2)} \ For the first term $E_1$ utilizing the equality $a\cdot (a \times b)=0$, applying H\"older's inequality, and incorporating the estimates \eqref{ES5} and \eqref{ES20}, we determine 
\begin{flalign*}
	&E_1= - \int_\Omega \left(\nabla m_n(t) \times \Delta m_n(t)\right) \cdot \nabla \Delta m_n(t)\ dx \leq \|\nabla m_n(t)\|_{L^4(\Omega)} \ \|\Delta m_n(t)\|_{L^4(\Omega)} \ \|\nabla \Delta m_n(t)\|_{L^2(\Omega)}&\\
	& \leq C\ \|\nabla m_n(t)\|_{L^2(\Omega)}^{\frac{1}{2}}   \|\Delta m_n(t)\|_{L^2(\Omega)} \ \|\nabla \Delta m_n(t)\|_{L^2(\Omega)}^{\frac{3}{2}}	\leq \epsilon \ \|\nabla \Delta m_n(t)\|_{L^2(\Omega)}^2 + C\ \|\nabla m_n(t)\|^2_{L^2(\Omega)}\ \|\Delta m_n(t)\|^4_{L^2(\Omega)}. 
\end{flalign*}

For the second term $E_2$ and the third term $E_3$, the inequalities \eqref{ES2}, \eqref{ES3} and the embedding $H^1(\Omega)\hookrightarrow L^4(\Omega)$ lead to the estimates
\begin{flalign*}
	E_2 &=  - \int_\Omega \bigg[ \Big(\nabla m_n(t) \times u(t)\big)\cdot  \nabla \Delta m_n(t)+\big(m_n(t) \times \nabla u(t)\big)\cdot  \nabla \Delta m_n(t)\bigg]\ dx&\\
	&\leq \  \|\nabla m_n(t)\|_{L^4(\Omega)} \ \|u(t)\|_{L^4(\Omega)} \ \|\nabla \Delta m_n(t)\|_{L^2(\Omega)} + \  \|m_n(t)\|_{L^\infty(\Omega)} \ \|\nabla u(t)\|_{L^2(\Omega)} \ \|\nabla \Delta m_n(t)\|_{L^2(\Omega)}\\
	&\leq \epsilon   \ \|\nabla \Delta m_n(t)\|^2_{L^2(\Omega)} + C(\epsilon) \ \|u(t)\|^2_{H^1(\Omega)} \left(\|m_n(t)\|^2_{L^2(\Omega)}+\|\Delta m_n(t)\|^2_{L^2(\Omega)} \right),
\end{flalign*}
and
\begin{flalign*}
	E_3 &= -\int_\Omega \nabla u(t)\cdot \nabla \Delta m_n(t) \ dx \leq \|\nabla u(t)\|_{L^2(\Omega)}\ \|\nabla \Delta m_n(t)\|_{L^2(\Omega)}\leq \epsilon \ \|\nabla \Delta m_n(t)\|^2_{L^2(\Omega)} + C(\epsilon) \ \|u(t)\|^2_{H^1(\Omega)}.&
\end{flalign*}
Now, for the final term $E_4$, applying H\"older's inequality, embedding $H^1(\Omega)\hookrightarrow L^6(\Omega)$ and implementing estimate \eqref{ES3} yield 
\begin{flalign*}
	E_4&= 2\int_{\Omega}\big( m_n(t)\cdot \nabla m_n(t)\big)\ m_n(t)\cdot \nabla \Delta m_n(t)\ dx +\int_{\Omega} \left(1+|m_n(t)|^2\right)\nabla m_n(t)\cdot \nabla \Delta m_n(t)\ dx&\\
	&\leq 3\ \|m_n(t)\|^2_{L^6(\Omega)}\ \|\nabla m_n(t)\|_{L^6(\Omega)} \|\nabla \Delta m_n(t)\|_{L^2(\Omega)} + \|\nabla m_n(t)\|_{L^2(\Omega)}\ \|\nabla \Delta m_n(t)\|_{L^2(\Omega)}\\
	&\leq \epsilon \ \|\nabla \Delta m_n(t)\|^2_{L^2(\Omega)} + C(\epsilon)\ \left(1+ \|m_n(t)\|^4_{H^1(\Omega)} \right)\ \|\Delta m_n(t)\|^2_{L^2(\Omega)}. 
\end{flalign*}
Substituting the estimates for $E_i$ from $i=1$ to $4$ in \eqref{REQ-1}, choosing $\epsilon = 1/8$ and combining with estimate \eqref{SS-E1}, we have
\begin{flalign*}
	&\frac{d}{dt} \left(\|m_n(t)\|^2_{L^2(\Omega)} + \|\Delta m_n(t)\|^2_{L^2(\Omega)} \right) +  \|\nabla \Delta m_n(t)\|^2_{L^2(\Omega)} \leq C\ \|u(t)\|^2_{H^1(\Omega)} \\
	&+  C\ \left(1+\|m_n(t)\|^4_{H^1(\Omega)}+\|\nabla m_n(t)\|^2_{L^2(\Omega)}\ \|\Delta m_n(t)\|^2_{L^2(\Omega)}+\|u(t)\|^2_{H^1(\Omega)} \right)  \left(\|m_n(t)\|^2_{L^2(\Omega)}+\|\Delta m_n(t)\|^2_{L^2(\Omega)} \right).
\end{flalign*}
Now, applying Gr\"onwall's inequality and implementing the estimate $\|m_n(0)\|_{L^2(\Omega)}\leq \|m_0\|_{L^2(\Omega)}$ and $\|\Delta m_n(0)\|_{L^2(\Omega)}\leq \|\Delta m_0\|_{L^2(\Omega)}$, we derive 
\begin{flalign}\label{SS-E8}
	\|m_n(t)\|^2_{L^2(\Omega)} &+ \|\Delta m_n(t)\|^2_{L^2(\Omega)} +\int_0^t \|\nabla \Delta m_n(\tau)\|^2_{L^2(\Omega)}\ d\tau \leq \left(\|m_0\|^2_{L^2(\Omega)}+\|\Delta m_0\|^2_{L^2(\Omega)}+\|u\|^2_{L^2(0,T;H^1(\Omega))}\right)\nonumber\\
	&\hspace{1cm}  \times  \exp\left\{C\left[1+\|m_n\|^2_{L^\infty(0,T;H^1(\Omega))} \|m_n\|^2_{L^2(0,T;H^2(\Omega))}+\|u\|^2_{L^2(0,T;H^1(\Omega))}\right]\right\}.
\end{flalign}
Now, applying the uniform bounds for $m_n$ in $L^\infty(0,T;H^1(\Omega))\cap L^2(0,T;H^2(\Omega))$ from estimate \eqref{SS-E5} in \eqref{SS-E8}, we find
\begin{equation}\label{SS-E9}
	\|m_n(t)\|^2_{L^2(\Omega)} + \|\Delta m_n(t)\|^2_{L^2(\Omega)} +\int_0^t \|\nabla \Delta m_n(\tau)\|^2_{L^2(\Omega)}\ d\tau \leq M(m_0,u,\Omega,T).	
\end{equation}
Therefore, as a consequence of estimates \eqref{ES4}, \eqref{ES10} and Lemma \eqref{CPP}, $\{m_n\}$ is uniformly bounded in $L^\infty(0,T;H^2(\Omega))\cap L^2(0,T;H^3(\Omega))$. Next, considering the $L^2(0,T;H^1(\Omega))$ norm of $(m_n)_t$ in equation \eqref{GA} and using the estimates from Lemma \ref{PROP2}, we find 
\begin{flalign*}
	\|(m_n)_t\|^2_{L^2(0,T;H^1(\Omega))} &\leq C\ \left(1+\|m_n\|^2_{L^\infty(0,T;H^2(\Omega))}\right)\ \left(\|m_n\|^2_{L^2(0,T;H^3(\Omega))}+ \|u\|^2_{L^2(0,T;H^1(\Omega))}\right)\\
	&\hspace{1.5cm} + C\ \|m_n\|^4_{L^\infty(0,T;H^1(\Omega))}\  \|m_n\|^2_{L^2(0,T;H^2(\Omega))}.
\end{flalign*}
Applying the uniform bounds for $\{m_n\}$ obtained from estimates \eqref{SS-E5} and \eqref{SS-E9} in the above estimate, we derive
\begin{flalign}\label{SS-E10}
	\|(m_n)_t\|&^2_{L^2(0,T;H^1(\Omega))} \leq \left(1+\|m_0\|^2_{L^2(\Omega)}+\|\Delta m_0\|^2_{L^2(\Omega)}+\|u\|^2_{L^2(0,T;H^1(\Omega))}\right)^2\nonumber\\
	&\hspace{0.5cm}  \times  \exp\left\{C \left[1+\|m_0\|^2_{H^1(\Omega)} + \|u\|^2_{L^2(0,T;H^1(\Omega))} \right]^3 \ \exp\left\{C\ \|u\|^2_{L^2(0,T;H^1(\Omega))}\right\}\right\}.
\end{flalign}
Thus, using the Aloglu weak$^*$ compactness and reflexive weak compactness theorems (Theorem 4.18, \cite{JCR}), we get	
\begin{eqnarray} \left\{\begin{array}{cccll}
		m_n &\overset{w}{\rightharpoonup} & m \  &\mbox{weakly in}& \ L^2(0,T;H^3(\Omega)),\\
		m_n &\overset{\ast}{\rightharpoonup} & m \ &\mbox{weak$^*$ in}& \ L^{\infty}(0,T;H^2(\Omega))\ \ \text{and}\\
		(m_n)_t &\overset{w}{\rightharpoonup} & m_t \ &\mbox{weakly in}& \ L^2(0,T;H^1(\Omega)) \ \ \ \mbox{as} \  \ n\to \infty. \label{wc1}
	\end{array}\right.	
\end{eqnarray}

By using the Aubin-Lions-Simon lemma (see, Corollary 4, \cite{JS}), we can obtain a sub-sequence of $\{m_n\}$ (again denoted as $\{m_n\}$) such that $m_n \to m$ strongly in $L^2(0,T;H^2(\Omega))$ and $C([0,T];H^1(\Omega))$. Finally, passing to the limit in \eqref{GA} using convergences from \eqref{wc1} (similar to \cite{SPSK}), we find that $m$ satisfies \eqref{NLP} almost everywhere. Also, taking the weak sequential lower semi-continuity in \eqref{SS-E9}, we get the required estimate \eqref{SSEE2}. Uniqueness can be demonstrated using a classical approach, similar to the one presented in \cite{SPSK}.


\noindent\textbf{Part II: (Local existence for n=3)}\ 
For $n=3$, estimates $E_2$, $E_3$ and $E_4$ will be same. The only change will be at estimating $E_1$. In fact this is the main reason why we will get only a local solution for $n=3$. 

Applying H\"older's inequality followed by the Gagliardo-Nirenberg interpolation inequality \eqref{ES9} and estimate \eqref{ES3}, we derive 
\begin{flalign*}
	&E_1= - \int_\Omega \left(\nabla m_n(t) \times \Delta m_n(t)\right) \cdot \nabla \Delta m_n(t)\ dx \leq \|\nabla m_n(t)\|_{L^6(\Omega)} \ \|\Delta m_n(t)\|_{L^3(\Omega)} \ \|\nabla \Delta m_n(t)\|_{L^2(\Omega)}&\\
	&\hspace{2cm} \leq C\  \|\Delta m_n(t)\|_{L^2(\Omega)}^{\frac{3}{2}} \ \|\nabla \Delta m_n(t)\|_{L^2(\Omega)}^{\frac{3}{2}}	\leq \epsilon \ \|\nabla \Delta m_n(t)\|_{L^2(\Omega)}^2 + C\ \|\Delta m_n(t)\|^6_{L^2(\Omega)}. 
\end{flalign*}
Substituting this estimate in \eqref{REQ-1} and combining with estimate \eqref{SS-E1}, we find
\begin{flalign*}
	&\frac{d}{dt} \left(\|m_n(t)\|^2_{L^2(\Omega)} + \|\Delta m_n(t)\|^2_{L^2(\Omega)} \right) +  \|\nabla \Delta m_n(t)\|^2_{L^2(\Omega)}  \\
	&\leq   C\ \left(1 +\|u(t)\|^2_{H^1(\Omega)} \right)  \left(1+\|m_n(t)\|^2_{L^2(\Omega)}+\|\Delta m_n(t)\|^2_{L^2(\Omega)} \right)^3.
\end{flalign*}
Therefore, by applying the classical comparison lemma and utilizing the estimate $1+\|m_n(0)\|^2_{L^2(\Omega)}+\|\Delta m_n(0)\|^2_{L^2(\Omega)} \leq 1+ \|m_0\|^2_{L^2(\Omega)}+\|\Delta m_0\|^2_{L^2(\Omega)}=:M_0$, we obtain
\begin{equation*}
	1+ \|m_n(t)\|^2_{L^2(\Omega)} + \|\Delta m_n(t)\|^2_{L^2(\Omega)} + \int_0^t \|\nabla \Delta m_n(\tau)\|^2_{L^2(\Omega)}\ d\tau < \frac{M_0}{\left(1-2M_0^2\ C(\Omega) \int_0^t \left(1 + \|u(\tau)\|^2_{H^1(\Omega)} \right) \ d\tau \right)^\frac{1}{2}},
\end{equation*} 
provided \ \ $\int_0^t \left(1 + \|u(\tau)\|^2_{H^1(\Omega)} \right) \ d\tau < \frac{1}{2M_0^2 C(\Omega)}$. If this hypothesis holds for every $t\in [0,T]$ then we can directly get a global solution. If not, then for any such control $u\in L^2(0,T;H^1(\Omega))$, there exists a time $T^*$ such that for any time $\widetilde{T}<T^*$, $\int_0^{\widetilde{T}} \left(1 + \|u(\tau)\|^2_{H^1(\Omega)}\right)  \ d\tau <  \frac{1}{2M_0^2 C(\Omega)}$. Therefore, for any fixed $\widetilde{T}<T^*$, the sequence of approximated solutions $\{m_n\}$ is uniformly bounded in $L^\infty(0,\widetilde{T};H^2(\Omega))\cap L^2(0,\widetilde{T};H^3(\Omega))$. Moreover, by considering the $L^2(0,T;H^1(\Omega))$ norm of $(m_n)_t$ in \eqref{GA} and using these bounds, we can obtain a uniform bound for $(m_n)_t$. Proceeding as we have done for $n=2$, we can find a subsequence that converges to a local regular solution of \eqref{NLP} on $[0,\widetilde{T}]$.  

\noindent \textbf{Part\ III:\ (Global existence for n=3)\:}\ 
Finally, we will demonstrate the global existence of a regular solution under the smallness assumption \eqref{SAOIC}. We will employ the same estimate derived for $E_1$ in the three-dimensional case, along with the estimates $E_2,E_3$ and $E_4$ as previously derived for $n=2$. 
Substituting these estimate in \eqref{REQ-1} and choosing $\epsilon=\frac{1}{8}$, we have
\begin{flalign*}
	\frac{d}{dt} \|\Delta m_n(t)\|&^2_{L^2(\Omega)} + \|\nabla \Delta m_n(t)\|^2_{L^2(\Omega)}  \leq   C\ \|\Delta m_n(t)\|^6_{L^2(\Omega)} + C\ \|u(t)\|^2_{H^1(\Omega)}\nonumber\\
	& + C\    \left(1+\|m_n(t)\|^4_{H^1(\Omega)}+\|u(t)\|^2_{H^1(\Omega)}  \right)\ \left( \|m_n(t) \|^2_{L^2(\Omega)}+\|\Delta m_n(t)\|^2_{L^2(\Omega)}\right).
\end{flalign*}
Now, combining with estimate \eqref{SS-E1} and using the inequality \eqref{ES7}, we obtain
\begin{flalign}\label{N3E2}
	&\frac{d}{dt} \left(\|m_n(t)\|^2_{L^2(\Omega)} + \|\Delta m_n(t)\|^2_{L^2(\Omega)} \right) +\frac{1}{2}\ \|\nabla \Delta m_n(t)\|^2_{L^2(\Omega)}+ \frac{1}{2}\   \left(1-\widetilde{C}\ \|\Delta m_n(t)\|^4_{L^2(\Omega)}\right) \|\nabla \Delta m_n(t)\|^2_{L^2(\Omega)} \nonumber\\
	&\hspace{1cm} \leq   C\ \|u(t)\|^2_{H^1(\Omega)} +C\    \left(1+\|m_n(t)\|^4_{H^1(\Omega)}+\|u(t)\|^2_{H^1(\Omega)}  \right)\ \left(\|m_n(t)\|^2_{L^2(\Omega)} + \|\Delta m_n(t)\|^2_{L^2(\Omega)} \right).
\end{flalign}
We claim that for all $t\in[0,T]$, the following inequality holds:
$$ \|m_n(t)\|^2_{L^2(\Omega)} + \|\Delta m_n(t)\|^2_{L^2(\Omega)} +\int_0^t \|\nabla \Delta m_n(t)\|^2_{L^2(\Omega)} \ dt \leq  M(m_0,u,\Omega,T).$$
From our assumption \eqref{SAOIC}, it is clear that $\|\Delta m_n(0)\|^2_{L^2(\Omega)}\leq  \|\Delta m_0\|^2_{L^2(\Omega)} <\frac{1}{\widetilde{C}^{\frac{1}{2}}}$. Since, $\|\Delta m_n(t)\|_{L^2(\Omega)}$ is continuous, there exists a time $t^*>0$ such that  
\begin{equation}\label{GEE-1}
\|\Delta m_n(t)\|^2_{L^2(\Omega)}< \frac{1}{\widetilde{C}^{\frac{1}{2}}}\ \ \ \\ \text{for all}\ t\in [0,t^*].	
\end{equation}
Let $T^*$ be the least upper bound of the set of all times $t^*\in [0,T]$ such that the  inequality \eqref{GEE-1} holds true. Therefore, by the definition of $T^*$, we have
$$1-\widetilde{C}\ \|\Delta m_n(t)\|^4_{L^2(\Omega)} \geq0 \ \ \ \text{for all}\ t\in [0,T^*].$$
Substituting this inequality into \eqref{N3E2}, and applying Gr\"onwall's inequality followed by the estimate \eqref{SS-E5}, we derive
\begin{equation}\label{GEE-2}
	\|m_n(t)\|^2_{L^2(\Omega)}+\|\Delta m_n(t)\|^2_{L^2(\Omega)} +\int_0^t\|\nabla\Delta m_n(t)\|^2_{L^2(\Omega)} \ dt\leq M(m_0,u,\Omega,T) <\frac{1}{\widetilde{C}^{\frac{1}{2}}}\ \ \ \forall\  t\in[0,T^*].
\end{equation}
If $T^*=T$, then we are done. If not, and $T^* <T$, then due to the continuity argument for $\Delta m_n$ and the previously derived estimate \eqref{GEE-2}, there exists a time $T^{\#}>T^*$ such that 
\begin{equation*}
	\|\Delta m_n(t)\|^2_{L^2(\Omega)}  <\frac{1}{\widetilde{C}^{\frac{1}{2}}}\ \ \ \forall\  t\in[0,T^\#].
\end{equation*}
This contradicts the definition of $T^*$. Therefore, we conclude that $T^*=T$.


Finally, using the uniform bounds, the weak and weak$^*$ convergences as in \eqref{wc}, and proceeding similarly to the previous case for $n=2$, we find that $m$ is a global regular solution on $[0,T]$. 
Finally, the uniqueness of regular solutions can be demonstrated in a classical manner by following the steps outlined in \cite{SPSK}.
\end{proof}

\section{Existence of Optimum and First Order Optimality Condition}\label{S-FOOC}

\subsection{Existence of Optimal Control}

The optimal control problem is meaningful only when a globally optimal solution exists. The necessity of demonstrating the existence of an optimal control becomes particularly intricate when dealing with nonlinear systems featuring non-convex cost functionals. In such scenarios, the complexities introduced by nonlinearity and non-convexity can pose formidable challenges. However, establishing the existence of an optimal control remains paramount even in these cases. This fact is established by the following theorem.

\begin{proof}[Proof of Theorem \ref{T-EOOC}]	
	We prove this theorem using the direct method of calculus of variations. Clearly, the bounds $a,b\in L^2(0,T;\mathbb{R}^N)$ for $n=2$, validates the non-emptiness of $\mathcal{A}$. Moreover, for $n=3$ we have made the assumption that $\mathbb{U}_{ad}$ is non-empty, which shows that admissible pairs $\mathcal{A}$ is non-empty. Since, the cost functional \eqref{CF-2} is bounded below, there exists real number $\alpha\geq 0$ such that $\alpha: = \inf_{(m,U) \in \mathcal{A}} \mathcal{J}(m,U)$. Furthermore, we can extract a minimizing sequence $\{(m_n,U_n)\}\subset \mathcal{A}$ such that the convergence $\mathcal{J}(m_n,U_n)\to \alpha$ as $n\to \infty$ holds true.

	
	Since, $\{(m_n,U_n)\}$ constitutes a minimizing sequence, according to the definition of the cost functional, $\{U_n\}$ is uniformly bounded in $L^2(0,T;\mathbb{R}^N)$. Therefore, we can extract a sub-sequence again denoted as $\{U_n\}$ such that $U_n \overset{w}{\rightharpoonup}\widetilde{U}$ weakly in $L^2(0,T;\mathbb{R}^N)$ for some element $\widetilde{U}\in L^2(0,T;\mathbb{R}^N)$. Moreover, as the set $\mathbb{U}_{a,b}$ is a closed and convex set, therefore it is weakly closed. So, $\widetilde{U}\in \mathbb{U}_{a,b}$ for both $n=2$ and $3$. Furthermore, as norm is weakly sequential lower semi-continuous, that is, $\|U\|_{L^2(0,T;\mathbb{R}^N)}\leq \liminf_{n \to \infty} \|U_n\|_{L^2(0,T;\mathbb{R}^N)}$, $\widetilde{U}\in \mathbb{U}_{sa}$ for $n=3$. Therefore, $\widetilde{U}\in \mathbb{U}_{ad}$ for $n=2,3$.

	
	
	Also, using estimate \eqref{PEST}, $\zetaup(U_n)$ is uniformly bounded in $L^2(0,T;H^1(\Omega))$. Utilizing the estimate \eqref{SSEE2} and \eqref{SS-E10}, we can establish a uniform bound for the sequence $\{m_n\}$ in space $\mathcal{M}$.
	Then, applying the Aubin–Lions–Simon compactness theorem, we conclude that the sequence $\{m_n\}$ is relatively compact in the space $C([0,T];H^1(\Omega)) \cap L^2(0,T;H^2(\Omega))$. As a result, there exist subsequences (again denoted as $\left\{ (m_n,U_n)\right\} \subset \mathcal{A}$) such that
	\begin{eqnarray} \left\{\begin{array}{ccccl}
			U_n &\overset{w}{\rightharpoonup} & \widetilde{U} \ &\mbox{weakly in}&  L^2(0,T;\mathbb{R}^N),\\
			\zetaup(U_n) &\overset{w}{\rightharpoonup} & \zetaup(\widetilde{U}) \ &\mbox{weakly in}&  L^2(0,T;H^1(\Omega)),\\
			m_n &\overset{w}{\rightharpoonup} & \widetilde{m} \  &\mbox{weakly in}&  L^2(0,T;H^3(\Omega)),\\
			m_n &\overset{*}{\rightharpoonup} & \widetilde{m} \  &\mbox{weak$^*$ in}&  L^\infty(0,T;H^2(\Omega)),\\
			(m_n)_t &\overset{w}{\rightharpoonup} & \widetilde{m}_t \ &\mbox{weakly in}&  L^2(0,T;H^1(\Omega))\ \text{and}\\
			m_n &\overset{s}{\to} & \widetilde{m} \  &\mbox{strongly in}& C([0,T];H^1(\Omega))\cap L^2(0,T;H^2(\Omega))\ \ \ \ \mbox{as} \  \ n\to \infty. \label{P2}
		\end{array}\right.	
	\end{eqnarray}
	Employing the convergences established in \eqref{P2}, we can confirm that $\widetilde{m}$ is a regular solution of system \eqref{NLP-EV} corresponding to the control $\widetilde{U}$, that is $(\widetilde{m},\widetilde{U})\in \mathcal{A}$. Leveraging the weak lower semicontinuity of the cost functional $\mathcal{J}$, we can deduce that
	\begin{equation}\label{I2}
		\mathcal J(\widetilde{m},\widetilde{U}) \leq \liminf_{n \to \infty}  \mathcal J(m_n,U_n)=\lim_{n \to \infty} \mathcal J(m_n,U_n)=\alpha.
	\end{equation}
	As $\alpha$ represents the infimum of the functional $\mathcal{J}$ over $\mathcal{A}$, it follows that $\alpha \leq \mathcal{J}(\widetilde{m},\widetilde{U})$. Thus, in conjunction with \eqref{I2}, we can deduce that $\displaystyle{\mathcal J(\widetilde{m},\widetilde{U}) = \alpha = \inf_{(m,U) \in \mathcal{A}} \mathcal J(m,U)}$.
	Thus the proof is completed.
\end{proof}

\subsection{Control-to-State Operator}

Let $\widetilde{m}$ be the unique regular solution of system \eqref{NLP-EV}, associated with the time dependent control $\widetilde{U}\in \mathbb{U}_R$, and initial data $m_0$ meets the condition \eqref{IC}. Now, let's analyze the following linearized system: 
\begin{equation}\label{CLE}
	(L-LLG)\begin{cases}
		\begin{array}{l}
			\mathcal{L}_{\widetilde{U}}z=f \ \ \ \ \text{in}\ \Omega_T,\\
			\frac{\partial z}{\partial \eta}=0 \ \ \ \ \ \  \text{in}\ \partial \Omega_T, \ \ \ z(x,0)=z_0\ \ \text{in}\ \Omega,
		\end{array}
	\end{cases}	
\end{equation}
where the operator $\mathcal{L}_{\widetilde{U}}$ is defined as
\begin{equation}\label{CLO}
	\mathcal{L}_{\widetilde{U}}z:= z_t-\Delta z -z\times \Delta \widetilde{m} -\widetilde{m}\times \Delta z - z \times \zetaup(\widetilde{U}) +2\ (\widetilde{m}\cdot z)\ \widetilde{m}+ \left(1+|\widetilde{m}|^2\right)z.	
\end{equation}

\begin{Lem}\label{L-SLS}
	Given any $f$ in $L^2(0,T;H^1(\Omega))$, we can find a unique regular solution $z$ in $L^2(0,T;H^3(\Omega))\cap L^{\infty}(0,T;H^2(\Omega))$ of the linearized system \eqref{CLE}. Furthermore, the subsequent estimation holds:
	\begin{align}\label{LSSE}
		\mathcal{B}(z_{\infty},z_2,{z_t}_2) \leq& \left(\|z_0\|^2_{L^2(\Omega)}+\|\Delta z_0\|^2_{L^2(\Omega)} +  \|f\|^2_{L^2(0,T;H^1(\Omega))}\right) \times \ \exp\bigg\{C\   \Big(1+\|\widetilde{m}\|^2_{L^\infty(0,T;H^2(\Omega))}\nonumber\\
		&+\|\widetilde{m}\|^2_{L^\infty(0,T;H^1(\Omega))}\|\widetilde{m}\|^2_{L^2(0,T;H^2(\Omega))}+\|\widetilde{m}\|^2_{L^2(0,T;H^3(\Omega))}+ \|\zetaup(\widetilde{U})\|^2_{L^2(0,T;H^1(\Omega))}\Big) \bigg\},
\end{align}
where $\mathcal{B}(z_{\infty},z_2,{z_t}_2):=\|z\|^2_{L^{\infty}(0,T;H^2(\Omega))}+ \|z\|^2_{L^2(0,T;H^3(\Omega))}+\|z_t\|^2_{L^2(0,T;H^1(\Omega))}$.
\end{Lem}
\begin{proof}
By employing the Faedo-Galerkin approximation technique, we are able to demonstrate the existence and uniqueness of the system \eqref{CLE}. However, in this context, we will focus solely on presenting a priori estimates for the solution. By considering the $L^2$ inner product of \eqref{CLE} with the function $z$, we have 
\begin{equation}\label{LSE1}
	\frac{d}{dt} \|z(t)\|^2_{L^2(\Omega)} +  \|\nabla z(t)\|^2_{L^2(\Omega)} \leq C\  \left(1+  \|\widetilde{m}(t)\|^2_{H^2(\Omega)} \right)\|z(t)\|^2_{H^1(\Omega)} + \|f(t)\|^2_{L^2(\Omega)}.	
\end{equation}
By calculating the gradient of the system \eqref{CLE} and subsequently considering the inner product with $-\nabla \Delta z$, we arrive at the following modification:
\begin{eqnarray}\label{LSH2E}
	\lefteqn{\frac{1}{2} \frac{d}{dt} \|\Delta z(t)\|^2_{L^2(\Omega)} +  \int_\Omega |\nabla \Delta z(t)|^2 dx = - \int_\Omega \nabla \big(z \times \Delta \widetilde{m}\big) \cdot \nabla \Delta z\ dx }\nonumber\\
	&&- \int_\Omega \nabla \big( \widetilde{m} \times \Delta z\big)  \cdot \nabla \Delta z\ dx  -  \int_\Omega \nabla \big(z \times \zetaup(\widetilde{U})\big)\cdot \nabla \Delta z\ dx +\ 2\  \int_\Omega \nabla \big( \big(\widetilde{m}\cdot z\big)\widetilde{m}\big) \cdot  \nabla \Delta z\ dx \nonumber\\
	&& + \int_\Omega \nabla \left( \left(1+|\widetilde{m}|^2\right)z\right)\cdot \nabla \Delta z\ dx - \int_\Omega \nabla f\cdot  \nabla \Delta z\ dx:= \sum_{i=1}^{6}\Gamma_i. \ \ \ \ \ \ \ 
\end{eqnarray}

Let's assess the terms on the right-hand side. By employing H\"older's inequality and utilizing the embeddings $H^1(\Omega) \hookrightarrow L^p(\Omega)$ for $p \in [1,6]$ and $H^2(\Omega) \hookrightarrow L^\infty(\Omega)$, we obtain
\begin{flalign*}
	\Gamma_1 &=- \int_\Omega \left[ \nabla z\times \Delta \widetilde{m} + z \times \nabla \Delta \widetilde{m} \right]\cdot \nabla \Delta z\ dx&\\
	&\leq \ \ \|\nabla z(t)\|_{L^4(\Omega)} \|\Delta \widetilde{m}(t)\|_{L^4(\Omega)}   \|\nabla \Delta z(t)\|_{L^2(\Omega)} + \|z(t)\|_{L^\infty(\Omega)} \|\nabla \Delta \widetilde{m}\|_{L^2(\Omega)} \|\nabla \Delta z\|_{L^2(\Omega)}\\
	&\leq \epsilon   \int_\Omega |\nabla \Delta z(t)|^2 dx +\ C(\epsilon)   \ \|\widetilde{m}(t)\|^2_{H^3(\Omega)} \ \|z(t)\|^2_{H^2(\Omega)},
\end{flalign*}
\begin{flalign*}
	\Gamma_5 &= \int_\Omega  \big[\nabla z+2\ \big(\widetilde{m}\cdot \nabla \widetilde{m}\big)z+ |\widetilde{m}|^2 \nabla z \big]\cdot \nabla \Delta z\ dx&\\
	&\leq \left[ \|\nabla z(t)\|_{L^2(\Omega)} + 2\  \|\widetilde{m}(t)\|_{L^6(\Omega)}\ \|\nabla \widetilde{m}(t)\|_{L^6(\Omega)}\ \|z(t)\|_{L^6(\Omega)}  + \|\widetilde{m}(t)\|^2_{L^6(\Omega)} \|\nabla z(t)\|_{L^6(\Omega)} \right] \ \|\nabla \Delta z(t)\|_{L^2(\Omega)}\\
	&\leq \epsilon   \int_\Omega |\nabla \Delta z(t)|^2 dx +\ C(\epsilon)   \ \left[1+ \|\widetilde{m}(t)\|^2_{H^1(\Omega)} \ \|\widetilde{m}(t)\|^2_{H^2(\Omega)} \right]  \ \|z(t)\|^2_{H^2(\Omega)}.
\end{flalign*}
We can obtain the estimate for $\Gamma_2, \Gamma_3, \Gamma_4$ and $\Gamma_6$ in a similar way. Now, substituting all these estimates in equation \eqref{LSH2E}, choosing $\epsilon=\frac{1}{12}$ and adding with \eqref{LSE1}, we get
\begin{eqnarray*}
	\lefteqn{\frac{d}{dt} \left(\|z(t)\|^2_{L^2(\Omega)}+\|\Delta z(t)\|^2_{L^2(\Omega)}\right)+  \|\nabla z(t)\|^2_{L^2(\Omega)}+\|\nabla \Delta z(t)\|^2_{L^2(\Omega)} } \nonumber\\
	&&\leq C\ \bigg[1+ \|\widetilde{m}(t)\|^2_{H^1(\Omega)}\|\widetilde{m}(t)\|^2_{H^2(\Omega)} + \|\widetilde{m}(t)\|^2_{H^3(\Omega)}  + \|\zetaup(\widetilde{U})(t)\|^2_{H^1(\Omega)}\bigg]\ \|z(t)\|^2_{H^2(\Omega)} + C\   \|f(t)\|^2_{H^1(\Omega)}.
\end{eqnarray*}
By utilizing the inequality $\|z\|_{H^2(\Omega)} \leq C \left(\|z\|_{L^2(\Omega)} + \|\Delta z\|_{L^2(\Omega)}\right)$ stated in Lemma \ref{EN}, and subsequently applying Gronwall's inequality, we find 
\begin{equation}\label{LSE3}
    \|z(t)\|^2_{L^2(\Omega)}+\|\Delta z(t)\|^2_{L^2(\Omega)} + \int_0^T \left(\|\nabla z(s)\|^2_{L^2(\Omega)}+\|\nabla \Delta z(s)\|^2_{L^2(\Omega)}\right)\ ds 
	\leq C_1(\Omega,T,z_0,f,\widetilde{m},\widetilde{u}),	
\end{equation}
\begin{flalign*}
\text{where}\ & C_1(\Omega,T,z_0,f,\widetilde{m},\widetilde{u}) = \left(\|z_0\|^2_{L^2(\Omega)}+\|\Delta z_0\|^2_{L^2(\Omega)} +  \|f\|^2_{L^2(0,T;H^1(\Omega))}\right)&\\
&\hspace{1.5cm}\times \ \exp\bigg\{C  \left(1+ \|\widetilde{m}\|^2_{L^\infty(0,T;H^1(\Omega))}\|\widetilde{m}\|^2_{L^2(0,T;H^2(\Omega))}+\|\widetilde{m}\|^2_{L^2(0,T;H^3(\Omega))}+ \|\zetaup(\widetilde{U})\|^2_{L^2(0,T;H^1(\Omega))}\right) \bigg\}.	
\end{flalign*}
Moreover, estimating the $L^2(0,T;H^1(\Omega))$ norm of $z_t$ in equation \eqref{CLE} using the inequalities from Lemma \ref{PROP2}, we derive
\begin{align}\label{LSE4}
	&\|z_t\|_{L^2(0,T;H^1(\Omega))} \leq  \left(1+\|\widetilde{m}\|^2_{L^\infty(0,T;H^2(\Omega))}\right)\|z\|^2_{L^2(0,T;H^3(\Omega))}\nonumber\\
	&\hspace{0.5cm} +\left(\|\widetilde{m }\|^2_{L^2(0,T;H^3(\Omega))}+\|\zetaup(\widetilde{U})\|^2_{L^2(0,T;H^1(\Omega))}+\|\widetilde{m}\|^2_{L^\infty(0,T;H^1(\Omega))} \|\widetilde{m}\|^2_{L^2(0,T;H^2(\Omega))}\right)\|z\|^2_{L^\infty(0,T;H^2(\Omega))}.
\end{align}
Finally, substituting the $\|z\|_{L^2(0,T;H^3(\Omega))}$ and $\|z\|_{L^\infty(0,T;H^2(\Omega))}$ norm estimates from \eqref{LSE3} in estimate \eqref{LSE4}, and adding the resultant again with \eqref{LSE3}, we obtain our required estimate \eqref{LSSE}.
Hence the proof.
\end{proof}


\begin{Lem}{(Lipschitz Continuity of G)}\label{L-LCCTS}	The control-to-state operator $G:\mathbb{U}_R \to \mathcal{M}$ exhibits Lipschitz continuity. In other words, there exists a positive constant $C_2$, which depends on $\Omega$, $T$, $R$, and $m_0$, such that the following holds:
	\begin{equation}\label{LCCTSO}
		\|G(U_1)-G(U_2)\|^2_{\mathcal{M}} \leq C_2\ \|U_1-U_2\|^2_{L^2(0,T;\mathbb{R}^N)}, \ \ \ \ \ \ \forall \ U_1,U_2 \in \mathbb{U}_R.
	\end{equation}	
\end{Lem}

\begin{proof}
	
	Assume $m_1$ and $m_2$ are two regular solutions of system \eqref{NLP-EV} corresponding to the controls $U_1$ and $U_2$, respectively. Define $\shat{m} := m_1 - m_2$ and $\shat{U} := U_1 - U_2$. The pair $(\shat{m}, \shat{U})$ satisfies the following system 
\begin{equation}\label{EPD}
	\begin{cases}
		\displaystyle \shat{m}_t- \Delta \shat{m}= \shat{m} \times \Delta m_1 + m_2 \times \Delta \shat{m}+ \shat{m}\times \zetaup(U_1)+ m_2 \times \zetaup(\shat{U})\\
		\hspace{2.5cm} -\ \big(\shat{m}\cdot (m_1+m_2)\big)m_1-\big(1+|m_2|^2\big)\shat{m} +\zetaup(\shat{U})  \ \ \ \ \ \ (x,t)\in \Omega_T,\\
		\frac{\partial \shat{m}}{\partial \eta}=0 \ \ \ \ \ (x,t) \in  \partial\Omega_T,\ \ \ \ \ \ \ \ \ \ \ 
		\shat{m}(\cdot,0)=0 \ \ \ \ \ \text{in} \ \Omega.
	\end{cases}	
\end{equation}
If we consider the $L^2$ inner product of system \eqref{EPD} with $\shat{m}$ and employing $(\shat{m}\times \Delta m_1)\cdot \shat{m}=0$, $(m_2\times \Delta \shat{m}) \cdot \shat{m}= -\nabla \shat{m}\cdot (\shat{m}\times \nabla m_2)$, we obtain
\begin{align}\label{TE-1}
	\frac{d}{dt}\|\shat{m}(t)\|^2_{L^2(\Omega)} + \|\nabla \shat{m}(t)\|^2_{L^2(\Omega)} \leq C\left(1+ \|m_1\|^2_{H^1(\Omega)}+\|m_2\|^2_{H^2(\Omega)}\right)\ \|\shat{m}\|^2_{H^1(\Omega)}+ \|\zetaup(\shat{U})\|^2_{L^2(\Omega)}
\end{align}
By calculating the gradient of equation \eqref{EPD} and subsequently taking the $L^2$ inner product with $\nabla \Delta \hat{m}$, while performing the estimates in a manner analogous to those in Lemma \ref{L-SLS}, we derive 
\begin{flalign}\label{TE-2}
	&\frac{d}{dt} \|\Delta \shat{m}(t)\|^2_{L^2(\Omega)} + \|\nabla \Delta \shat{m}(t)\|^2_{L^2(\Omega)} \leq C\ \left(1+\|m_2(t)\|^2_{H^2(\Omega)}\right)\|\zetaup(\shat{U})(t)\|^2_{H^1(\Omega)} +C\ \bigg[1+\|m_1(t)\|^2_{H^3(\Omega)} \nonumber\\	
	&\hspace{1.5cm} + \|m_2(t)\|^2_{H^3(\Omega)} + \|\zetaup(U_1)(t)\|^2_{H^1(\Omega)} +\|m_1(t)\|^4_{H^2(\Omega)} + \|m_2(t)\|^4_{H^2(\Omega)} \bigg] \ \|\shat{m}(t)\|^2_{H^2(\Omega)}.
\end{flalign}
Now, we will combine the estimates \eqref{TE-1} and \eqref{TE-2}, and then proceed to apply Gr\"onwall's inequality. This leads to the following estimate:
\begin{flalign}\label{TE-3}
	&\|\shat{m}\|^2_{L^\infty(0,T;H^2(\Omega))}+\|\shat{m}\|^2_{L^2(0,T;H^3(\Omega))} \leq   \left(1+\|m_2\|^2_{L^\infty(0,T;H^2(\Omega))}\right)\|\zetaup(\shat{U})\|^2_{L^2(0,T;H^1(\Omega))} \exp\left\{C\left[T+ \|m_1\|^2_{L^2(0,T;H^3(\Omega))}\right.\right.\nonumber\\
	&\hspace{1.2cm} \left.\left.+\|m_2\|^2_{L^2(0,T;H^3(\Omega))}+\  \|\zetaup(U_1)\|^2_{L^2(0,T;H^1(\Omega))}+T\ \|m_1\|^4_{L^\infty(0,T;H^2(\Omega))} + T\ \|m_2\|^4_{L^\infty(0,T;H^2(\Omega))}\right] \right\}.
\end{flalign}
Furthermore, we evaluate the $L^2(0,T;H^1(\Omega))$ norm of $\shat{m}_t$ in equation \eqref{EPD}. We then substitute the estimates for $\|\shat{m}\|^2_{L^\infty(0,T;H^2(\Omega))}$ and $\|\shat{m}\|^2_{L^2(0,T;H^3(\Omega))}$ from estimate \eqref{TE-3}. Additionally, incorporating the bounds for $\|m_k\|_{\mathcal{M}}$ for $k=1,2$ from the estimate \eqref{SSEE2}, along with the initial condition $\hat{m}(0)=m_1(0)-m_2(0)=0$, we derive the desired conclusion \eqref{LCCTSO}. This completes the proof.
\end{proof}

\begin{Pro}\label{P-CTS}
	Given any control $\bar{U}$ in the control set $\mathbb{U}_R$, let $m_{\bar{U}}$ denote the corresponding regular solution of system \eqref{NLP-EV}.
	Then the following statements hold:
	\begin{enumerate}[label=(\roman*)]
		\item The control-to-state mapping $G$ is Fr\'echet differentiable on $\mathbb{U}_R$. This means that for any $\bar{U}\in \mathbb{U}_R$, there exists a bounded linear operator $G'(\bar{U}):L^2(0,T;\mathbb{R}^N) \to \mathcal{M}$ with $G'(\bar{U})[U]:=z$, such that 
		$$\frac{\|G(\bar{U}+U)-G(\bar{U})-G'(\bar{U})[U]\|_{\mathcal{M}}}{\|U\|_{L^2(0,T;\mathbb{R}^N)}}\to 0 \ \ \ \ \ \text{as} \ \ \|U\|_{L^2(0,T;\mathbb{R}^N)}\to 0,$$
		where $z$ is the unique regular solution of the following linearized system:
		\begin{equation}\label{LS2}
			\ \ \ \ \ \ \ \ \ \ \ \begin{cases}
				\mathcal{L}_{\bar{U}}z= \zetaup(U)+ m_{\bar{U}} \times \zetaup(U)  \ \  \ \text{in} \ \ \Omega_T,\\
				\frac{\partial z}{\partial \eta}=0 \ \ \ \ \text{on} \ \partial \Omega_T,\\
				z(0)=0 \ \ \text{in} \ \Omega.	
			\end{cases}
		\end{equation}
		\item The Fr\'echet derivative $G'$ is Lipschitz continuous on $\mathbb{U}_R$, that is, for any controls $U_1,U_2 \in \mathbb{U}_R$ and $U\in L^2(0,T;\mathbb{R}^N)$, there exists a constant $C_3>0$ depending on $\Omega,T,R,m_0$ such that 
		\begin{equation}\label{LC-CTS}
		\|G'(U_1)[U]-G'(U_2)[U] \|_{\mathcal{M}} \leq C_3\  \|U_1-U_2\|_{L^2(0,T;\mathbb{R}^N)} \ \|U\|_{L^2(0,T;\mathbb{R}^N)}. 	
		\end{equation}
\end{enumerate}
\end{Pro}
\begin{proof}
	 Clearly, for any control $U\in L^2(0,T;\mathbb{R}^N)$, the terms $\zetaup(U), m_{\bar{U}}\times \zetaup(U) \in L^2(0,T;H^1(\Omega))$. Consequently, according to Lemma \ref{L-SLS}, there exists a unique regular solution $z \in \mathcal{M}$ for system \eqref{LS2}. This formalizes the clarification, making it clear that $z$ is just a notation and not a derivative yet. If $m_{\bar{U}+U}$ be the state associated to the control $\bar{U}+U$, then $w:=m_{\bar{U}+U}-m_{\bar{U}}-z$ satisfies the subsequent system:
	\begin{equation}\label{LOCTS}
	\begin{cases}
		\displaystyle \mathcal{L}_{\bar{U}}w=  \shat{m} \times \Delta \shat{m} +\shat{m}\times \zetaup(U)-|\shat{m}|^2\shat{m}-|\shat{m}|^2m_{\bar{U}} -2 (\shat{m}\cdot m_{\bar{U}})\shat{m} \ \ \ \ \ \text{in}\ \Omega_T,\\
		\frac{\partial w}{\partial \eta}=0 \ \ \ \ \text{on} \ \partial \Omega_T,  \ \ \ \ w(0)=0 \ \ \text{in} \ \Omega,	
	\end{cases}
\end{equation}
where $\mathcal{L}_{\bar{U}}$ is defined in \eqref{CLO}. Applying estimates \eqref{EE-2} and \eqref{EE-3} for the first two terms on the right hand side of equation \eqref{LOCTS}, we find
\begin{flalign*}
	\|\shat{m}\times \Delta \shat{m}\|^2_{L^2(0,T;H^1(\Omega))} &\leq C\ \|\shat{m}\|^2_{L^\infty(0,T;H^2(\Omega))} \ \|\shat{m}\|^2_{L^2(0,T;H^3(\Omega))},\\
	\|\shat{m}\times \zetaup(U)\|^2_{L^2(0,T;H^1(\Omega))} &\leq C\ \|\shat{m}\|^2_{L^\infty(0,T;H^2(\Omega))}\ \|\zetaup(U)\|^2_{L^2(0,T;H^1(\Omega))}.
\end{flalign*}
Now, by employing estimate \eqref{EE-6} for the last three terms of equation \eqref{LOCTS}, we obtain
\begin{align*}
		&\left\| |\shat{m}|^2\shat{m}\right\|^2_{L^2(0,T;H^1(\Omega))} +\left\| |\shat{m}|^2m_{\bar{U}}\right\|^2_{L^2(0,T;H^1(\Omega))} + \big\|2(\shat{m}\cdot m_{\bar{U}})\shat{m}\big\|^2_{L^2(0,T;H^1(\Omega))} \\
		&\hspace{1.5cm} \leq C\ \left(\|m_{\bar{U}}\|^2_{L^2(0,T;H^2(\Omega))} +\|\shat{m}\|^2_{L^2(0,T;H^2(\Omega))}\right)\|\shat{m}\|^4_{L^\infty(0,T;H^2(\Omega))}. 
\end{align*}
Utilizing all of these estimates and invoking Lemma \ref{L-SLS}, we can establish the subsequent estimate:
\begin{eqnarray*}
	\lefteqn{\mathcal{B}(w_{\infty},w_2,{w_t}_2) \leq \bigg[\ \left( \|\shat{m}\|^2_{L^2(0,T;H^3(\Omega))}+\|\zetaup(U)\|^2_{L^2(0,T;H^1(\Omega))}\right)\|\shat{m}\|^2_{L^\infty(0,T;H^2(\Omega))}}\nonumber\\
	&&\hspace{1cm} +\ \left(\|m_{\bar{U}}\|^2_{L^2(0,T;H^2(\Omega))} +\|\shat{m}\|^2_{L^2(0,T;H^2(\Omega))}\right)\|\shat{m}\|^4_{L^\infty(0,T;H^2(\Omega))}\bigg]\times \ \exp\bigg\{C  \left(1+\|m_{\bar{U}}\|^2_{L^\infty(0,T;H^2(\Omega))} \right.\nonumber\\
	&&\hspace{1cm}\left.+\ \|m_{\bar{U}}\|^2_{L^\infty(0,T;H^1(\Omega))} \|m_{\bar{U}}\|^2_{L^2(0,T;H^2(\Omega))}+\|m_{\bar{U}}\|^2_{L^2(0,T;H^3(\Omega))}+ \|\zetaup(\bar{U})\|^2_{L^2(0,T;H^1(\Omega))}\right) \bigg\}.\hspace{1.5cm}
\end{eqnarray*}
Using the estimate \eqref{SSEE2}, we can establish a bound for $\|m_{\bar{U}}\|_{\mathcal{M}}$ in terms of $\Omega,T, m_0,\bar{U}$. By leveraging the Lipschitz continuity estimate for the control-to-state operator \eqref{LCCTSO} and \eqref{PEST}, we obtain
\begin{equation}\label{X1}
\mathcal{B}(w_{\infty},w_2,{w_t}_2)  \leq C\ \|U\|^4_{L^2(0,T;\mathbb{R}^N)}+C\ \|U\|^6_{L^2(0,T;\mathbb{R}^N)}.	
\end{equation}
Hence, $\frac{\|w\|_{\mathcal{M}}}{\|U\|_{L^2(0,T;\mathbb{R}^N)}}\to 0 $ as $\|U\|_{L^2(0,T;\mathbb{R}^N)} \to 0$. This implies the Fr\`echet differentiability of the control-to-state operator and confirms the equality $z=G'(\bar{U})[U]$. Consequently, the proof of (i) is concluded. 

Next, we will show the Lipschitz continuity of the Fr\'echet derivative of the control-to-state operator. Suppose $U_1,U_2$ be any two control variables in $\mathbb{U}_R$ and $U\in L^2(0,T;\mathbb{R}^N)$. Consider $z_{U_1}:=G'(U_1)[U]$ and $z_{U_2}:=G'(U_2)[U]$ as two distinct regular solutions of the linearized system \eqref{LS2}. If we denote  $\shat{z}=z_{U_1}-z_{U_2}$, $\shat{m}=m_{U_1}-m_{U_2}$ and $\shat{U}=U_1-U_2$, then the pair $(\shat{z},\shat{U})$ satisfies the subsequent system:
	\begin{equation*}
	\begin{cases}
		\displaystyle \mathcal{L}_{U_1}\shat{z}= \sum_{k=1}^{3} \Theta_k,\\
		\frac{\partial \shat{z}}{\partial \eta}=0 \ \ \ \ \text{on} \ \partial \Omega_T,  \ \ \ \ \shat{z}(0)=0 \ \ \text{in} \ \Omega,	
	\end{cases}
\end{equation*}
where $\mathcal{L}_{U_1}$ is defined in \eqref{CLO} and the terms $\Theta_k$'s are given by
$$
\begin{array}{llll}
	\Theta_1= z_{U_2} \times \Delta \shat{m} + \shat{m}\times \Delta z_{U_2}, \ \ \ \  \Theta_2= z_{U_2} \times \zetaup(\shat{U})+\shat{m} \times \zetaup(U),\\
	\Theta_3= -2(\shat{m} \cdot z_{U_2})m_{U_1} -2(m_{U_2}\cdot z_{U_2}) \shat{m} -\big(\shat{m}\cdot (m_{U_1}+m_{U_2})\big)z_{U_2}.
\end{array}
$$
Again by employing Lemma \ref{L-SLS} and utilizing the initial condition $z(0)=0$, we obtain 
\begin{align}\label{DCTS-E4}
	\mathcal{B}(\shat{z}_{\infty}&,\shat{z}_2,{\shat{z}_t}_2) \leq \sum_{k=1}^3   \|\Theta_k\|^2_{L^2(0,T;H^1(\Omega))} \times \exp\bigg\{C  \left(1+\|m_{U_1}\|^2_{L^\infty(0,T;H^1(\Omega))} \|m_{U_1}\|^2_{L^2(0,T;H^2(\Omega))}   \right.\nonumber\\
	&\hspace{2cm}\left.+\ \|m_{U_1}\|^2_{L^\infty(0,T;H^2(\Omega))}+\|m_{U_1}\|^2_{L^2(0,T;H^3(\Omega))}+ \|\zetaup(U_1)\|^2_{L^2(0,T;H^1(\Omega))}\right) \bigg\}.\ \ \ \ \ \ \ \ \ \  
\end{align}

Now, let's proceed to estimate the terms containing $\Theta_k$'s on the right-hand side. By applying estimates \eqref{EE-2} and \eqref{EE-3} for the terms involving $\Theta_1$ and $\Theta_2$ respectively, we obtain
\begin{align*}
	\|\Theta_1\|^2_{L^2(0,T;H^1(\Omega))} &\leq C\ \|z_{U_2}\|^2_{L^\infty(0,T;H^2(\Omega))}  \|\shat{m}\|^2_{L^2(0,T;H^3(\Omega))} + C\ \|\shat{m}\|^2_{L^\infty(0,T;H^2(\Omega))}  \|z_{U_2}\|^2_{L^2(0,T;H^3(\Omega))},\\
	\|\Theta_2\|^2_{L^2(0,T;H^1(\Omega))} &\leq C\ \|z_{U_2}\|^2_{L^\infty(0,T;H^2(\Omega))}\ \|\zetaup(\shat{U})\|^2_{L^2(0,T;H^1(\Omega))}+C\ \|\shat{m}\|^2_{L^\infty(0,T;H^2(\Omega))}\ \|\zetaup(\shat{U})\|^2_{L^2(0,T;H^1(\Omega))}.
\end{align*}
Likewise, utilizing estimate \eqref{EE-6} for the term $\|\Theta_3\|^2_{L^2(0,T;H^1(\Omega))}$, we deduce
\begin{align*}
	\|\Theta_3\|^2_{L^2(0,T;H^1(\Omega))} \leq C\  \left(\|m_{U_1}\|^2_{L^2(0,T;H^2(\Omega))}+\|m_{U_2}\|^2_{L^2(0,T;H^2(\Omega))}\right) \|\shat{m}\|^2_{L^\infty(0,T;H^2(\Omega))}\|z_{U_2}\|^2_{L^\infty(0,T;H^2(\Omega))}.
\end{align*}

 Given that for each $k=1,2$, the control $U_k\in \mathbb{U}_R$ and $m_{U_k}$ constitutes a regular solution of system \eqref{NLP-EV} corresponding to the control $U_k$. Therefore, implementing estimate \eqref{SSEE2}, we can find the bound as $\|m_{U_k}\|_{\mathcal{M}}\leq C(\Omega,T,R,m_0)$. By incorporating these bounds in the energy estimate for $z_{U_2}$, we deduce from \eqref{FDCTS-S} that $\|z_{U_2}\|_{\mathcal{M}}\leq C\ \|U\|_{L^2(0,T;\mathbb{R}^N)}$. Finally, we combine all the estimates for $\|\Theta_k\|^2_{L^2(0,T;H^1(\Omega))}$ for $k=1,2,3$ and substitute in \eqref{DCTS-E4}. Furthermore, making use of the Lipschitz continuity property of the control-to-state operator from estimate \eqref{LCCTSO} and boundedness of the operator $\zetaup$ from estimate \eqref{PEST}, we arrive at
$$\mathcal{B}(\shat{z}_{\infty},\shat{z}_2,{\shat{z}_t}_2) \leq C(\Omega,T,R)\ \|\shat{U}\|^2_{L^2(0,T;\mathbb{R}^N)} \ \|U\|^2_{L^2(0,T;\mathbb{R}^N)}.$$
This completes the proof.
\end{proof}

\begin{Cor}\label{C-CTSD}
The Fr\'echet derivative of the control-to-state operator satisfies the following inequality:
\begin{equation}\label{FDCTS-S}
	\|G'(U)[V]\|_{\mathcal{M}} \leq C(\Omega,T, U,m_0)\ \|V\|_{L^2(0,T;\mathbb{R}^N)}.
\end{equation} 	
\end{Cor}
\noindent The proof of Corollary \ref{C-CTSD} is a direct consequence of Lemma \ref{L-SLS} with $z(0)=0$ and $f=\zetaup(V)+m_{\bar{U}}\times \zetaup(V)$.

\subsection{Solvability of Adjoint System:}\label{SS-SOAS}

The adjoint problem is indispensable for deriving optimality conditions as it systematically incorporates constraints and facilitates the formulation of necessary conditions through Lagrange multipliers. 

Assume $(\widetilde{m},\widetilde{U})$ represents an admissible pair for the control problem (OCP). Now, consider the following adjoint system: 
\begin{equation}\label{CLAS}
	(AL-LLG)\begin{cases}
		\begin{array}{l}
			\mathcal{E}_{\widetilde{U}}\phi=g \ \ \ \ \text{in}\ \Omega_T,\\
			\frac{\partial \phi}{\partial \eta}=0 \ \ \ \ \ \  \text{in}\ \partial \Omega_T, \ \ \ \phi(x,T)=\phi_T\ \ \text{in}\ \Omega,
		\end{array}
	\end{cases}	
\end{equation}
where the operator $\mathcal{E}_{\widetilde{U}}$ is defined as
\begin{equation*}\label{CLAO}
	\mathcal{E}_{\widetilde{U}}\phi:= \phi_t + \Delta \phi  + \Delta (\phi \times \widetilde{m})+(\Delta  \widetilde{m}\times \phi)-(\phi \times \zetaup(\widetilde{U})) -\left(1+|\widetilde{m}|^2\right)\phi -2 \big(\widetilde{m} \cdot \phi \big) \widetilde{m}.	
\end{equation*}

The weak formulation of system \eqref{CLAS} follows from the same notion given in Definition \ref{AWSD} by replacing the right-hand side term with $\int_0^T  \langle g(t),\vartheta(t)																																																																																																																																																																																																																																																																																																																																																																																																																																																																																																																																																																																																																																																																																																																									\rangle_{H^1(\Omega)^*\times H^1(\Omega)}\ dt$, for any $g\in L^2(0,T;H^1(\Omega)^*)$.

\begin{Lem}\label{AL-SLS}
	For any $g\in L^2(0,T;H^1(\Omega)^*)$ there exists a unique weak solution $\phi \in \mathcal{Z}$ of the linear system \eqref{CLAS} in the sense of Definition \ref{AWSD}. Moreover, the following estimate holds:
	\begin{equation*}
		\mathcal{D}(\phi_{\infty},\phi_2,{\phi_t}_2) \leq \left(\|\phi_T\|^2_{L^2(\Omega)}+  \|g\|^2_{L^2(0,T;H^1(\Omega)^*)}\right) \ \mathcal{K}(\widetilde{m},\widetilde{U}),
	\end{equation*}
	where $\mathcal{D}(\phi_{\infty},\phi_2,{\phi_t}_2):=\|\phi\|^2_{L^{\infty}(0,T;L^2(\Omega))}+ \|\phi\|^2_{L^2(0,T;H^1(\Omega))} +\|\phi_t\|^2_{L^2(0,T;H^1(\Omega)^*)}$ and 
	
	$\mathcal{K}(\widetilde{m},\widetilde{U}):= \exp\bigg\{C  \left(1+ \|\widetilde{m}\|^4_{L^\infty(0,T;H^1(\Omega))}+\|\widetilde{m}\|^2_{L^\infty(0,T;H^2(\Omega))}+\|\widetilde{m}\|^2_{L^2(0,T;H^3(\Omega))}+\|\zetaup(\widetilde{U})\|^2_{L^2(0,T;H^1(\Omega))} \right) \bigg\}$.
\end{Lem}
\begin{proof}
Consider the identical orthonormal basis $\{w_j\}_{j=1}^{\infty}$ established in Theorem \ref{T-SS}. Through the application of solvability of ordinary differential equations, we can ascertain a solution $\phi_n=\sum_{j=1}^{n}g_{jn}(t)w_j$ for each $j=1,\cdots,n$ of the following approximated system:
\begin{equation}\label{APGS}
	\begin{cases}
		-\big(\phi_n'(t),w_j\big)+\big(\nabla\phi_n(t),\nabla w_j\big)= -\ \big(\nabla  (\phi_n(t) \times \widetilde{m}(t)),\nabla w_j\big)+\big(\Delta \widetilde{m}(t)\times \phi_n(t),w_j\big)\\
		\ \ \ \ \  -\ \big(\phi_n(t) \times  \zetaup(\widetilde{U})(t) ,w_j\big)-\ \left(\left(1+|\widetilde{m}(t)|^2\right)\phi_n(t),w_j\right) -2 \left( \big(\widetilde{m}(t)\cdot \phi_n(t)\big)\widetilde{m}(t),w_j\right) -\ \langle g(t), w_j \rangle,\ \ \ \ \ \ \ \  \\
		\phi_n(T)=\mathbb{P}_n\big(\phi_T\big).
	\end{cases}
\end{equation}
By performing multiplication on \eqref{APGS} with the factor $g_{jn}(t)$, summing the resulting expression over values of $j$ ranging from $1$ to $n$, and applying the property that $a\cdot (a \times b)=0$, we derive
\begin{align*}
	&-\frac{1}{2}\frac{d}{dt}\|\phi_n(t)\|^2_{L^2(\Omega)}+ \int_\Omega |\nabla \phi_n(t)|^2 dx+\int_{\Omega} \left(1+|\widetilde{m}(t)|^2\right) |\phi_n(t)|^2 + 2 \int_\Omega \big| \widetilde{m}(t)\cdot \phi_n(t)\big|^2 \ dx  \\
	&\hspace{3.5cm} = -\int_\Omega \nabla\big(\phi_n(t) \times \widetilde{m}(t)\big)\cdot \nabla\phi_n(t) \ dx - \langle g(t),\phi_n(t) \rangle.
\end{align*}	 
Utilizing H\"older's inequality along with the embeddings $H^1(\Omega)\hookrightarrow L^4(\Omega)$ and $H^2(\Omega) \hookrightarrow L^\infty(\Omega)$, we arrive at
\begin{equation*}
	-\frac{d}{dt}\|\phi_n(t)\|^2_{L^2(\Omega)}+ \int_\Omega |\nabla \phi_n(t)|^2 dx\leq C\  \|g(t)\|^2_{H^1(\Omega)^*} + \ C\ \left(1+\|\widetilde{m}(t)\|^2_{H^3(\Omega)}\right) \|\phi_n(t)\|^2_{L^2(\Omega)}.
\end{equation*}
Through the utilization of Gronwall's inequality and making use of the inequality $\|\phi_n(T)\|_{L^2(\Omega)}\leq \|\phi_T\|_{L^2(\Omega)}$, we deduce the following inequality on $[0,T]$:
\begin{equation}\label{AEE1}
	\|\phi_n(t)\|^2_{L^2(\Omega)}+\int_t^T\int_\Omega |\nabla \phi_n(\tau)|^2 dx\ d\tau \leq  \Big(\|\phi_T\|^2_{L^2(\Omega)} +  \|g\|^2_{L^2(0,T;H^1(\Omega)^*)}  \Big) \ \exp{\left\{C \ \left(1+ \int_0^T \|\widetilde{m}(\tau)\|^2_{H^3(\Omega)} d\tau\right) \right\} }.
\end{equation}
Thus, the sequence $\big\{\phi_n\big\}$ is uniformly bounded in $L^\infty(0,T;L^2(\Omega)) \cap L^2(0,T;H^1(\Omega))$.

Next, proceeding in a similar way to Theorem 5.5 of \cite{SPSK}, we can obtain a uniform bound for the sequence $\big\{\phi_n'\big\}$ in the space $L^2(0,T;H^1(\Omega)^*)$ as follows:
\begin{align}
	\int_0^T \|\phi_n'(t)\|^2_{H^1(\Omega)^*} dt &\leq C\  \left[  \left(1+ \|\widetilde{m}\|^4_{L^\infty(0,T;H^1(\Omega))} + \|\widetilde{m}\|^2_{L^\infty(0,T;H^2(\Omega))}\right) \|\phi_n\|^2_{L^2(0,T;H^1(\Omega))} \right.&\nonumber\\
	& \left.\ \ \ \ \  +\   \|\zetaup(\widetilde{U})\|^2_{L^2(0,T;H^1(\Omega))}\  \|\phi_n\|^2_{L^\infty(0,T;L^2(\Omega))}+\ \|g\|^2_{L^2(0,T;H^1(\Omega)^*)}\right].\nonumber
\end{align}

Hence, by incorporating the uniform bound for $\big\{\phi_n\big\}$ in $L^\infty(0,T;L^2(\Omega))$ and $L^2(0,T;H^1(\Omega))$ from inequality \eqref{AEE1} into the previously obtained estimate, it follows that $\{\phi_n^\prime\}$ is uniformly bounded in $L^2(0, T;H^1(\Omega)^*)$. Moreover, invoking both the Alaoglu's weak* compactness theorem and the reflexive weak compactness theorem, we can conclude:
\begin{eqnarray} \left\{\begin{array}{cccll}
		\phi_n &\overset{w}{\rightharpoonup} & \phi \  &\mbox{weakly in}& \ L^2(0,T;H^1(\Omega)),\\
		\phi_n &\overset{w^\ast}{\rightharpoonup} & \phi \ &\mbox{weak$^*$ in}& \ L^{\infty}(0,T;L^2(\Omega)),\\
		\phi^\prime_n &\overset{w}{\rightharpoonup} & \phi^\prime \ &\mbox{weakly in}& \ L^2(0,T;H^1(\Omega)^*), \ \ \mbox{as} \  \ n\to \infty. \label{AEWC}
	\end{array}\right.	
\end{eqnarray}

Once more, the Aubin-Lions-Simon lemma (see, Corollary 4, \cite{JS}) plays a crucial role in establishing the existence of a subsequence from $\{\phi_n\}$ (which we continue to denote as $\{\phi_n\}$) such that $\phi_n \overset{s}{\to} \phi$ strongly in $L^2(0,T;L^2(\Omega))$. Leveraging this strong convergence in conjunction with \eqref{AEWC}, we can verify that $\phi$ satisfies condition (i) of Definition \ref{AWSD} for each $\vartheta \in \text{span}(w_1,w_2,...)$. By the density of such functions in $H^1(\Omega)$, this holds true for every $\vartheta \in H^1(\Omega)$. This concludes the proof.
\end{proof}

\subsection{First Order Optimality Condition}\label{SS-FOOC}

In general, Theorem \ref{T-EOOC} does not guarantee the uniqueness of the globally optimal solution. Due to the inherent nonlinearity of the control-to-state operator, we cannot necessarily assume convexity of the cost functional. Consequently, the optimization problem might encompass multiple locally optimal solutions or even several distinct globally optimal solutions. In the ensuing discussion, it is important to note that numerical methods typically tend to identify local minimizers, thereby motivating our focus on delineating locally optimal solutions through essential optimality conditions.

\begin{proof}[Proof of Theorem \ref{FOOCT}]
	
	 As $\mathbb{U}_{ad}$ is a convex set, we have that $\widetilde{U}+\epsilon (U-\widetilde{U})$ lies within $\mathbb{U}_{ad}$ for any control $U,\widetilde{U}\in \mathbb{U}_{ad}$ and $0\leq \epsilon\leq 1$. Since the control-to-state operator is Fréchet differentiable, the reduced functional $\mathcal{I}$ is consequently also Fréchet differentiable. Moreover, since we assume that $\widetilde{U}$ is an optimum, the functional $\mathcal{I}(\cdot)$ satisfies $\mathcal{I}'(\widetilde{U})[U-\widetilde{U}] \geq 0$ for all $ U \in \mathbb{U}_{ad}$. Now, by setting $v=U-\widetilde{U}$ and applying chain rule, we have
	\begin{flalign}
		 \mathcal{I}' &(\widetilde{U}) [v]  = \frac{\partial \mathcal{J}}{\partial m}(\widetilde{m},\widetilde{U}) \circ\left(  G'(\widetilde{U}) [v]\right) + \frac{\partial \mathcal J}{\partial U} (\widetilde{m},\widetilde{U})[v]\nonumber\\
		&=\sum_{i=1}^N\int_0^T  \widetilde{U}_i\ v_i\ dt + \int_{\Omega_T} ( \widetilde{m}- m_d)\cdot  z\ dx\ dt+ \int_\Omega \left(\widetilde{m}(T)-m_\Omega\right)\cdot z(T)\ dx,\label{KKR}
	\end{flalign}
	where $z=G'(\widetilde{U})[v] \in \mathcal{M}$ \ is a unique regular solution of the linearized system \eqref{LS2}.

	Next, we consider $\vartheta =z_v$ in the weak adjoint formulation given in Definition \ref{AWSD} of system \eqref{AS} and proceed further using the equality $\int_0^T \langle \phi'(t),z_v(t)\rangle_{H^1(\Omega)^*\times H^1(\Omega)}\ dt= -\int_0^T \big(\phi(t),z_v'(t)\big)\ dt+\int_{\Omega}(\widetilde{m}(x,T)-m_{\Omega})\cdot z(x,T)\ dx $. 	On the other hand, we take the inner product of the linearized system \eqref{LS2} with $\phi$,
	combining both these resultants and then substituting in \eqref{KKR}, we deduce	
	\begin{equation*}
		\mathcal{I}'(\widetilde{U}) [v] = \sum_{i=1}^N\int_0^T  \widetilde{U}_i\ v_i\ dt    +\int_{\Omega_T} \Big(\phi \times \widetilde{m}+ \phi\Big)\cdot \zetaup(v)\ dx\ dt.
	\end{equation*}
	Since $\mathcal{I}'(\widetilde{U})[U-\widetilde{U}] \geq 0$, for all $U \in \mathbb{U}_{ad}$, therefore invoking $v=U-\widetilde{U}$ in the above equality, we arrive at the desired optimality condition \eqref{FOOC}.
\end{proof}

\section{Second Order Optimality Condition}\label{S-SOOC}

\subsection{Control-to-Costate Operator}\label{S-CTCO}

\begin{Lem}\label{LCCTC}
	The operator $\Phi:\mathbb{U}_R\to \mathcal{Z}$, which maps from the control space $\mathbb{U}_R$ to the costate space $\mathcal{Z}$, exhibits Lipschitz continuity, that is, there exists a positive constant $C_3$, which relies on $\Omega$, $T$, $R$, and $m_0$, and satisfies the following inequality
	\begin{equation}\label{LCI}
		\|\Phi(U_1)-\Phi(U_2)\|^2_{\mathcal{Z}} \leq C_3 \ \|U_1-U_2\|^2_{L^2(0,T;\mathbb{R}^N)} \ \ \ \ \ \ \forall \ U_1,U_2\in \mathbb{U}_R.
	\end{equation}
\end{Lem}
\noindent The proof of this lemma can be carried out in a manner similar to the proof of Lemma \ref{L-LCCTS}.

\begin{Pro}\label{P-CTC}
	Assume that $\bar{U}\in \mathbb{U}_R$ be the control and $m_{\bar{U}}\in \mathcal{M}$ be the corresponding regular solution. Then the following two conclusions hold:
	\begin{enumerate}[label=(\roman*)]
		\item The control-to-costate mapping $\Phi$ is Fr\'echet differentiable on $\mathbb{U}_R$, that is, for any $\bar{U}\in \mathbb{U}_R$, there exists a bounded linear operator $\Phi'(\bar{U}):L^2(0,T;\mathbb{R}^N) \to \mathcal{Z}$ such that
		$$\frac{\|\Phi(\bar{U}+U)-\Phi(\bar{U})-\Phi'(\bar{U})[U]\|_{\mathcal{Z}}}{\|U\|_{L^2(0,T;\mathbb{R}^N)}}\to 0 \ \ \ \ \ \text{as}  \ \|U\|_{L^2(0,T;\mathbb{R}^N)}\to 0,$$
		where $\phi':=\Phi'(\bar{U})[U]$ is the unique weak solution of the following system:
		\begin{equation}\label{ASD}
			\begin{cases}
				\mathcal{E}_{\bar{U}}\phi' = -\Delta (\phi_{\bar{U}}\times z)- (\Delta z\times \phi_{\bar{U}})+\big(\phi_{\bar{U}}\times \zetaup(U)\big) +2\ (m_{\bar{U}}\cdot z)\ \phi_{\bar{U}}\\
				\hspace{2.4cm}+2\ (z\cdot \phi_{\bar{U}})\ m_{\bar{U}} +2\ (m_{\bar{U}}\cdot \phi_{\bar{U}})\ z-z \ \ \ \ \text{in}\ \Omega_T,\vspace{0.1cm}\\
				\frac{\partial \phi'}{\partial \eta}=0 \ \ \ \text{in}\ \partial \Omega_T,\vspace{0.1cm}\\
				\phi'(T)=z(T)\ \ \ \text{in} \ \Omega.
			\end{cases}
		\end{equation}
		\item The Fr\'echet derivative $\Phi'$ is Lipschitz continuous, that is, for any controls $U_1,U_2 \in \mathbb{U}_R$ and $U\in L^2(0,T;\mathbb{R}^N)$, there exists a constant $C_4>0$ depending on $\Omega,T,R,m_0$ such that 
		\begin{equation}\label{LC-AD}
		\|\Phi'(U_1)[U]-\Phi'(U_2)[U] \|_{\mathcal{Z}} \leq C_4\  \|U_1-U_2\|_{L^2(0,T;\mathbb{R}^N)}\  \|U\|_{L^2(0,T;\mathbb{R}^N)}.	
		\end{equation}
		\end{enumerate}
\end{Pro}

\begin{proof}
		Suppose $\bar{U}\in \mathbb{U}_R$ is a fixed control, and $m_{\bar{U}}$ be the corresponding regular solution. Let $z=G'(\bar{U})[U]$ represent the unique regular solution of the linearized system \eqref{LS2}, and $\phi_{\bar{U}}$ denote the unique weak solution of the adjoint system \eqref{AS}. Building upon the conclusions of Lemma \ref{AL-SLS}, we can readily establish that the system \eqref{ASD} admits a unique weak solution denoted by $\phi'$. Just for clarification, $\phi'$ is currently used solely as a notation and does not denote the derivative of the control-to-costate operator at this point. 
		
	Now, let us take $\shat{\phi}=\phi_{\bar{U}+U}-\phi_{\bar{U}}$, $\shat{m}=m_{\bar{U}+U}-m_{\bar{U}}$ and $w=m_{\bar{U}+U}-m_{\bar{U}}-z$. Consequently, the variable $\xi:=\phi_{\bar{U}+U}-\phi_{\bar{U}}-\phi'$ constitutes a solution for the following system adhering to the weak solvability outlined in Definition \ref{AWSD}: 
	\begin{equation}\label{ALO}
		\begin{cases}
			\displaystyle \mathcal{E}_{\bar{U}}\xi = \sum_{k=1}^{8} \Psi_k\ \ \ \ \text{in}\ \Omega_T,\\
			\frac{\partial \xi}{\partial \eta}=0 \ \ \ \ \text{on} \ \partial \Omega_T,  \ \ \ \ \xi(T)=m_{\bar{u}+u}(T)-m_{\bar{u}}(T)-z(T) \ \ \text{in} \ \Omega,	
		\end{cases}
	\end{equation}
	where $\mathcal{E}_{\bar{U}}$ is defined in \eqref{CLAS},  and the terms $\Psi_k$ are given by
	\begin{flalign*}
		\Psi_{1}&= -\Delta (\shat{\phi}\times \shat{m}) -\Delta (\phi_{\bar{U}}\times w) ,\hspace{1.7cm}
		\Psi_{2}= - (\Delta \shat{m}\times \shat{\phi})-(\Delta w\times \phi_{\bar{U}}),\\ 
		\Psi_{3}&= \shat{\phi}\times \zetaup(U)-w,\hspace{3.5cm} 
		\Psi_{4}= \big(\shat{m}\cdot \shat{m}\big)\ \phi_{\bar{U}+U} +2\ \big(\shat{m}\cdot m_{\bar{U}}\big)\ \shat{\phi} +2\ \big(m_{\bar{U}}\cdot w\big)\ \phi_{\bar{U}},\\
		\Psi_{5}&= 2\ \big(\shat {m}\cdot \shat{\phi}\big)\ m_{\bar{U}+U}+2 \ \big(\shat{m}\cdot \phi_{\bar{U}}\big)\ \shat{m}+ 2\  \big(w\cdot \phi_{\bar{U}}\big)\ m_{\bar{U}}+2 \ \big(m_{\bar{U}}\cdot \shat{\phi}\big)\ \shat{m} +2\ \big(m_{\bar{U}}\cdot \phi_{\bar{U}}\big)\ w.
	\end{flalign*}

Utilizing Lemma \ref{AL-SLS}, we can establish an energy estimate concerning the solution of system \eqref{ALO}, that is, any function $\xi\in \mathcal{Z}$ that satisfies system \eqref{ALO} in the sense of Definition \ref{AWSD}, satisfies the following estimate:
\begin{equation}\label{DAO-E1}
	\mathcal{D}(\xi_{\infty},\xi_2,{\xi_t}_2) \leq \left(\|m_{\bar{U}+U}(T)-m_{\bar{U}}(T)-z(T)\|^2_{L^2(\Omega)}+ \sum_{k=1}^{5}  \|\Psi_k\|^2_{L^2(0,T;H^1(\Omega)^*)}\right) \mathcal{K}(m_{\bar{U}},\bar{U}).	
\end{equation}
Now, let's proceed to estimate the terms involving $\Psi_k$'s on the right-hand side. By employing estimates \eqref{AEE-3} and \eqref{AEE-4} for the $L^2(0,T;H^1(\Omega)^*)$ norm of $\Psi_1$ and $\Psi_2$ respectively, we deduce
\begin{flalign*}
	\|\Psi_1\|^2_{L^2(0,T;H^1(\Omega)^*)}+ 	\|\Psi_2\|^2_{L^2(0,T;H^1(\Omega)^*)} &\leq C\  \|\shat{m}\|^2_{L^\infty(0,T;H^2(\Omega))}\|\shat{\phi}\|^2_{L^2(0,T;H^1(\Omega))}+C \ \|w\|^2_{L^\infty(0,T;H^2(\Omega))} \|\phi_{\bar{U}}\|^2_{L^2(0,T;H^1(\Omega))}.
	\end{flalign*}
Similarly, employing estimates \eqref{AEE-8} and \eqref{AEE-9} to evaluate the $L^2(0,T;H^1(\Omega)^*)$ norm of $\Psi_3$ and $\Psi_4$ respectively, we obtain
	\begin{flalign*}
		&\|\Psi_3\|^2_{L^2(0,T;H^1(\Omega)^*)} \leq C\ \|\shat{\phi}\|^2_{L^\infty(0,T;L^2(\Omega))} \|\zetaup(U)\|^2_{L^2(0,T;H^1(\Omega))} + C\ \|w\|^2_{L^2(0,T;L^2(\Omega))}, \\
		&\|\Psi_4\|^2_{L^2(0,T;H^1(\Omega)^*)} \leq C\ \|\shat{m}\|^4_{L^\infty(0,T;H^1(\Omega))}\   \|\phi_{\bar{U}+U}\|^2_{L^2(0,T;L^2(\Omega))}\\
		&\hspace{2cm} + C\ \left(\|\shat{m}\|^2_{L^\infty(0,T;H^1(\Omega))}\ \|\shat{\phi}\|^2_{L^2(0,T;L^2(\Omega))} + \|w\|^2_{L^\infty(0,T;H^1(\Omega))} \|\phi_{\bar{U}}\|^2_{L^2(0,T;L^2(\Omega))}\right)\  \|m_{\bar{U}}\|^2_{L^\infty(0,T;H^1(\Omega))}.
	\end{flalign*}
	Lastly, estimating $\|\Psi_5\|_{L^2(0,T;H^1(\Omega)^*)}$ using \eqref{AEE-10}, we find
	\begin{flalign*}
		&\|\Psi_5\|^2_{L^2(0,T;H^1(\Omega)^*)} \leq C\ \|\shat{m}\|^2_{L^\infty(0,T;H^1(\Omega))}\  \|\shat{\phi}\|^2_{L^2(0,T;L^2(\Omega))}\ \|m_{\bar{U}+U}\|^2_{L^\infty(0,T;H^1(\Omega))}&\\
		&\hspace{2cm}+C\ \left( \|\shat{m}\|^2_{L^\infty(0,T;H^1(\Omega))} + \|m_{\bar{U}}\|^2_{L^\infty(0,T;H^1(\Omega))}\right) \|\shat{m}\|^2_{L^\infty(0,T;H^1(\Omega))}\  \|\shat{\phi}\|^2_{L^2(0,T;L^2(\Omega))}\\
		&\hspace{2cm}+C\ \|w\|^2_{L^\infty(0,T;H^1(\Omega))}\ \|m_{\bar{U}}\|^2_{L^\infty(0,T;H^1(\Omega))}\ \|\phi_{\bar{U}}\|^2_{L^2(0,T;L^2(\Omega))}.
	\end{flalign*}

	Next, we will combine all the aforementioned estimates, substitute them into \eqref{DAO-E1}, and divide by $\|U\|^2_{L^2(0,T;\mathbb{R}^N)}$. As $\mathbb{U}_R$ is an open subset of $L^2(0,T;\mathbb{R}^N)$, there exists a $\rho>0$ such that for any control $U\in L^2(0,T;\mathbb{R}^N)$ with $\|U\|_{L^2(0,T;\mathbb{R}^N)}< \rho$, we have $\bar{U}+U \in \mathbb{U}_R$. By virtue of estimate \eqref{SSEE}, we can establish a uniform bound for $\|m_{\bar{U}}\|_{\mathcal{M}}$ and $\|m_{\bar{U}+U}\|_{\mathcal{M}}$ for such small controls. Accordingly, by employing estimate \eqref{AEEE} we can find a bound for $\|\phi_{\bar{U}}\|_{\mathcal{Z}}$ and $\|\phi_{\bar{U}+U}\|_{\mathcal{Z}}$. 
	
	From the estimates \eqref{LCCTSO} and \eqref{LCI}, it is evident that $\|\shat{m}\|^2_{\mathcal{M}}\leq C\ \|U\|^2_{L^2(0,T;\mathbb{R}^N)}$ and $\|\shat{\phi}\|^2_{\mathcal{Z}}\leq C\ \|U\|^2_{L^2(0,T;\mathbb{R}^N)}$. Additionally, from estimate \eqref{X1}, we can observe that $\|w\|^2_{\mathcal{M}}\leq C\ \|U\|^4_{L^2(0,T;\mathbb{R}^N)}  + C\  \|U\|^6_{L^2(0,T;\mathbb{R}^N)}$. Therefore, by utilizing all these observations, we can conclude that  $\frac{\|\Psi_k\|^2_{L^2(0,T;H^1(\Omega)^*)}}{\|U\|^2_{L^2(0,T;\mathbb{R}^N)}}\to 0$ as $\|U\|_{L^2(0,T;\mathbb{R}^N)}\to 0$ for $k=1,2,...,5$. Moreover,
	\begin{align*}
	&\frac{\|m_{\bar{U}+U}(T)-m_{\bar{U}}(T)-z(T)\|^2_{L^2(\Omega)}}{\|U\|^2_{L^2(0,T;\mathbb{R}^N)}}\leq  \frac{\|m_{\bar{U}+U}-m_{\bar{U}}-z\|^2_{L^\infty(0,T;L^2(\Omega))}}{\|U\|^2_{L^2(0,T;\mathbb{R}^N)}}\\
	&\hspace{1.1cm} \leq \frac{\|m_{\bar{U}+U}-m_{\bar{U}}-z\|^2_{\mathcal{M}}}{\|U\|^2_{L^2(0,T;\mathbb{R}^N)}} = \frac{\|w\|^2_{\mathcal{M}}}{\|U\|^2_{L^2(0,T;\mathbb{R}^N)}} \leq C\ \frac{\|U\|^4_{L^2(0,T;\mathbb{R}^N)}  +   \|U\|^6_{L^2(0,T;\mathbb{R}^N)}}{\|U\|^2_{L^2(0,T;\mathbb{R}^N)}}\to 0 .		
	\end{align*}
Hence, by dividing both sides of estimate \eqref{DAO-E1} by $\|U\|^2_{L^2(0,T;\mathbb{R}^N)}$, we deduce that   $\frac{\|\xi\|^2_{\mathcal{Z}}}{\|U\|^2_{L^2(0,T;\mathbb{R}^N)}}\to 0$ as $\|U\|_{L^2(0,T;\mathbb{R}^N)}\to 0$ for $k=1,2,...,5$. This implies that the control-to-costate operator is Fr\`echet differentiable. 

Let $\phi'_{U_1}:= \Phi'(U_1)[U]$ and $\phi'_{U_2}:=\Phi'(U_2)[U]$ represent two distinct weak solutions of system \eqref{ASD}. Introducing the notation $\shat{\psi}=\phi'_{U_1}-\phi'_{U_2}$, $\shat{m}=m_{U_1}-m_{U_2}$ and $\shat{U}=U_1-U_2$, we observe that $\shat{\psi}$ will satisfy the following system weakly in the sense of Definition \ref{AWSD}:
	\begin{equation}\label{DOAO}
	\begin{cases}
		\displaystyle \mathcal{E}_{U_1} \shat{\psi} = \sum_{k=1}^{5} \mathcal{F}_k\ \ \ \ \text{in}\ \Omega_T,\\
		\frac{\partial \shat{\psi}}{\partial \eta}=0 \ \ \ \ \text{on} \ \partial \Omega_T,  \ \ \ \ \shat{\psi}(T)=z_{U_1}(T)-z_{U_2}(T) \ \ \text{in} \ \Omega,	
	\end{cases}
\end{equation}
where $\mathcal{E}_{U_1}$ is defined in \eqref{CLAS},  and the terms $\mathcal{F}_k$ are given by
\begin{flalign*}
\mathcal{F}_1&=-\Delta (\phi'_{U_2}\times \shat{m})  -\Delta (\shat{\phi}\times z_{U_1})-\Delta (\phi_{U_2}\times \shat{z}),\\
\mathcal{F}_2&=-\Delta \shat{m}\times \phi'_{U_2} -(\Delta \shat{z} \times \phi_{U_1}) -(\Delta z_{U_2}\times \shat{\phi}),\\
\mathcal{F}_3&=\phi'_{U_2}\times \zetaup(\shat{U})+ \shat{\phi}\times \zetaup(U)-\shat{z},\\
\mathcal{F}_4&= \big(\shat{m}\cdot\big(m_{U_1}+m_{U_2}\big)\big)\  \phi_{U_2}' +2 \ \big(\shat{m}\cdot z_{U_1}\big)\ \phi_{U_1}+ 2\ \big(m_{U_2}\cdot \shat{z}\big)\ \phi_{U_1}+2\ \big(m_{U_2}\cdot z_{U_2}\big)\ \shat{\phi} , \\
\mathcal{F}_5&=2\ \big(\shat{m}\cdot \phi_{U_2}'\big)\ m_{U_1} +2\ \big(m_{U_2}\cdot \phi_{U_2}'\big)\ \shat{m}+2\ \big(\shat{z}\cdot \phi_{U_1}\big)\ m_{U_1}+2\ \big(z_{U_2}\cdot \shat{\phi}\big)\ m_{U_1}\\
&\hspace{1cm} +2\ \big(z_{U_2}\cdot \phi_{U_2}\big)\ \shat{m} +2\ \big(\shat{m}\cdot \phi_{U_1}\big)\ z_{U_1} +2\ \big(m_{U_2} \cdot \shat{\phi}\big)\ z_{U_1} +2\ \big(m_{U_2} \cdot \phi_{U_2}\big)\ \shat{z}.
\end{flalign*}
Once more, utilizing Lemma \ref{AL-SLS}, we can derive an estimate for the solution of the system \eqref{DOAO} as follows:
\begin{equation}\label{DAO-E2}
	\mathcal{D}(\shat{\psi}_{\infty},\shat{\psi}_2,{\shat{\psi}_t}_2) \leq \left(\|z_{U_1}(T)-z_{U_2}(T) \|^2_{L^2(\Omega)}+ \sum_{k=1}^{5}  \|\mathcal{F}_k\|^2_{L^2(0,T;H^1(\Omega)^*)}\right) \ 	\mathcal{K}(m_{U_1},U_1).
\end{equation}
Performing calculations analogous to Part-(i), we can establish the following inequality:
$$\sum_{k=1}^5\|\mathcal{F}_k\|^2_{L^2(0,T;H^1(\Omega)^*)} \leq C\ \|U_1-U_2\|^2_{L^2(0,T;\mathbb{R}^N)}\  \|U\|^2_{L^2(0,T;\mathbb{R}^N)}.$$
Additionally, for the final time data, using the inequality \eqref{LC-CTS}, we obtain the estimate $\|z_{U_1}(T)-z_{U_2}(T)\|^2_{L^2(\Omega)} \leq \|z_{U_1}-z_{U_2}\|^2_{L^\infty(0,T;L^2(\Omega))} \leq C\ \|U_1-U_2\|^2_{L^2(0,T;\mathbb{R}^N)}\ \|U\|^2_{L^2(0,T;\mathbb{R}^N)}$. By substituting these estimates into \eqref{DAO-E2}, we arrive at the desired result \eqref{LC-AD}. This concludes the proof.
\end{proof}

\begin{Cor}\label{C-CTCD}
	The Fr\'echet derivative of the control-to-costate operator satisfies the following estimate:
	\begin{equation}\label{FDCTC-S}
		\|\Phi'(U)[V]\|_{\mathcal{Z}} \leq C(\Omega,T,U,m_U)\ \|V\|_{L^2(0,T;\mathbb{R}^N)}.
	\end{equation}
\end{Cor}

This estimate plays a fundamental role in understanding the sensitivity of the costate with respect to changes in the control, which is crucial in optimization and control problems involving partial differential equations.

\subsection{Second Order Optimality Condition}\label{SS-SOOC}
Since the state equation is in non-linear form, the associated control problem exhibits non-convex characteristics, potentially leading to multiple solutions based on the necessary first-order conditions. To gain deeper insights, the analysis can be extended using second-order conditions.

\begin{proof}[Proof of Theorem \ref{T-SOLO}]
	We'll establish the theorem by employing a proof by contradiction argument as in Theorem 7.5, \cite{ECFT}. Let's assume that $\widetilde{U}$ doesn't satisfy the growth condition as specified by inequality \eqref{QGC}. This implies the existence of a sequence of control functions $\{U_k\}^{\infty}_{k=1}$ in $\mathbb{U}_{ad}$ such that $U_k \to \widetilde{U}$ in $L^2(0,T;\mathbb{R}^N)$, and the following inequality holds:
	\begin{equation}\label{CQGC}
		\mathcal{J}(U_k) < \mathcal{J}(\widetilde{U}) + \frac{1}{k}\ \|U_k-\widetilde{U}\|^2_{L^2(0,T;\mathbb{R}^N)}\ \ \ \ \forall \ k.
	\end{equation}
	Let us introduce $\rho_k=\|U_k-\widetilde{U}\|_{L^2(0,T;\mathbb{R}^N)}$ and $h_k=\frac{1}{\rho_k}\ (U_k-\widetilde{U})$. Since $\|h_k\|_{L^2(0,T;\mathbb{R}^N)}=1$, we are able to extract a sub-sequence (again denoted as $\{h_k\}$), such that $h_k \rightharpoonup h$ weakly in $L^2(0,T;\mathbb{R}^N)$.
	
	\noindent	\underline{\textit{Step-I}} : $\frac{\partial \mathcal{J}}{\partial U}(\widetilde{U})[h]=0$
	
	Applying the mean value theorem, we obtain the equality $\displaystyle 
		\mathcal{J}(U_k)=\mathcal{J}(\widetilde{U}) + \rho_k\frac{\partial \mathcal{J}}{\partial U}(v_k)[h_k]$, where $v_k$ is a point between $\widetilde{U}$ and $U_k$. Substituting this equality in estimate \eqref{CQGC}, we find 
	\begin{equation}\label{SO-E1}
		\frac{\partial \mathcal{J}}{\partial U}(v_k)[h_k] < \frac{1}{k\rho_k}\|U_k-\widetilde{U}\|^2_{L^2(0,T;\mathbb{R}^N)}=\frac{1}{k} \|U_k-\widetilde{U}\|_{L^2(0,T;\mathbb{R}^N)}.	
	\end{equation}
	The Gateaux derivative of the functional $\mathcal{J}$ at point $v_k$ in the direction of $h_k$ is expressed as
	\begin{equation}\label{FOV}
		\frac{\partial \mathcal{J}}{\partial U}(v_k)[h_k]=   \int_0^T v_k \cdot h_k\ dt+  \int_{\Omega_T} \phi_{v_k}\cdot \zetaup(h_k)\ dx\ dt +  \int_{\Omega_T} \big(\phi_{v_k}\times m_{v_k}\big) \cdot \zetaup(h_k)\ dx\ dt. 
	\end{equation}
	As $v_k$ lies between $U_k$ and $\widetilde{U}$, and $U_k$ converges strongly to $\widetilde{U}$ in $L^2(0,T;\mathbb{R}^N)$, it follows that $v_k$ converges strongly to $\widetilde{U}$ strongly in $L^2(0,T;\mathbb{R}^N)$. Now, based on estimates \eqref{LCCTSO} and \eqref{LCI}, it becomes apparent that $m_{v_k}$ converges strongly to $\widetilde{m}$ in $\mathcal{M}$, and $\phi_{v_k}$ converges strongly to  $\phi_{\widetilde{U}}$ in $\mathcal{Z}$. Furthermore, the weak convergence of $h_k$ to $h$ in $L^2(0,T;\mathbb{R}^N)$ implies that $\zetaup(h_k)\rightharpoonup\zetaup(h)$ weakly in $L^2(0,T;L^2(\Omega))$. Thus the following convergences hold: 
	\begin{enumerate}[label=(\roman*)]
		\item $\displaystyle \int_0^T v_k \cdot h_k\ dt\to \int_0^T \widetilde{U}\cdot h\ dt,$
		\item $\displaystyle \int_{\Omega_T} \phi_{v_k}\cdot \zetaup(h_k)\ dx\ dt \to \int_{\Omega_T} \phi_{\widetilde{U}} \cdot \zetaup(h)\ dx\ dt$,
		\item $\displaystyle \int_{\Omega_T} (\phi_{v_k}\times m_{v_k})\cdot \zetaup(h_k)\ dx\ dt \to \int_{\Omega_T} (\phi_{\widetilde{U}}\times \widetilde{m}) \cdot \zetaup(h)\ dx\ dt$.
	\end{enumerate}
	We will focus solely on demonstrating the third convergence. Similar approaches can be applied to establish the first and second convergences. Beginning with the application of H\"older's inequality and the embedding $H^1(\Omega)\hookrightarrow L^4(\Omega)$, we derive 
	\begin{align*}
		\bigg| &\int_{\Omega_T} \big(\phi_{v_k}\times m_{v_k}\big)\cdot \zetaup(h_k)\ dx\ dt - \int_{\Omega_T} \big(\phi_{\widetilde{U}}\times \widetilde{m}\big)\cdot \zetaup(h)\ dx\ dt\ \bigg|&\\
		&\leq \int_{\Omega_T} \big| \phi_{v_k}\times m_{v_k} - \phi_{\widetilde{U}}\times \widetilde{m}\big|\ \big|\zetaup(h_k)\big|\ dx\ dt +\Big|\int_{\Omega_T} \big(\phi_{\widetilde{U}}\times \widetilde{m}\big) \cdot \big(\zetaup(h_k)-\zetaup(h)\big)\ dx\ dt\Big| \\
		&\leq \left(\|\phi_{v_k}-\phi_{\widetilde{U}}\|_{L^2(0,T;L^2(\Omega))}\ \|m_{v_k}\|_{L^\infty(0,T;H^1(\Omega))}+\|\phi_{\widetilde{U}}\|_{L^2(0,T;H^1(\Omega))}\ \|m_{v_k}-\widetilde{m}\|_{L^\infty(0,T;L^2(\Omega))} \right)\ \|\zetaup(h_k)\|_{L^2(0,T;H^1(\Omega))}\\
		&\hspace{1.5cm}+\Big|\int_{\Omega_T} \big(\phi_{\widetilde{U}}\times \widetilde{m}\big) \cdot \big(\zetaup(h_k)-\zetaup(h)\big)\ dx\ dt\ \Big|\ \to 0 \ \text{as}\ k \to \infty.
	\end{align*}
	Consequently, from \eqref{SO-E1}, we can deduce that $\displaystyle \frac{\partial \mathcal{J}}{\partial U}(\widetilde{U})[h] = \lim_{k\to \infty} 	\frac{\partial \mathcal{J}}{\partial U}(v_k)[h_k] \leq 0$. Moreover, from the first-order variational inequality \eqref{FOOC}, we know that $\displaystyle \frac{\partial \mathcal{J}}{\partial U}(\widetilde{U})[h_k]=\frac{1}{\rho_k}\frac{\partial \mathcal{J}}{\partial U}(\widetilde{U})[U_k-\widetilde{U}] \geq 0$. Hence, leveraging the weak convergence of ${h_k}\rightharpoonup h$ in $L^2(0,T;\mathbb{R}^N)$, we ascertain from \eqref{FOV} that 
	$\displaystyle \frac{\partial \mathcal{J}}{\partial U}(\widetilde{U})[h]=\lim_{k\to\infty} \frac{\partial \mathcal{J}}{\partial U}(\widetilde{U})[h_k] \geq 0$. Therefore, we conclude that $\displaystyle \frac{\partial \mathcal{J}}{\partial U}(\widetilde{U})[h]=0$.
	
\noindent	\underline{\textit{Step-II}} : $h \in \Lambda(\widetilde{U})$ 
	
	The collection of controls in $L^2(0,T;\mathbb{R}^N)$ that are non-negative when $\widetilde{U}(t)=a(t)$ and non-positive when $\widetilde{U}(t)=b(t)$, denoted as $\mathcal{S}$, is convex, closed, and thus, weakly closed. Since $h_k\in\mathcal{S}$ for each $k$, as a consequence of this weakly closed property, $h$ also belongs to $\mathcal{S}$.
	
	Finally, to establish $h\in \Lambda(\widetilde{U})$, we only need to demonstrate that $h=0$ if $\varUpsilon_{\widetilde{U}}\neq 0$, where $\varUpsilon_{\widetilde{U}}$ is given in Definition \ref{D-CC}. From the first-order variational inequality  \eqref{FOOC}, we have
	$$\sum_{i=1}^N\int_0^T \left(\varUpsilon_{\widetilde{U}}\right)_i(t)\cdot \big(U_i(t)-\widetilde{U}_i(t)\big)\ dt \geq 0,\ \ \ \ \forall \ U=(U_1,U_2,...,U_N)\in \mathbb{U}_{ad}.$$
	Furthermore, the following condition can be deduced from the above inequality:
	\begin{equation}\label{SO-E4}
		\begin{cases}
			\widetilde{U}_i(t)=a_i(t) \ \ \ \text{if}\ \ \left(\varUpsilon_{\widetilde{U}}\right)_i>0,\\
			\widetilde{U}_i(t)=b_i(t) \ \ \ \text{if}\ \ \left(\varUpsilon_{\widetilde{U}}\right)_i<0.
		\end{cases}
	\end{equation}
	Since $h$ belongs to the control set $\mathcal{S}$ and $\widetilde{U}$ satisfies the condition in \eqref{SO-E4}, we can establish the following equality:
	$$\sum_{i=1}^N \int_0^T |\varUpsilon_{\widetilde{U}_i}(t) \ h_i(t)|\ dt= \sum_{i=1}^N \int_0^T  \varUpsilon_{\widetilde{U}_i}(t)\ h_i(t)\ dt=\frac{\partial \mathcal{J}}{\partial U}(\widetilde{U}) [h]=0.$$
	This equality indicates that  $h_i(t)$ must be equal to $0$ whenever $\varUpsilon_{\widetilde{U}_i}\neq 0$. This concludes that $h\in \Lambda(\widetilde{U})$. 
	
	\noindent	\underline{\textit{Step-III}} :  $h=0$
	
	To establish our claim $h=0$, we aim to prove $\frac{\partial^2 \mathcal{J}}{\partial U^2}(\widetilde{U})[h,h]\leq 0$. We begin by employing the second-order Taylor series expansion for $\mathcal{J}(U_k)$. This allows us to express the function as follows:
	\begin{equation*}\label{SO-E6}
		\mathcal{J}(U_k)=\mathcal{J}(\widetilde{U})+\rho_k \frac{\partial \mathcal{J}}{\partial U}(\widetilde{U})[h_k]+\frac{\rho_k^2}{2}\frac{\partial^2 \mathcal{J}}{\partial U^2}(w_k)[h_k,h_k],
	\end{equation*}
	where $w_k$ is an intermediate point between $\widetilde{U}$ and $U_k$. By adding and subtracting $\frac{\rho_k^2}{2}\frac{\partial^2 \mathcal{J}}{\partial U^2}(\widetilde{U})[h_k,h_k]$ on both sides of the above equality, we can rewrite it as:
	\begin{equation}\label{SO-E7}
		\rho_k \frac{\partial \mathcal{J}}{\partial U}(\widetilde{U})[h_k]+\frac{\rho_k^2}{2}\frac{\partial^2 \mathcal{J}}{\partial U^2}(\widetilde{U})[h_k,h_k] =\mathcal{J}(U_k)-\mathcal{J}(\widetilde{U})+ \frac{\rho_k^2}{2} \left(\frac{\partial^2 \mathcal{J}}{\partial U^2}(\widetilde{U})-\frac{\partial^2 \mathcal{J}}{\partial U^2}(w_k)\right)[h_k,h_k].	
	\end{equation}
	By substituting $\|U_k-\widetilde{U}\|_{L^2(0,T;L^2(\Omega))}=\rho_k$ in estimate \eqref{CQGC}, we obtain $\mathcal{J}(U_k)-\mathcal{J}(\widetilde{U})< \frac{\rho_k^2}{k}$ for each $k$. Moreover, since $\frac{\partial \mathcal{J}}{\partial U}(\widetilde{U})[h_k] \geq 0$ and $\rho_k\to 0$, we have $\rho_k \frac{\partial \mathcal{J}}{\partial U}(\widetilde{U})[h_k] \geq \rho_k^2 \frac{\partial \mathcal{J}}{\partial U}(\widetilde{U})[h_k]$, for $k\geq k_0$ with sufficiently large $k_0$. Finally, substituting all these inequalities in equation \eqref{SO-E7}, we obtain
	\begin{equation}\label{SO-E9}
		\frac{\partial \mathcal{J}}{\partial U}(\widetilde{U})[h_k]+\frac{1}{2}\frac{\partial^2 \mathcal{J}}{\partial U^2}(\widetilde{U})[h_k,h_k] < \frac{1}{k} + \frac{1}{2} \left(\frac{\partial^2 \mathcal{J}}{\partial U^2}(\widetilde{U})-\frac{\partial^2 \mathcal{J}}{\partial U^2}(w_k)\right)[h_k,h_k].		
	\end{equation}
The second derivative of the functional $\mathcal{J}$ through \eqref{KKR} can be expressed as:
\begin{align}\label{SODE}
	\frac{\partial^2 \mathcal{J}}{\partial U^2}(\widetilde{U})[h_k,h_k]&=  \int_0^T  h_k^2 \ dt+ \int_{\Omega_T} \big(\phi'_{\widetilde{U}}[h_k]\times \widetilde{m}\big)\cdot \zetaup(h_k)\ dx \ dt  \nonumber\\
	&\hspace{.5cm} + \int_{\Omega_T} \big( \phi_{\widetilde{U}} \times m'_{\widetilde{U}}[h_k]\big) \cdot \zetaup(h_k) \ dx\ dt + \int_{\Omega_T} \phi'_{\widetilde{U}}[h_k] \cdot \zetaup(h_k) \ dx \ dt,
\end{align}
where $m'_{\widetilde{U}}[h_k]$ and $\phi'_{\widetilde{U}}[h_k]$ are solutions of system \eqref{LS2} and \eqref{ASD} respectively.

	In order to prove the second term on the right hand side of estimate \eqref{SO-E9} vanishes as $k \to \infty$, we have to show the following convergences in view of \eqref{SODE}:
	\begin{enumerate}[label=(\roman*)]
		\item $\displaystyle \bigg|\int_{\Omega_T} \Big(\phi'_{\widetilde{U}}[h_k] \times \widetilde{m}\Big) \cdot \zetaup(h_k)\ dx\ dt -\int_{\Omega_T} \Big(\phi'_{w_k}[h_k] \times m_{w_k}\Big) \cdot \zetaup(h_k)\ dx\ dt\ \bigg|  \to 0,$\\
		\item $\displaystyle  \bigg| \int_{\Omega_T} \big( \phi_{\widetilde{U}} \times m'_{\widetilde{U}}[h_k]\big) \cdot \zetaup(h_k) \ dx\ dt-\int_{\Omega_T} \big( \phi_{w_k} \times m'_{w_k}[h_k]\big) \cdot \zetaup(h_k) \ dx\ dt \ \bigg|  \to 0$,\\
		\item $\displaystyle  \bigg| \int_{\Omega_T} \phi'_{\widetilde{U}}[h_k] \cdot \zetaup(h_k) \ dx \ dt-\int_{\Omega_T} \phi'_{w_k}[h_k] \cdot \zetaup(h_k) \ dx \ dt \  \bigg|  \to 0\ \ \text{as}\ k\to \infty.$
	\end{enumerate}


We will only show the first convergence. Applying H\"older's inequality followed by Lipschitz continuity arguments \eqref{LCCTSO} and \eqref{LC-AD}, and the boundedness estimataes \eqref{FDCTC-S}, we derive 
	\begin{align*}
		\bigg|\int_{\Omega_T}&\Big(\phi'_{\widetilde{U}} [h_k] \times \widetilde{m}\Big) \cdot \zetaup(h_k)\ dx\ dt -\int_{\Omega_T} \Big(\phi'_{w_k} [h_k] \times m_{w_k}\Big) \cdot \zetaup(h_k)\ dx\ dt\ \bigg|&\\
		&\leq \int_{\Omega_T} \bigg(\big| \phi'_{\widetilde{U}} [h_k]-\phi'\big(w_k\big) [h_k] \big|\ |\widetilde{m}| +\big|\phi'_{w_k} [h_k]\big|\ |\widetilde{m}-m_{w_k}|\ \bigg)\ |\zetaup(h_k)|\ dx \ dt\\
		&\leq C\ \big\| \phi'_{\widetilde{U}} [h_k]-\phi'_{w_k} [h_k]\big\|_{L^2(0,T;L^2(\Omega))}\ \|\widetilde{m}\|_{L^\infty(0,T;H^1(\Omega))}\ \|\zetaup(h_k)\|_{L^2(0,T;H^1(\Omega))}\\
		&\hspace{1cm} + C\ \big\|\phi'_{w_k} [h_k]\big\|_{L^\infty(0,T;L^2(\Omega))} \ \|\widetilde{m}-m_{w_k}\|_{L^2(0,T;H^1(\Omega))}\ \|\zetaup(h_k)\|_{L^2(0,T;H^1(\Omega))}\\
		&\leq C\  \|\widetilde{U}-w_k\|_{L^2(0,T;\mathbb{R}^N(\Omega))}\ \|h_k\|^2_{L^2(0,T;\mathbb{R}^N(\Omega))} \to 0 \ \text{as}\ k\to \infty.
	\end{align*}
	Continuing in a similar fashion for (ii) and (iii) convergences, we obtain that the right hand side of \eqref{SO-E9} converges to 0 as $k\to \infty$.
	
	Now, let's examine the convergence of the second term on the left hand side of \eqref{SO-E9}. Since  $h_k \rightharpoonup h$ weakly in $L^2(0,T;\mathbb{R}^N)$, by weak sequential lower semi-continuity, we have $\int_0^T h^2\ dt\leq \liminf_{k \to \infty} \int_0^T h_k^2\ dt$. Furthermore, considering $\zetaup(h_k) \rightharpoonup \zetaup(h)$ weakly in $L^2(0,T;L^2(\Omega))$, and using the uniform boundedness, compact embeddings, and the results from Propositions \ref{P-CTS} and \ref{P-CTC}, we establish that $m'_{\widetilde{U}}[h_k]\to m'_{\widetilde{U}}[h]$ strongly in $L^2(0,T;H^2(\Omega))$ and $\phi'_{\widetilde{U}}[h_k]\to \phi'_{\widetilde{U}}[h]$ strongly in $L^2(0,T;L^2(\Omega))$. Therefore, applying these convergences for the third term of \eqref{SODE}, we have 
	\begin{align*}
		\bigg|\int_{\Omega_T}& \big(\phi_{\widetilde{U}} \times m'_{\widetilde{U}}[h_k]\big)\cdot \zetaup(h_k)\ dx\ dt  -\int_{\Omega_T} \big(\phi_{\widetilde{U}} \times m'_{\widetilde{U}}[h]\big)\cdot \zetaup(h)  \  dx \ dt\ \bigg| \\
		&\leq C\ \int_{\Omega_T} \big|\phi_{\widetilde{U}} \times m'_{\widetilde{U}}[h_k] - \phi_{\widetilde{U}} \times m'_{\widetilde{U}}[h]\big|\    |\zetaup(h_k)| \ dx\ dt + \bigg|\int_{\Omega_T} \big(\phi_{\widetilde{U}} \times m'_{\widetilde{U}}[h]\big)\cdot \big(\zetaup(h_k)-\zetaup(h)\big)  \  dx \ dt\ \bigg|\\
		&\leq C\ \|\phi_{\widetilde{U}}\|_{L^\infty(0,T;L^2(\Omega))}\ \|m'_{\widetilde{U}}[h_k]-m'_{\widetilde{U}}[h]\|_{L^2(0,T;H^1(\Omega))}\ \|\zetaup(h_k)\|_{L^2(0,T;H^1(\Omega))}\\
		&\hspace{2cm}+  \bigg|\int_{\Omega_T} \big(\phi_{\widetilde{U}} \times m'_{\widetilde{U}}[h]\big)\cdot \big(\zetaup(h_k)-\zetaup(h)\big)  \  dx \ dt\ \bigg| \to 0 \ \text{as}\ k\to \infty.
	\end{align*}
	Continuing similarly for the remaining terms in \eqref{SODE}, we get $\displaystyle \frac{\partial^2 \mathcal{J}}{\partial U^2}(\widetilde{U})[h,h] \leq \liminf_{k\to \infty} \frac{\partial^2 \mathcal{J}}{\partial U^2}(\widetilde{U})[h_k,h_k]$. Finally, as $k\to \infty$ in \eqref{SO-E9}, we establish $\frac{\partial^2 \mathcal{J}}{\partial U^2}(\widetilde{U})[h,h]\leq 0$, which concludes our assertion that $h=0$.
	
		\noindent	\underline{\textit{Step-IV}} : Contradiction to $\|h_k\|_{L^2(0,T;\mathbb{R}^N)}=1$ for each $k$

	Next, we will demonstrate that $h_k \to 0$ strongly in $L^2(0,T;\mathbb{R}^N)$.
	Since we have already established that $h=0$, the convergence of $\phi'_{\widetilde{U}}[h_k]$ and $m'_{\widetilde{U}}[h_k]$ to $0$ follows naturally. Implementing these results, we derive 	
	\begin{flalign*}
		0&< \lim_{k\to \infty} \int_0^T h_k^2 \ dt = \lim_{k\to \infty} \left\{\frac{\partial^2  \mathcal{J}}{\partial U^2}(\widetilde{U})[h_k,h_k] - \int_{\Omega_T} \Big(\phi'_{\widetilde{U}} [h_k] \times \widetilde{m}\Big) \cdot \zetaup(h_k)\ dx\ dt\right.\\ 		
		& \left.- \int_{\Omega_T} \big(\phi_{\widetilde{U}} \times m'_{\widetilde{U}}[h_k]\big)\cdot \zetaup(h_k)\ dx\ dt -  \int_{\Omega_T} \phi'_{\widetilde{U}} [h_k] \cdot \zetaup(h_k)\ dx\ dt \right\}\leq 0,
	\end{flalign*}
	which is a contradiction as $\|h_k\|_{L^2(0,T;\mathbb{R}^N)}=1$ for each k. This completes the proof.
\end{proof}


\section{Global Optimality Condition}\label{S-GOC}
The second order conditions typically offer local insights and often fall short of determining whether the given point represents a global minimum for the optimization problem. Now, assuming that we have an admissible control $\widetilde{U}$ which satisfies the first-order necessary condition \eqref{FOOC}, we will formulate a condition on the state and adjoint variable that guarantees that $\widetilde{U}$ is a global solution of MOCP.


\begin{proof}[Proof of Theorem \ref{T-GO}]
	
	In order to establish the global optimality of the control $\widetilde{U}$, our objective is to demonstrate that $ \mathcal{I}(U)-\mathcal{I}(\widetilde{U})\geq 0$ holds for all $U\in \mathbb{U}_{ad} \backslash \{\widetilde{U}\}$. Through a straightforward calculation, we arrive at the following equality:
\begin{align}\label{GO-1}
	\mathcal{I}(U)-\mathcal{I}(\widetilde{U}) =& \ \frac{1}{2}\int_0^T |\shat{U}|^2\  dt + \int_0^T \shat{U}  \cdot \widetilde{U} \ dt+ \int_{\Omega} \shat{m}(x,T)\cdot \big(\widetilde{m}(x,T)-m_{\Omega}(x)\big) \ dx \nonumber\\
	&+ \frac{1}{2} \int_{\Omega} |\shat{m}(x,T)|^2\ dx\ dt + \frac{1}{2} \int_{\Omega_T} |\shat{m}|^2\ dx\ dt+ \int_{\Omega_T} \shat{m} \cdot ( \widetilde{m}- m_d)\  \ dx\ dt,
\end{align}
 where $\shat{m}:= m-\widetilde{m}$, $\shat{U}:= U-\widetilde{U}$ and $m$ is the regular solution of \eqref{NLP-EV} corresponding to $U\in \mathbb{U}_{ad}$. Now, by substituting the first-order variational inequality \eqref{FOOC} into \eqref{GO-1}, we obtain 
\begin{equation}\label{GO-2}
	\mathcal{I}(U)-\mathcal{I}(\widetilde{U}) \geq  \ \frac{1}{2}\int_0^T  |\shat{U}|^2\ dt  + \frac{1}{2} \int_{\Omega} |\shat{m}(x,T)|^2\ dx\ dt + \frac{1}{2} \int_{\Omega_T} |\shat{m}|^2\ dx\ dt + \mathcal{R},	
\end{equation}
\begin{flalign*}
	\text{where}\ \mathcal{R} :=&  \int_{\Omega_T} \shat{m} \cdot \big( \widetilde{m}- m_d\big) \ dx\ dt+ \int_{\Omega} \shat{m}(x,T)\cdot \big(\widetilde{m}(x,T)-m_{\Omega}(x)\big) \ dx&\\
	&\hspace{1cm}-\int_{\Omega_T} \phi \cdot \zetaup(\shat{U})\ dx\ dt -\int_{\Omega_T} (\phi \times \widetilde{m})\cdot \zetaup(\shat{U}) \ dx\ dt.
\end{flalign*}

Furthermore, it is evident that the pair  $(\shat{m},\shat{U})$ satisfies the ensuing system:
\begin{equation}\label{GO-S}
	\begin{cases}
		\begin{array}{l}
			\mathcal{L}_{\widetilde{U}}\shat{m}= \shat{m} \times \Delta \shat{m} + \shat{m} \times \zetaup(\shat{U}) +\widetilde{m} \times \zetaup(\shat{U}) -|\shat{m}|^2 \shat{m} -|\shat{m}|^2\widetilde{m} -2\big(\shat{m}\cdot \widetilde{m}\big)\shat{m}+ \zetaup(\shat{U}) \  \text{in}\ \Omega_T,\\
			\frac{\partial \shat{m}}{\partial \eta}=0 \ \ \ \ \ \  \text{in}\ \partial \Omega_T,\\
			\shat{m}(x,0)=0\ \ \text{in}\ \Omega,
		\end{array}
	\end{cases}	
\end{equation}
where the operator $\mathcal{L}_{\widetilde{U}}$ is defined in \eqref{CLO}.

Considering the $L^2(\Omega_T)$ inner product of \eqref{GO-S} with $\phi$ and carrying out a time integration by parts using the identity $\int_0^T \langle \phi'(t),\shat{m}\rangle _{H^1(\Omega)^*\times H^1(\Omega)}\ dt=-\int_{\Omega_T}\phi\cdot \shat{m}_t\ dx\ dt + \int_{\Omega} \phi(T)\cdot \shat{m}(T)\ dx$, we obtain 
\begin{align}\label{GO-3}
	&-\int_0^T \langle \phi',\shat{m}\rangle \ dt +\int_{\Omega} \phi(T)\cdot \shat{m}(T)\ dx-\int_{\Omega_T} \Delta  \shat{m} \cdot \phi \ dx\ dt-\int_{\Omega_T} \big( \widetilde{m}\times \Delta \shat{m})\cdot \phi \ dx\ dt \nonumber\\
	&\hspace{1cm} -\int_{\Omega_T} \big(\shat{m} \times \Delta \widetilde{m}\big)\cdot \phi \ dx\ dt-\int_{\Omega_T} \big(\shat{m}\times \zetaup(\widetilde{U})\big)\cdot \phi \ dx\ dt +\int_{\Omega_T} \left(1+|\widetilde{m}|^2\right) \shat{m} \cdot \phi  \ dx\ dt  \nonumber\\
	&\hspace{1cm}+2\int_{\Omega_T} \big( \widetilde{m}\cdot \shat{m}\big)\ \widetilde{m}\cdot \phi \ dx\ dt=  \int_{\Omega_T}  \big(\shat{m} \times \Delta \shat{m} \big)\cdot \phi\ dx\ dt +\int_{\Omega_T} \big( \shat{m} \times \zetaup(\shat{U}) \big)\cdot \phi \ dx\ dt\nonumber\\
	&\hspace{1cm}+ \int_{\Omega_T} \big(\widetilde{m}\times \zetaup(\shat{U}) \big)\cdot \phi\ dx\ dt-\int_{\Omega_T} |\shat{m}|^2\ \shat{m} \cdot \phi \ dx\ dt -\int_{\Omega_T} |\shat{m}|^2\ \widetilde{m} \cdot \phi \ dx\ dt\nonumber\\
	&\hspace{1cm} -2 \int_{\Omega_T} \big(\shat{m}\cdot \widetilde{m}\big)\ \shat{m}\cdot \phi \ dx\ dt+ \int_{\Omega_T} \zetaup(\shat{U}) \cdot \phi \ dx\ dt.
\end{align}
After replacing $\vartheta$ with $\shat{m}$ in the weak formulation of the adjoint problem \eqref{AS} and conducting integration by parts for specific terms, we arrive at the following outcome:
\begin{align}\label{GO-4}
	&\int_0^T \langle \phi', \shat{m} \rangle \ dt +\int_{\Omega_T} \phi\cdot \Delta \shat{m}\ dx\ dt+ \int_{\Omega_T}  \big( \widetilde{m}\times \Delta \shat{m})\cdot \phi \ dx\ dt +\int_{\Omega_T} \big(\shat{m} \times \Delta \widetilde{m}\big)\cdot \phi \ dx\ dt\nonumber\\
	&\hspace{1cm} +\int_{\Omega_T} \big(\shat{m}\times \zetaup(\widetilde{U}) \big)\cdot \phi \ dx\ dt -\int_{\Omega_T} \left(1+|\widetilde{m}|^2\right) \shat{m} \cdot \phi  \ dx\ dt -2\int_{\Omega_T} \big( \widetilde{m}\cdot \shat{m}\big)\ \widetilde{m}\cdot \phi \ dx\ dt\nonumber\\
	&\hspace{1cm} =-\int_{\Omega_T} \big(\widetilde{m}-m_d\big)\cdot \shat{m}\ dx\ dt.
\end{align}
By combining equation  \eqref{GO-3} with \eqref{GO-4}, and comparing with the value of $\mathcal R$ in \eqref{GO-2}, we obtain 
\begin{align*}
	&\mathcal{R} =  \int_{\Omega_T}  \big(\shat{m} \times \Delta \shat{m} \big)\cdot \phi\ dx\ dt +\int_{\Omega_T} \big( \shat{m} \times \zetaup(\shat{U}) \big)\cdot \phi \ dx\ dt-\int_{\Omega_T} |\shat{m}|^2\ \shat{m} \cdot \phi \ dx\ dt\nonumber\\
	&\hspace{2cm} -\int_{\Omega_T} |\shat{m}|^2\ \widetilde{m} \cdot \phi \ dx\ dt -2 \int_{\Omega_T} \big(\shat{m}\cdot \widetilde{m}\big)\ \shat{m}\cdot \phi \ dx\ dt.
\end{align*}

Let's proceed to estimate the bounds for each term in $\mathcal{R}$. For the first two terms, we will apply H\"older's inequality and utilize the continuous embedding $H^1(\Omega)\hookrightarrow L^6(\Omega)$. Furthermore, using the Lipschitz continuity of the control-to-state operator \eqref{LCCTSO} and the uniform bounds for the control term $\|U\|^2_{L^2(0,T;\mathbb{R}^N)} \leq \|a\|^2_{L^2(0,T;\mathbb{R}^N)}+\|b\|^2_{L^2(0,T;\mathbb{R}^N)}:=C_{a,b}$, we find
\begin{flalign*}
	\int_{\Omega_T} \big( \shat{m} \times \Delta \shat{m}\big)&\cdot \phi\ dx\ dt \leq  \int_0^T \| \shat{m}(t)\|_{L^\infty(\Omega)} \ \|\Delta  \shat{m}(t)\|_{L^2(\Omega)} \ \|\phi(t)\|_{L^2(\Omega)}\ dt&\\
	&\leq C(\Omega,T)\  \|\phi\|_{L^2(0,T;L^2(\Omega))} \ \|\shat{m}\|^2_{L^\infty(0,T;H^2(\Omega))} \leq C(\Omega,T)\  \|\phi\|_{L^2(0,T;L^2(\Omega))} \ \|\shat{U}\|^2_{L^2(0,T;\mathbb{R}^N)},
\end{flalign*}
\begin{flalign*}
	\int_{\Omega_T} | \shat{m}|^2\ \shat{m}\cdot  \phi\ dx\ dt & \leq \int_0^T \|\shat{m}(t)\|^3_{L^6(\Omega)}\  \|\phi(t)\|_{L^2(\Omega)}\ dt \leq C(\Omega) \int_0^T \| \shat{m}(t)\|^3_{H^1(\Omega)} \ \|\phi(t)\|_{L^2(\Omega)}\ dt&\\
	& \leq C(\Omega,T)\ \|\phi\|_{L^2(0,T;L^2(\Omega))} \ \|\shat{m}\|^3_{L^{\infty}(0,T;H^1(\Omega))}\leq C\ \|\phi\|_{L^2(0,T;L^2(\Omega))} \ \|\shat{U}\|^3_{L^2(0,T;\mathbb{R}^N)}\\
	&\leq C(\Omega,T)\ \|\phi\|_{L^2(0,T;L^2(\Omega))}\ \left(\|U\|_{L^2(0,T;\mathbb{R}^N)}+\|\bar{U}\|_{L^2(0,T;\mathbb{R}^N)}\right)\ \|\shat{U}\|^2_{L^2(0,T;\mathbb{R}^N)}\\
	&\leq C(\Omega,T,C_{a,b})\ \|\phi\|_{L^2(0,T;L^2(\Omega))}\ \|\shat{U}\|^2_{L^2(0,T;\mathbb{R}^N)}.
\end{flalign*}
Continuing as above, we can derive the bounds for the third and forth terms as follows:
$$\int_{\Omega_T} | \shat{m}|^2\ \widetilde{m}\cdot  \phi\ dx\ dt + 2\int_{\Omega_T} \big(\shat{m}\cdot \widetilde{m}\big)\ \shat{m}\cdot \phi\ dx\ dt \leq C(\Omega,T)\ \|\widetilde{m}\|_{L^2(0,T;H^1(\Omega))}\ \|\phi\|_{L^2(0,T;L^2(\Omega))}\  \|\shat{U}\|^2_{L^2(0,T;\mathbb{R}^N)}.$$
The estimate for the term involving the control can also be derived by utilizing H\"older's inequality, the embedding $H^1(\Omega)\hookrightarrow L^p(\Omega)$ for $p\in [1,6]$, and \eqref{LCCTSO}. This leads us to the following derivation:
\begin{flalign*}
	&\int_{\Omega_T} \big(\shat{m} \times \zetaup(\shat{U})\big)\cdot \phi\ dx\ dt \leq  \int_0^T \|\shat{m}(t)\|_{L^4(\Omega)}\ \|\zetaup(\shat{U})(t)\|_{L^4(\Omega)} \ \|\phi(t)\|_{L^2(\Omega)}\ dt&\\
	&\hspace{0.2cm}\leq C(\Omega)\ \|\phi\|_{L^2(0,T;L^2(\Omega))}\ \|\shat{m}\|_{L^\infty(0,T;H^1(\Omega))}\ \|\zetaup(\shat{U})\|_{L^2(0,T;H^1(\Omega))} \leq C(\Omega,T,C_{a,b},m_0)\ \|\phi\|_{L^2(0,T;L^2(\Omega))} \|\shat{U}\|^2_{L^2(0,T;\mathbb{R}^N)}.
\end{flalign*}
Combining all the aforementioned estimates, we arrive at the subsequent bound for $\mathcal{R}$:
\begin{equation*}
	|\mathcal{R}| \leq C(\Omega,T,C_{a,b},m_0)\ \left\{ 1+ \|\widetilde{m}\|_{L^2(0,T;H^1(\Omega))} \right\} \ \|\phi\|_{L^2(0,T;L^2(\Omega))} \ \|\shat{U}\|^2_{L^2(0,T;\mathbb{R}^N)}.
\end{equation*}
By substituting this bound into estimate \eqref{GO-2}, we readily obtain our desired result. 
\end{proof}

\section{Uniqueness of Optimum}\label{S-UOO}

\begin{proof}[Proof of Theorem \ref{UOO}]
	Suppose there exists another optimal control $U\in \mathbb{U}_{ad}$ of the minimization problem OCP. Then through an application of the projection formula \eqref{PF}, we can derive
	\begin{flalign*}
	|U(t)-&\widetilde{U}(t)|\leq   \|\phi_{\widetilde{U}}(t) \|_{L^2(\Omega)} \ \|m_U(t)-m_{\widetilde{U}}(t)\|_{L^4(\Omega)}\ \|B\|_{L^4(\Omega)}\\
	&\hspace{0.5cm}  +\|\phi_U(t)-\phi_{\widetilde{U}}(t)\|_{L^2(\Omega)}\ \left( \|m_U(t)-\widetilde{m}(t)\|_{L^4(\Omega)}+ \|\widetilde{m}(t)\|_{L^4(\Omega)}+\|1\|_{L^4(\Omega)}\right) \ \|B\|_{L^4(\Omega)}.
	\end{flalign*}
	Then considering the $L^2(0,T)$ norm of $|U(t)-\widetilde{U}(t)|$ and applying the estimate $\|f\|_{L^4(\Omega)}\leq C_{4,n}\ \|f\|_{H^1(\Omega)}$, we obtain
	\begin{flalign}\label{UO-1}
		\int_0^T &|U(t)-\widetilde{U}(t)|^2\ dt \leq 2\ C^4_{4,n}\  \|\phi_{\widetilde{U}}\|^2_{L^\infty(0,T;L^2(\Omega))}\ \|m_{U}-\widetilde{m}\|^2_{L^2(0,T;H^1(\Omega))} \ \|B\|^2_{H^1(\Omega)} \nonumber\\
		&+ 2\ C^4_{4,n}\ \|\phi_U-\phi_{\widetilde{U}}\|^2_{L^2(0,T;L^2(\Omega))}\ \left( \|m_U-\widetilde{m}\|^2_{L^{\infty}(0,T;H^1(\Omega))}+\|\widetilde{m}\|^2_{L^\infty(0,T;H^1(\Omega))}+|\Omega|\right)\ \|B\|^2_{H^1(\Omega)}.
	\end{flalign}
Now, by employing a energy estimate argument on the equality $ \frac{d}{dt} \|\shat{\phi}(t)\|^2_{L^2(\Omega)}= 2\ \langle \shat{\phi}_t, \shat{\phi}\rangle_{H^1(\Omega)^*\times H^1(\Omega)}$ with $\shat{\phi}=\phi_U-\phi_{\widetilde{U}}$, we derive
	\begin{flalign*}
	\| \phi_U(t)-\phi_{\widetilde{U}}(t)\|^2_{L^2(\Omega)} &\leq 2\int_t^T \left\|\partial_{\tau}\phi_U(\tau)-\partial_{\tau}\phi_{\widetilde{U}}(\tau)\right\|_{H^1(\Omega)^*}\ \|\phi_U(\tau)-\phi_{\widetilde{U}}(\tau)\|_{H^1(\Omega)}\ d\tau + \| \phi_U(T)-\phi_{\widetilde{U}}(T)\|^2_{L^2(\Omega)}\\
	&\leq  2\ \left\|\partial_t\phi_U -\partial_t\phi_{\widetilde{U}}\right\|_{L^2(0,T;H^1(\Omega)^*)}\ \|\phi_U-\phi_{\widetilde{U}}\|_{L^2(0,T;H^1(\Omega))} +\|m_U-\widetilde{m}\|^2_{L^\infty(0,T;L^2(\Omega))}.
	\end{flalign*}
By integrating over $(0,T)$ and implementing the Lipschitz continuity estimates \eqref{LCCTSO} and \eqref{LCI}, we find
\begin{equation}\label{UO-2}
\int_0^T \|\phi_U(t)-\phi_{\widetilde{U}}(t)\|^2_{L^2(\Omega)}\ dt \leq \left( C_2 +2 C_3\right)\ T \ \|U-\widetilde{U}\|^2_{L^2(0,T;\mathbb{R}^N)}.	
\end{equation}
Similarly, proceeding for the state variable through the equality  $ \frac{d}{dt} \|\shat{m}(t)\|^2_{H^1(\Omega)}= 2\ \left(\frac{\partial \shat{m}}{\partial t}, \shat{m}-\Delta \shat{m}\right)_{L^2(\Omega)}$ with $\shat{m}=m_U-\widetilde{m}$, we can derive
\begin{align}\label{UO-3}
	&\int_0^T \left(\|m_U(t)-\widetilde{m}(t)\|^2_{L^2(\Omega)} + \|\nabla m_U(t)-\nabla \widetilde{m}(t)\|^2_{L^2(\Omega)}\right) \ dt\nonumber\\
	&\hspace{1cm}\leq 2 \int_0^T \left(\|m_U-\widetilde{m}\|_{L^2(0,T;L^2(\Omega))} +\|\Delta m_U-\Delta \widetilde{m}\|_{L^2(0,T;L^2(\Omega))}\right)\  \left\|\partial_t m_U-\partial_t \widetilde{m}\right\|_{L^2(0,T;L^2(\Omega))} \ dt \nonumber\\
	&\hspace{1cm} \leq 4C_2\  \|U-\widetilde{U}\|^2_{L^2(0,T;\mathbb{R}^N)}\ T.
\end{align}

Finally, substituting estimates \eqref{UO-2} and \eqref{UO-3} in \eqref{UO-1}, and using the the inequality $\|m_U-\widetilde{m}\|^2_{L^\infty(0,T;H^1(\Omega))} \leq C_2\ \|U-\widetilde{U}\|^2_{L^2(0,T;\mathbb{R}^N)} \leq 2C_2\ \left(\|U\|^2_{L^2(0,T;\mathbb{R}^N)}+\|\widetilde{U}\|^2_{L^2(0,T;\mathbb{R}^N)}\right)\leq 4C_2 C_{a,b}$, we obtain
\begin{flalign}\label{UO-4}
	\int_0^T& |U(t)-\widetilde{U}(t)|^2\ dt \leq  T \  2C^4_{4,n}\|B\|^2_{H^1(\Omega)}\bigg(4\ C_2\ \|\phi_{\widetilde{U}}\|^2_{L^\infty(0,T;L^2(\Omega))} \nonumber\\
	&\hspace{1cm} + \left( C_2 +2 C_3\right)\left( 4C_2\ C_{a,b} + \|\widetilde{m}\|_{L^\infty(0,T;H^1(\Omega))}+|\Omega|\right) \bigg) \  \|U-\widetilde{U}\|^2_{L^2(0,T;\mathbb{R}^N)}.
\end{flalign}

\noindent 	However, from assumption \eqref{ULOC} and estimate \eqref{UO-4}, it is clear that $\|U-\widetilde{U}\|_{L^2(0,T;\mathbb{R}^N)}=0$, that is, the optimal control is unique under the condition \eqref{ULOC}. Hence the proof.	
\end{proof}
\noindent {\bf Declarations} \medskip\\
\noindent {\bf Funding:} The National Board for Higher Mathematics has funded the second author's work with research grant (No.: 02011/13/2022/R$\&$D-II/10206).  \medskip\\
{\bf Conflict of interest:}  The authors declare that they have no conflict of interest.

\end{document}